\documentclass[reqno,12pt]{amsart}

\usepackage[foot]{amsaddr}

\usepackage{amsmath}
\usepackage{amssymb}
\usepackage{amsfonts}
\usepackage{mathtools}
\usepackage{enumerate}
\allowdisplaybreaks{}

\usepackage{palatino}
\usepackage[T1]{fontenc}
\usepackage[mathscr]{eucal}
\usepackage{dsfont}

\usepackage[numbers]{natbib}
\newcommand*{\doi}[1]{doi:\,\texttt{\href{http://dx.doi.org/#1}{#1}}}

\usepackage{fullpage}

\usepackage{latexsym}
\usepackage{paralist}
\usepackage{wasysym}
\usepackage{xspace}

\usepackage{xcolor}
\definecolor{webgreen}{rgb}{0,.5,0}
\definecolor{Maroon}{HTML}{800000}
\usepackage{hyperref}
\hypersetup{colorlinks, breaklinks, urlcolor=Maroon, linkcolor=Maroon, citecolor=webgreen} 

\usepackage[sort&compress,capitalize,nameinlink]{cleveref}
\crefname{app}{Appendix}{Appendices}

\crefrangeformat{equation}{\upshape(#3#1#4)\textendash(#5#2#6)}

\usepackage[textwidth=30mm]{todonotes}

\theoremstyle{plain}
\newtheorem{theorem}{Theorem}
\newtheorem{corollary}[theorem]{Corollary}
\newtheorem*{corollary*}{Corollary}
\newtheorem{lemma}[theorem]{Lemma}
\newtheorem{proposition}[theorem]{Proposition}

\theoremstyle{definition}
\newtheorem{definition}[theorem]{Definition}
\newtheorem{assumption}{Assumption}

\theoremstyle{remark}
\newtheorem{remark}[theorem]{Remark}
\newtheorem*{remark*}{Remark}
\newtheorem*{notation*}{Notational remark}

\numberwithin{theorem}{section}
\numberwithin{assumption}{section}
\numberwithin{equation}{section}

\newcommand{\procG}{(\cG^\ell,\cT^\ell)_{\ell \ge 0}}
\newcommand{\procH}{(\cH^i)_{i \ge 0}}
\newcommand{\si}{\sigma}
\newcommand{\ent}{{\mathrm{ent}}}
\newcommand{\tv}{{\textsc{tv}}}
\newcommand{\tent}{T_{\ent}}
\newcommand{\mix}{{\mathrm{mix}}}
\newcommand{\tmix}{T_{\mix}}
\newcommand{\out}{{\mathrm{out}}}
\newcommand{\muout}{\mu_{\out}}
\newcommand{\muin}{\mu_{\rm in}}
\newcommand{\var}{{\rm Var}}
\newcommand{\wt}{\widetilde}
\newcommand{\w}{\mathbf{w}}
\newcommand{\PP}{\mathbf{P}}

\newcommand{\ind}{\mathbf{1}}
\newcommand{\p}{\mathfrak{p}}
\newcommand{\tx}{\mathtt{Tx}}
\newcommand{\ee}{\mathrm{e}}
\newcommand\hh{\mathrm{H}}

\newcommand\hhigh{\overline{\hh}}
\newcommand\bfd{\mathbf{d}}

\newcommand{\cA}{\ensuremath{\mathcal A}}
\newcommand{\cB}{\ensuremath{\mathcal B}}
\newcommand{\cC}{\ensuremath{\mathcal C}}

\newcommand{\cE}{\ensuremath{\mathcal E}}
\newcommand{\cF}{\ensuremath{\mathcal F}}
\newcommand{\cG}{\ensuremath{\mathcal G}}
\newcommand{\cH}{\ensuremath{\mathcal H}}
\newcommand{\cI}{\ensuremath{\mathcal I}}
\newcommand{\cJ}{\ensuremath{\mathcal J}}

\newcommand{\cL}{\ensuremath{\mathcal L}}
\newcommand{\cM}{\ensuremath{\mathcal M}}

\newcommand{\cP}{\ensuremath{\mathcal P}}
\newcommand{\cQ}{\ensuremath{\mathcal Q}}
\newcommand{\cR}{\ensuremath{\mathcal R}}
\newcommand{\cS}{\ensuremath{\mathcal S}}
\newcommand{\cT}{\ensuremath{\mathcal T}}

\newcommand{\cW}{\ensuremath{\mathcal W}}

\newcommand{\cY}{\ensuremath{\mathcal Y}}
\newcommand{\cZ}{\ensuremath{\mathcal Z}}

\newcommand{\bbE}{{\ensuremath{\mathbb E}}}

\newcommand{\bbN}{{\ensuremath{\mathbb N}}}

\newcommand{\bbP}{{\ensuremath{\mathbb P}}}

\newcommand{\bbR}{{\ensuremath{\mathbb R}}}

\newcommand{\E}{\ensuremath{\mathbb{E}}}

\newcommand{\N}{\ensuremath{\mathbb{N}}}

\newcommand{\R}{\ensuremath{\mathbb{R}}}
\renewcommand{\P}{\ensuremath{\mathbb{P}}}

\newcommand{\pl}{\ensuremath{\left\langle}}
\newcommand{\pr}{\ensuremath{\right\rangle}}

\def\({\left(}
        \def\){\right)}
\def\[{\left[}
    \def\]{\right]}

\DeclarePairedDelimiter{\floor}{\lfloor}{\rfloor}
\DeclarePairedDelimiter{\ceil}{\lceil}{\rceil}
\DeclarePairedDelimiter{\abs}{\lvert}{\rvert}%

\usepackage{subcaption}

\usepackage{acro}

\DeclareAcronym{whp}{
    short = whp,
    long  = with high probability
}

\DeclareAcronym{scc}{
    short = \textsc{scc},
    long  = strongly connected component
}

\DeclareAcronym{dcm}{
    short = \textsc{dcm},
    long  = directed configuration model
}

\DeclareAcronym{bfs}{
    short = \textsc{bfs},
    long  = breath-first search
}

\DeclareAcronym{wrt}{
    short = wrt,
    long  = with respect to
}

\DeclareAcronym{iid}{
    short = iid,
    long  = independent and identically distributed
}

\DeclareAcronym{sfpe}{
    short = \textsc{sfpe},
    long  = stochastic fixed point equation
}

\DeclareAcronym{irg}{
    short = \textsc{irg},
    long  = inhomogeneous random graph
}

\DeclareAcronym{dpa}{
    short = \textsc{dpa},
    long  = directed preferential attachment model
}

\DeclareAcronym{rhs}{
    short = \textsc{rhs},
    long  = right-hand-side
}

\DeclareAcronym{lhs}{
    short = \textsc{lhs},
    long  = left-hand-side
}

\newcommand\cond{\:\middle|\:}
\newcommand{\eql}{\,{\buildrel d \over \sim}\,}

\newcommand{\Bin}{\mathop{\mathrm{Bin}}}

\newcommand\inprob{\overset{\bbP}{\longrightarrow}}

\begin{document}
\title[]{Rankings in directed configuration models \\ with heavy tailed in-degrees}

\author{Xing Shi Cai$^{\dagger}$}
\address{$^{\dagger}$Mathematics Department, Duke Kunshan University, China.}
\email{\href{mailto:xingshi.cai@dukekunshan.edu.cn}{xingshi.cai@dukekunshan.edu.cn}}

\author{Pietro Caputo$^{\dagger\dagger}$}
\address{$^{\dagger\dagger}$ Dipartimento di Matematica e Fisica, Universit\`a di Roma Tre, Largo S. Leonardo Murialdo 1, 00146 Roma, Italy.}
\email{\href{mailto:caputo@mat.uniroma3.it}{caputo@mat.uniroma3.it}}

\author{Guillem Perarnau$^{*}$}
\address{$^{*}$ IMTech, Universitat Polit\`ecnica de Catalunya, and Centre de Recerca Matem\`atica, Spain.}
\email{\href{mailto:guillem.perarnau@upc.edu}{guillem.perarnau@upc.edu}}

\author[M.~Quattropani]{Matteo Quattropani$^{**}$}
\address{$^{**}$ Dipartimento di Economia e Finanza, LUISS, Viale Romania 32, 00197 Roma, Italy.}
\email{\href{mailto:mquattropani@luiss.it}{mquattropani@luiss.it}}

\date{\today}

\begin{abstract}
    We consider the extremal values of the stationary distribution of sparse directed random
    graphs with given degree sequences 
    and their relation to the extremal values of the in-degree sequence.
    The graphs are generated by the directed configuration model. Under the assumption of  bounded $(2+\eta)$-moments on the in-degrees and of bounded out-degrees,
    we obtain tight comparisons between the maximum value of the stationary distribution and the
    maximum in-degree.
    Under the further assumption that the order statistics of the in-degrees have a power-law
    behavior, we show that the extremal values of the stationary distribution also have a power-law
    behavior with the same index.
    In the same setting, we prove that these results extend to the PageRank  scores of the random
    digraph, thus confirming a version of the so-called \emph{power-law hypothesis}.
    Along the way, we establish several facts about the model, including the mixing time cutoff and
    the characterization of the typical values of the stationary distribution, which were previously
    obtained under the assumption of bounded in-degrees.
\end{abstract}
\maketitle    

\tableofcontents{}

\section{Introduction}

The stationary distribution of the simple random walk on a directed graph (digraph) provides a natural measure
of the ranking of its nodes.
The potentially non-local nature of the stationary distribution in
directed networks makes the analysis of ranking a challenging task,
and it is of interest to relate the ranking statistics to much simpler local statistics such as the
in-degrees of the nodes.
In this paper we compare the maximum values of the stationary distribution to the maximum values of
the in-degrees in the setting of directed configuration models with sparse degree sequences.
We start with a  presentation of the model and the main results, and then return to a general
discussion of the problems involved, main motivations and relations to previous work.

\subsection{The model and the assumptions}

Let $[n]\coloneqq \{1,\dots,n\}$ be a set of $n$ nodes.
Let $\bfd_n=((d^{-}_1,d_1^{+}),\dots, (d^{-}_n,d^{+}_n)) \in \N^{2 \times n}$ be a bi-degree sequence with
\begin{equation}\label{eq:def-m}
m\coloneqq \sum_{v\in [n]} d^{+}_v = \sum_{v \in [n]}d^{-}_{v}.
\end{equation}

The \ac{dcm} is the random directed multigraph (digraph) on $[n]$,
$G=G_n$, generated as follows:
Assign $d^{-}_{v}$ \emph{heads} and $d^{+}_{v}$ \emph{tails} to vertex $v$,
match the $m$ heads and the $m$ tails with a uniformly random bijection,
and finally add a directed edge $(u,v)$ for each tail from $u$ that is matched to a head from $v$.

For a node $v\in[n]$, $d_v^-$ and $d_v^+$ are called the in-degree and the out-degree of $v$ respectively.
Let $\Delta^{\pm}=\Delta^{\pm}_n \coloneqq \max_{v\in[n]} d^{\pm}_v$ denote the maximum in/out-degree.
Unless otherwise specified, we will \emph{always} assume that the sequence of bi-degree sequences $(\bfd_n, n\in \bbN)$ satisfies the following condition:
\begin{assumption}\label{cond:main}
There exist constants $\eta,C>0$ and $K\geq 2$ such that for all $n\in\N$
    \begin{enumerate}[(i)]
        \item\label{it1}  minimum out-degree: $\min_{v\in [n]}d_v^+\ge 2$;

        \item\label{it2}  bounded  maximum out-degree: $\Delta^+_n\le K$

        \item\label{it3} bounded $(2+\eta)$-moment for in-degrees:
              \begin{equation}\label{eq:2+eta}
                  \sum_{v\in[n]}(d_v^-)^{2+\eta}\le C n.
              \end{equation}
    \end{enumerate}
\end{assumption}
Note that we do not assume any lower bound on the in-degrees. For the sake of brevity, we will use $\bfd_n$ to refer to the sequence $(\bfd_n, n\in \bbN)$.

\subsection{Definitions and notations} We recall some standard definitions and fix some notations.
We write
$\ind(E)=\ind_{E}$ for the indicator function of an event $E$.
A sequence of events $(E_{n})_{n \ge 0}$ occurs
\ac{whp}
if $\P(E_{n}) =\bbE[\ind(E_n)]\to 1$ as $n \to \infty$.
We write $X_{n} \inprob X$ whenever a sequence of random variables $X_{n}$
converges in probability to the random variable $X$, i.e., when $\P\(\abs{X-X_{n}} > \varepsilon\) \to 0$ for all $\varepsilon>0$.
We also use $o_{\P}(1)$ to denote an implicit sequence random variables which converges to $0$ in
probability~\cite{janson2011}. To avoid repetitions, it is often understood that our inequalities hold provided that $n$ is sufficiently large.

Under \cref{cond:main}, the probability that $G$ is simple (neither loops nor multiple edges) is bounded away from zero~\cite{blanchet2013,janson2009}. Furthermore, conditional on being simple, $G$ has the uniform distribution over simple digraphs on $[n]$ with degree sequence $\bfd_n$. Thus, all results in this paper that hold whp can be transferred to uniform simple digraphs.

Under an assumption weaker than \cref{cond:main},
it is known that \ac{whp}
the resulting digraph has a unique strongly connected component
which is globally attractive~\cite{cai2020}; this is false in general if vertices of out-degree at most one are allowed.
In particular, there exists a unique stationary distribution $\pi$ characterized by the equations
\begin{equation}\label{eq:stat}
    \pi(x)=\sum_{y\in [n]}\pi(y)P(y,x)\,,\qquad x\in[n],
\end{equation}
with the normalization $\sum_{x\in[n]}\pi(x)=1$. Here $P=P_G$ is the transition matrix of  the simple random walk on the multigraph $G$, defined as
\begin{equation}\label{eq:Pxy}
    P(y,x)
    =\frac{m(y,x)}{d^+_y},
\end{equation}
where we write $m(y,x)$ for the multiplicity of the directed edge $(y,x)$ in  $G$.

We write $(X_{t}, t \ge 0)$ for the simple random walk on $G$.
Thus, $P^{t}(v, \cdot)$ denotes the distribution of $X_{t}$ on $[n]$ conditioned on $X_{0}=v$.
If $t\ge 0$ is not an integer,
for simplicity we often write $t$ instead of $\floor{t}$ so that for example $P^t(v,\cdot)$
represents the distribution of the walk after $\floor{t}$ steps.
A standard measure of the distance to stationarity is $\|P^{t}(v, \cdot)- \pi\|_{\tv}$,
where the total variation distance between two probability measures $\mu,\nu$ on $[n]$ is defined by
\begin{equation}
    \label{eq:def-TV}
    \|\mu - \nu\|_{\tv}
    =  \max_{A\subset[n]}\,
    \abs*{\mu(A) - \nu(A)}.
\end{equation}
 The \emph{mixing time} of the random walk $(X_{t}, t \ge 0)$
is defined, for $\varepsilon\in(0,1)$, by
\begin{equation}
    \label{eq:def-mix}
    \tmix(\varepsilon)
    =
    \inf
    \{
        t \in \N \;:\;
        \max_{v \in [n]}
        \|P^{t}(v, \cdot), \pi\|_{\tv}
        < \varepsilon
    \}
    .
\end{equation}
The \emph{in-degree distribution} $\muin$ and \emph{out-degree distribution} $\muout$ on $[n]$ are defined by
\begin{equation}
    \label{eq:def-muout}
    \muin(v)
    =
    \frac{d_{v}^{-}}{m} \,,\quad
    \muout(v)
    =
    \frac{d_{v}^{+}}{m}
    ,
    \qquad
    \forall v \in [n]
    .
\end{equation}
Following~\cite{bordenave2019}, we define the
\emph{entropic time} by
\begin{equation}\label{eq:H}
    \tent
    =\frac{\log n}{\hh}
    ,
    \qquad
    \text{where}
    \qquad
    \hh=\sum_{v\in[n]}\muin(v)\log(d^+_v)
    .
\end{equation}

\subsection{Results}

\subsubsection{Mixing time}
Our first result concerns the mixing time and the cutoff phenomenon.
\begin{theorem}\label{cutoff}
    Let ${\bfd}_n$ be a bi-degree sequence satisfying \cref{cond:main}.
    Then, for all $\rho\neq 1$,
    \begin{equation}
        \max_{x\in [n]}\left|\| P^{\,\rho\,\tent}(x,\cdot)-\pi\|_{\tv} - \ind({\rho<1})\right| \inprob  0.
    \end{equation}
    In particular, for any $\varepsilon\in(0,1)$, \ac{whp} $\tmix(\varepsilon)=(1+o(1))\tent$.
\end{theorem}
\begin{remark}
\cref{cutoff} is an extension of the cutoff results from~\cite{bordenave2018} which were obtained in
the case of bounded degrees $\Delta^{\pm}=O(1)$.
It shows that $\tmix(\varepsilon)$ is, to leading order,
independent of $\varepsilon$, that is the Markov chain satisfies the \emph{cutoff phenomenon}.
Theorem~2 in~\cite{bordenave2018} considers also the cutoff window,
namely the behavior of the function $t \mapsto \| P^{t}(x,\cdot)-\pi\|_{\tv}$ on the finer
scale $t=\tent + aw_n$,
where $a\in \bbR$  and $w_n=O(\sqrt{\log n})$,
and showed that it approaches a universal Gaussian shape  for all $x \in [n]$.
One can check that the techniques we use here to prove  \cref{cutoff} are sufficient to obtain this refinement in our more general setting.
\end{remark}

\subsubsection{Typical values of the stationary distribution}

Our second result addresses the convergence
of the empirical distribution
\begin{equation}
    \label{eq:def-emperical}
    \psi_{n} \coloneqq \frac{1}{n}\sum_{v\in[n]}\delta_{n\pi(v)},
\end{equation}
where $\delta_a$ is the Dirac distribution centered at $a\in\bbR$.
Note that $\psi_n$ represents the law of $n \pi(v)$ when $v$ is picked uniformly at random in $[n]$.
We recall that the $1$-Wasserstein (or Kantorovich-Rubinstein) distance between two probability
measures $\mu$, $\nu$ on $\R$ is defined by
\begin{equation}
    \label{eq:def-W1}
    \cW_{1}(\mu,\nu )
    =\sup_{f}
    \abs*{\int_{\R}f \,\mathrm{d}\mu - \int_{\R}f \,\mathrm{d}\nu},
\end{equation}
where the supremum runs over of $f:\bbR\to\bbR$ such that $|f(x)-f(y)|\le |x-y|$ (see, for
example,~\cite[Chapter~6]{villani2009}).

Consider the sequence of deterministic
    probability measures on $\bbR_+ \coloneqq [0,\infty)$, $(\cL_n, n\in\N)$, where for $n\in\N$, $\cL_n$ is defined as the law of the random variable $X_n$ that satisfies $\E[X_n]=1$, and
    \begin{equation}\label{eq:distr-lim-mart}
    X_n\overset{d}{=}\frac{n}{m}\sum_{k=1}^{d_{\cI}^-}Z_k,
    \end{equation}
    where $\cI$ is a uniformly sampled vertex in $[n]$ and $\{Z_k\}_{k\geq 1}$
    are \ac{iid} random variables satisfying the \ac{sfpe}
    \begin{equation}\label{eq:distr_rec}
    Z_1\overset{d}{=}\frac{1}{d_{\cJ}^+}\sum_{k=1}^{d_\cJ^-}Z_k,
    \end{equation}
    where $\cJ$ is a random vertex in $[n]$ distributed as $\muout$. Existence and uniqueness of solutions to recursive distributional equations of the type \cref{eq:distr-lim-mart} is well known; see for example~\cite{aldous2005}.  

\begin{theorem}\label{th:bulk}
    Let ${\bfd}_n$ be a bi-degree sequence satisfying \cref{cond:main}.
	We have
    \begin{equation}\label{eq:w1-conv}
        \cW_1\(\psi_{n}, \cL_n\) \inprob 0.
    \end{equation}
\end{theorem}

\subsubsection{Extremal values of the stationary distribution}
Our main concern in this paper will be the behavior of the extremal values
of the stationary distribution. We start with
the maximum
$\pi_{\max} \coloneqq \max_{v\in[n]}\pi(v)$ and its relation to the maximum in-degree $\Delta^-$.

\begin{theorem}\label{th:main1}
    Let ${\bfd}_n$ be a bi-degree sequence satisfying \cref{cond:main}. Then,
    \begin{enumerate}[(i)]
                \item There exists an absolute constant $C>0$ such that, \ac{whp}
              \begin{equation}\label{eq:upb}
                  \pi_{\max}\le C\log(n)\frac{\Delta^-}{m}.
              \end{equation}
      \item If $\Delta^-\to\infty$, then $\forall\varepsilon >0$, \ac{whp}
              \begin{equation}\label{eq:lowb}
                  \pi_{\max}\ge (1-\varepsilon) \frac{\Delta^-}{m}.
              \end{equation}
\end{enumerate}
   \end{theorem}
\begin{remark}
The bounds  \cref{eq:upb} and \cref{eq:lowb} on $\pi_{\max}$ are essentially optimal under this
generality. Clearly,  if the graph is Eulerian, that is if $d_{v}^{+}=d_{v}^{-}$ for all $v \in
    [n]$, then  $\pi=\muin=\muout$ and thus $\pi_{\max}=\Delta^-/m$. On the other hand, 
Theorem 1.6 in~\cite{caputo2020a} shows the existence of bounded degree sequences for which
$\pi_{\max}= \log^{1-o(1)}(n)\frac{\Delta^-}{m}$.
We remark that the assumption $\Delta^-\to\infty$ in \cref{eq:lowb} is not really restrictive since by~\cite{caputo2020a} we already know that 
the bound \cref{eq:lowb} is always satisfied if $\Delta^-=O(1)$.
It is an interesting open question to determine whether the logarithmic term is necessary given that
$\Delta^-$ diverges sufficiently fast.
In \cref{sec:examples}, we refine \cref{eq:upb} for a wide class of sequences,
called \emph{extremal},
proving that in these cases
\begin{equation}
\pi_{\max}=  (1+o(1))\frac{\Delta^-}{m}.
\end{equation}
In particular, this proves in a strong sense the asymptotic tightness of \cref{eq:lowb}.
Moreover, we will see that for such extremal sequences the vertex with the maximum in-degree
coincides with the vertex with maximum stationary value.
\end{remark}

\subsubsection{Power-law behavior}

We turn to the analysis of the order statistics of the stationary distribution,
in the case where the in-degrees have an approximate power-law behavior.
We will consider the following notion of heavy tails.
Let $\cM_n$ denote the set of empirical distributions of size $n$ on $\bbR_+$,
that is the set of probability measures  of the form
\begin{equation}\label{emp}
    \mu_n=\frac{1}{n}\sum_{v\in[n]}\delta_{x_v},
\end{equation}
for some fixed vector $(x_1,\dots,x_n)\in\bbR_+^n$.
For any $\mu_n\in\cM_n$ and $t\ge 0$, let
\begin{equation}
    \mu_n(t,\infty)=\frac1n\,|\{i\in[n]:\, x_i>t\}|
\end{equation}
denote the right tail of $\mu_n$.
\begin{definition}\label{def:PL}
Given a constant $\kappa>0$, and a sequence of measures $\mu_n\in \cM_n$,
we say that $\mu_n$ has \emph{power-law behavior with index $\kappa$} if for all $\varepsilon>0$ and
for all $a\in(0,1/\kappa)$,
\begin{equation}\label{hp}
    n^{-a\kappa-\varepsilon} \le \mu_n(n^a,\infty) \le n^{-a\kappa+\varepsilon},\qquad \mu_n(n^{1/\kappa + \varepsilon},\infty)=0,
\end{equation}
 for all sufficiently large $n$. If the measures $\mu_n$ are random elements in $\cM_n$,
we say that $\mu_n$ has \emph{power-law behavior with index $\kappa$ with high probability},
if for all $\varepsilon>0$, \ac{whp} \cref{hp} holds for any $a\in(0,1/\kappa)$.
\end{definition}

{Since $\mu_n\in\cM_n$  has minimal mass $1/n$, the upper bounds in \cref{hp} are equivalent to the requirement that $\mu_n(n^a,\infty)\le n^{-a\kappa+\varepsilon}$ for all $\varepsilon>0$ and all $a>0$.}
Notice that if $(x_1,x_2,\dots)$ are \ac{iid} random variables with probability density $f(t)\propto
\min\{1, t^{-1-\kappa}\}$,
$t\in\bbR_+$, for some $\kappa>0$,
then the sequence of random empirical measures $(\mu_n)_{n \in \N}$ in \cref{emp} has power-law behavior with
index $\kappa$ with high probability (see, e.g.,~\cite{resnick2007}).

We apply this notion to our degree sequence.
Let $\phi$ be the empirical in-degree distribution; that is, for $k\ge 0$,
\begin{equation}\label{nt}
    \phi(k) =\phi_{n}(k)
      \coloneqq\frac1n\sum_{v\in[n]}\ind(d_v^-=k).
\end{equation}

Both $\phi,\psi$ define sequences of distributions $\phi_n,\psi_n\in\cM_n$,
but for simplicity we often drop the subscript $n$ from our notation.
The distribution $\phi$ has mean value $m/n$,
while $\psi$ has mean value $1$ for all $n$.

\begin{theorem}\label{th:main2}
    Let ${\bfd}_n$ be a bi-degree sequence satisfying \cref{cond:main}
    and assume that its empirical in-degree distribution $\phi$ has power-law behavior
    with index $\kappa>2$.
    Then, \ac{whp} the distribution $\psi$  has  power-law behavior
    with the same index $\kappa$,
    that is for all $\varepsilon>0$, \ac{whp}, $\psi(n^{1/\kappa + \varepsilon},\infty) =0$ and for all $a\in(0,1/\kappa)$,
    \begin{equation}\label{hppsi1}
       n^{-a\kappa-\varepsilon}\le \psi(n^a,\infty) \le n^{-a\kappa+\varepsilon}.
    \end{equation}
\end{theorem}

\begin{remark}\label{rem:general_ubtail}
If we only assume \cref{cond:main},
setting $\kappa_0=2+\eta$,
where $\eta>0$ is such that \cref{eq:2+eta} holds,
then we will see that for all $\varepsilon>0$,
$\psi$ satisfies the following upper bound \ac{whp}:
for all $a>0$,
\begin{equation}\label{psitail}
    \psi(n^a,\infty)
  \le n^{-a\kappa_0+\varepsilon},
\end{equation}
that is the right tail of $\psi$ is dominated by a heavy tail with index $\kappa_0$.
In some sense, this indicates that among all in-degree distributions with bounded $\kappa_0$-moment, the ones  with power-law behavior with index $\kappa_0$ ``maximize'' the upper tail of the  stationary distribution.
\cref{psitail} will be proved in~\cref{sec:PL} together with~\cref{th:main2},
as a consequence of a more general upper bound on $\psi(n^a,\infty)$.
\end{remark}

\subsubsection{PageRank surfer}

Next, we discuss the power-law behavior of PageRank.
Fix $\alpha \in(0,1)$ and let $\lambda$ be a probability distribution on $[n]$,
which we refer to as the \emph{teleporting probability}
and the \emph{teleporting distribution} respectively. The factor $1-\alpha$ is also referred in the literature as the \emph{damping factor}.
Consider the $(\alpha,\lambda)$-PageRank surfer,
that is the Markov chain with transition matrix
\begin{equation}\label{def-surfer}
P_{\alpha,\lambda}(x,y)=(1-\alpha)P(x,y)+\alpha \lambda(y).
\end{equation}
We call $(\alpha,\lambda)$-PageRank score the stationary distribution $\pi_{\alpha,\lambda}$ of this Markov chain, which is known to always be unique and to satisfy
\begin{equation}\label{eq:pi-PR}
\pi_{\alpha,\lambda}=\sum_{k=0}^\infty \alpha(1-\alpha)^k\lambda P^k,
\end{equation}
see, e.g.,~\cite{caputo2021}.
We will further assume that the teleporting distribution $\lambda$ is uniform up to multiplicative sub-polynomial factors,
that is, for all $\varepsilon>0$, 
\begin{equation}\label{hp:lambda}
 n^{-1-\varepsilon}\le \lambda(x)\le n^{-1+\varepsilon},
\end{equation}
for all sufficiently large $n$, uniformly in $x\in [n]$.

Let $\psi_{\alpha,\lambda}\in\cM_n$ be the empirical distribution in \cref{emp} corresponding to $x_v=n \pi_{\alpha,\lambda}(v)$.
\begin{theorem}\label{lbtailPR}
    Let ${\bfd}_n$ be a bi-degree sequence satisfying \cref{cond:main}
    and assume that its  empirical in-degree distribution $\phi$ has
    power-law behavior with index $\kappa>2$.
    For any constant $\alpha\in(0,1)$, and probability distribution $\lambda$ satisfying~\cref{hp:lambda},
    \ac{whp} $\psi_{\alpha,\lambda}$ has power-law behavior with the same index $\kappa$,
    that is, for all $\varepsilon>0$,  \ac{whp}, $\psi_{\alpha,\lambda}(n^{1/\kappa +\varepsilon},\infty)=0$ and for all $a\in(0,1/\kappa)$,
    \begin{equation}\label{hppsi}
       n^{-a\kappa-\varepsilon}\le \psi_{\alpha,\lambda}(n^a,\infty) \le n^{-a\kappa+\varepsilon}.
    \end{equation}
\end{theorem}

\begin{remark}
We will actually show a stronger result which holds for non-constant $\alpha$. Namely, that the upper bound $\psi_{\alpha,\lambda}(n^a,\infty)\le n^{-a\kappa + \varepsilon}$ holds \emph{uniformly} for arbitrary sequences $\alpha=\alpha_n\in[0,1]$. Indeed, as far as the upper bounds on the stationary distribution are concerned, it turns out that the presence of the parameter $\alpha$ can only make our analysis simpler; see \cref{rem:remalpha}. Moreover, the lower bound
$\psi_{\alpha,\lambda}(n^a,\infty)\ge  n^{-a\kappa - \varepsilon}$ holds under the only assumption that $\limsup_{n\to\infty}\alpha_n <1$; see \cref{sec:pagerank}.

\end{remark}

\begin{remark}\label{rem:maxPR}
Concerning the maximum PageRank score we will see that the following bounds hold \ac{whp} for any bi-degree sequence satisfying \cref{cond:main}, for all $\alpha\in[0,1]$ and any probability $\lambda$ on $[n]$:
    \begin{gather}\label{eq:pimaxPR}
                \alpha(1-\alpha)\lambda_{\min}\,\frac{\Delta^-}{\Delta^+}\le    \max_{x\in[n]}\pi_{\alpha,\lambda}(x)\le
                  C\log(n)\left(\lambda_{\max}+\frac{\Delta^-}{m}\right),
    \end{gather}
where $C$ is an absolute constant, $\lambda_{\min}=\min_{y\in[n]}\lambda(y)$, and $\lambda_{\max}=\max_{x\in[n]}\lambda(x)$. These bounds will be a simple consequence of our main results; see \cref{rem:pPR}.
       \end{remark}

\subsection{Motivation and related work}

Random walks on random undirected graphs have attracted a lot of attention in the last
decade~\cite{ben-hamou2017,berestycki2018,fountoulakis2008,lubetzky2010}.
Contrarily, much less is known in random digraphs.
The non-reversible nature of random walks in directed environments poses the challenge of developing
new techniques to study their properties.

One of the most natural models for random digraphs is the \acf{dcm},
which has been introduced in the literature as a directed analogue
of the configuration model~\cite{chen2013,cooper2004,newman2001}.
We refer the interested reader to~\cite{cai2020,cooper2004} for results
on its component structure and to~\cite{cai2020a,caputo2020a,hoorn2018} for the study of its distance profile.
Bordenave, the second author and Salez~\cite{bordenave2018}
recently initiated the study of random walks on the \ac{dcm}.
Provided that the minimum out-degree is at least $2$ and the maximum in-degree and out-degree are bounded,
they showed that the mixing time coincides with the \emph{entropic time},
defined in \cref{eq:H},
exhibits cutoff and has a Gaussian behavior inside the cutoff window.
Moreover, they showed that the stationary distribution of a uniformly random vertex $\cI$
converges (in the $1$-Wasserstein sense) to the solution
of the \acf{sfpe} displayed in~\cref{eq:distr-lim-mart}.
These results are extended to other models of  non-reversible sparse random Markov chains in~\cite{bordenave2019}.
Our results in \cref{cutoff} and \cref{th:bulk} show that the hypothesis on the degree sequence in~\cite{bordenave2018}
can be further relaxed to~\cref{cond:main}.

One of the questions left  open in~\cite{bordenave2018} is
the determination of the extremal behavior of the stationary values.
The second and fourth authors~\cite{caputo2020a} showed
that, in the bounded degree setting, the extremal (minimum and maximum) values of the stationary distribution
exhibit logarithmic fluctuations around the average value,
the exponents of the logarithm being essentially determined
by the minimum and maximum in- and out-degrees. In particular, regarding $\pi_{\max}$,
Theorem 1.6 in~\cite{caputo2020a} shows that if ${\bfd}_n$ is a bi-degree sequence satisfying
$2\le d_v^\pm=O(1)$, and such that there are linearly many vertices with degrees
$(\Delta^-,\delta^+)$, where $\delta^+=\min_v d_v^+$,  then there exists a constant $C>1$ such that \ac{whp}
\begin{equation}\label{eq:pi_max_bounded}
\frac{n\pi_{\max}}{\log^{1-\kappa_0} n} \in [C^{-1},C],
\end{equation}
where $\kappa_0=\frac{\log \delta^+}{\log \Delta^-}$.
\cref{th:main1} shows that if we allow the in-degrees to grow with the order of the digraph then
$\pi_{\max}$ will have much larger fluctuations.

In a similar spirit, the first and third authors~\cite{cai2020d}
proved that by dropping the condition on the minimum in-degree,
the minimum stationary value may become polynomially smaller than the average,
with the exponent given by the solution of an optimization problem
involving subcritical branching processes
and large deviation rate functions of the bi-degree distribution. Moreover, their results also give an implicit description of the lower tail of $\psi$, complementing \cref{rem:general_ubtail}.
In both works~\cite{cai2020d,caputo2020a},
controlling the minimum stationary values allows us
to estimate the cover time of a random walk in \ac{dcm}.

Stationary measures have also been studied for other random digraphs models.
Cooper and Frieze~\cite{cooper2012} determined the stationary distribution
of the directed Erd\"os-R\'enyi random graph in the strong connectivity regime,
motivated by their systematic study of the cover time in  random graph models.
Addario-Berry, Balle and the third author~\cite{addario-berry2020}
provided estimations for the extremal values of the stationary distribution
in random out-regular digraphs,
with applications to random deterministic finite automata.

While the analysis of random walks on \ac{dcm} only started recently, its stationary distribution has in
fact received a lot of attention under the framework of the PageRank algorithm.  PageRank was
introduced in~\cite{page1999} as a ranking measure for the webgraph and is a core element in
Google's search engine. The PageRank score is simply defined as the stationary distribution of the
PageRank surfer defined in~\cref{def-surfer}. We refer to~\cite{caputo2021} for
the mixing properties of the PageRank surfer on \ac{dcm}.  Compared to the in-degree ranking, the
PageRank score is less susceptible to assign high priority to spam pages~\cite{haveliwala2003}.
Nevertheless, empirical observations give a high average correlation between in-degrees and
PageRank~\cite{amento2000}. The so-called \emph{power-law hypothesis} ventures a more precise
description for scale-free networks: if the in-degree of a network is power-law distributed, then
its PageRank score also follows a power-law distribution with the same exponent. This has been
experimentally confirmed in several real-world
networks~\cite{donato2004,pandurangan2002,upstill2003}, in the  particular case of the webgraph, the
in-degree and PageRank are both approximately power-law distributions with index $\kappa\approx 1.1$. The
effect of the teleporting factor $\alpha$ has also been studied in~\cite{becchetti2006}, observing
that the top $10\%$ ranked elements follow a power-law distribution regardless of $\alpha$.

The abundance of empirical evidence has motivated the mathematical analysis of the power-law hypothesis.
A series of papers~\cite{litvak2007,volkovich2010,volkovich2007} proposed an idealized stochastic
model proving that the power-law distributions of the in-degree and of the PageRank score of a uniformly random
vertex $\cI$ only differ by a multiplicative factor.
Chen, Litvak
and Olvera-Cravioto~\cite{chen2014,chen2017} initiated the rigorous analysis of PageRank on \ac{dcm},
proving that the score of $\cI$ can be approximated by the PageRank score of the root of certain
infinite random tree, under the assumption that the in- and out-degrees are independently
distributed. In particular, if the in-degree distribution of $\cI$ is a power-law, so is its
PageRank. Olvera-Cravioto has recently extended these results to degree-degree correlated
distributions~\cite{olvera-cravioto2019}. In particular, the distribution of the PageRank of $\cI$
weakly converges to the attractive endogenous solution of an \ac{sfpe} that
generalizes~\cref{eq:distr-lim-mart}. The asymptotic properties of the solution imply that upper
tail of the PageRank of $\cI$ is asymptotically distributed as power-law with the right exponent.

The PageRank has also been studied in other directed random networks such as
\ac{irg}~\cite{lee2020,olvera-cravioto2019} and the \ac{dpa}~\cite{avrachenkov2006,banerjee2021}.
Remarkably, the power-law hypothesis is only partially true in \ac{dpa}: PageRank exhibits a
power-law distribution with different index than the index of the in-degree distribution. An
approach based on local weak convergence was given in~\cite{garavaglia2020}, yielding lower bounds
for the PageRank of a random vertex for any sequence of digraphs that has a local weak limit.

All aforementioned results describe the PageRank score of a vertex 
$\cI$ picked uniformly at random in $[n]$, or of a fixed given vertex, as in the case
of the oldest vertex in \ac{dpa} obtained
in~\cite{banerjee2021}. However, in most of the applications (such as web indexing), it is of
foremost importance to identify the top ranked elements~\cite{avrachenkov2011}.
To our best knowledge,
\cref{lbtailPR} is the first result that establishes the power-law hypothesis in the \emph{large
deviation} sense, providing the shape of the upper tail of the PageRank distribution in \ac{dcm}, not
only the upper tail for a typical vertex in the bulk of the digraph.

We refer to \cref{sec:future} for a discussion of open problems and future research directions.

\section{Preliminary results}

\subsection{Bounded moments}

We will use frequently the following deterministic property of degree sequences
with bounded $2+\eta$ moment of in-degrees.
\begin{lemma}\label{lem:s}
    Let ${\bfd}_n$ be a bi-degree sequence satisfying~\cref{eq:2+eta} with $\eta \in (0,1)$.
    Then
    \begin{equation}\label{eq:Delta:m}
        \Delta^{-} = O(n^{\frac{1}{2}-\frac{\eta}{6}}).
    \end{equation}
    Moreover, for any $S \subset [n]$ with $\abs{S}\le n^{1-\eta}$,
    \begin{equation}\label{eq:nuI}
        \sum_{v\in S} d_v^{-}
        =
        O\Big((\abs{S} n)^{\frac{1}{2}-\frac{\eta^2}{12}}\Big).
    \end{equation}
\end{lemma}

\begin{proof}
    By H\"older's inequality~\cite[Theorem~1.5.2]{durrett2010} with $p=2+\eta$,
    \begin{equation}\label{eq:partial}
        \sum_{v\in S} d_v^{-}
        \le
        \Big(\sum_{v\in S} (d_v^{-})^{p}\Big)^{\frac{1}{p}} \abs{S}^{1-\frac{1}{p}}
        = O\Big(\(n/|S|\)^{\frac{1}{p}}\Big)\abs{S} = O\(n^{\frac{1}{2}-\frac{\eta}{6}}\abs{S}^{\frac{1}{2}+\frac{\eta}{6}}\),
    \end{equation}
    where in the last inequality we used that $\frac{1}{p}\le \frac{1}{2}-\frac{\eta}{6}$ for $\eta\in (0,1)$.
    Taking $|S|=1$ we obtain \cref{eq:Delta:m}.
    Finally, using $\abs{S}\le n^{1-\eta}$ we obtain \cref{eq:nuI} from \cref{eq:partial}.
\end{proof}

\subsection{Local structure}\label{suse:local}

Let $\bfd_n=((d^{-}_1,d_1^{+}),\dots, (d^{-}_n,d^{+}_n))$ be a bi-degree sequence. For each $v\in [n]$, assign a set $E^-_v$ of $d^-_v$ labeled heads, and a set $E^+_v$ of $d^+_v$ labeled tails, and let
$E^\pm = \cup_{v\in [n]} E_v^{\pm}$. Throughout the paper, we will use $f$ to denote heads in $E^-$ and $e$ to denote tails in $E^+$. Denote by $v_e$ (or $v_f$) the vertex incident to $e$ (or $f$).
Every bijection $\omega:E^+\to E^-$ induces a multi-digraph $G=G_n(\omega)$ with vertex set $[n]$ and bi-degree sequence $\bfd_n$ by assigning a directed edge to every pair of vertices $(v_e,v_f)$ such that $\omega(e)=f$.
For simplicity, the multi-digraph $G$ will be often referred to as the digraph.

For $h\in\N$, a \emph{path} of length $h$ is a sequence of edges
\begin{equation}\label{eq:traj}
    \mathfrak{p} =
    \left\{
        (e_0,f_1),
        \left(e_1,f_2\right),
        \dots,
        \left(e_{h-1},f_{h}\right)
    \right\}
    ,
\end{equation}
where $e_{j-1}\in E^+$, $f_j\in E^-$, $\omega(e_{j-1})=f_{j}$,
and $v_{f_{j}}=v_{e_{j}}$ for all $j\in [h]$.
If $x=v_{e_0}$ and $y=v_{f_h}$,
we say that $\p$ is a path starting at $x$ and ending at $y$.
The \emph{weight} of the path $\p$ is the product of the inverse of the out-degrees of all
vertices along $\p$ except the last one, that is
\begin{equation}\label{eq:weight}
    \w( \mathfrak{p}) =
    \prod_{j=0}^{h-1}\frac1{d^+_{v_{e_j}}}   .
\end{equation}
Let $\cP(x,y,h,G)$ denote the set of all paths of length $h$ starting at $x$ and ending at $y$ in
the multi-digraph $G$.
A path is called \emph{simple} if it never visits the same vertex more than once.

For any $x\in [n]$ and $h\in \N$, the out-neighborhood of $x$ of depth $h$, $\cB_x^+(h)$, is the subgraph induced by all paths of length at most $h$ starting at $x$. Similarly, for any $y\in [n]$ the in-neighborhood of $y$ of depth $h$, $\cB_y^-(h)$, is the subgraph induced by all paths of length at most $h$ ending at $y$. We often identify $\cB_v^\pm(h)$ with its vertex set.
The boundary of $\cB_x^+(h)$, that is  the set of vertices $v$ such that the shortest path starting at $x$ and ending $v$ has length $h$, is denoted by $\partial \cB_x^+(h)$. Similarly $\partial \cB_y^-(h)$ represents the set of vertices $v$ such that the shortest path starting at $v$ and ending at $y$ has length $h$.

\subsubsection{Sequential generation.}\label{sec:SG}
For each $n\in\N$ the digraph $G=G_n(\omega)$ can be generated by matching tails and heads one at a time as follows.
Given a priority rule $\cR$,
\begin{enumerate}[(i)]
    \item choose an unmatched head $f\in E^-$ (if any) according to $\cR$;
    \item choose an unmatched tail $e\in E^+$ uniformly at random;
    \item set $\omega(e)=f$, and proceed.
\end{enumerate}
Observe that the roles of tails and heads can be reversed.

To explore an in-neighborhood $\cB^-_y(h)$,
we run the previous procedure with the priority rule given by the \ac{bfs} order.
In other words, at each time we choose a head closest from $y$ that has not been matched yet,
and pair it with a uniformly random unmatched tail.
We halt the procedure whenever all unmatched heads are at distance at least $h$ from $y$.
Similarly, reversing tails and heads, one can explore out-neighborhoods.

As in many sparse random models, one may expect that the neighborhoods are locally tree-like.
It will be important to see how much  they differ from a tree, motivating the following definition.
The \emph{tree-excess} of a multi-digraph $G=(V,E)$ is the number of additional edges it has with respect to a tree; that is,
\begin{equation}\label{def:tree_excess}
    \tx(G) \coloneqq 1+|E(G)|-|V(G)|.
\end{equation}
A step of the generating procedure is called a \emph{collision} if the vertex $v_f$ of the head $f$ such that $\omega(e)=f$ had been exposed during one of the previous pairings. Collisions indicate the appearance of additional edges in neighborhoods. In particular, the number of collisions in the \ac{bfs} generation of $\cB^+_x(h)$ is $\tx(\cB^+_x(h))$.

\subsubsection{Coupling with marked Galton-Watson trees.}\label{sec:GW}
For $x\in [n]$, let $\cT^-_y$ be the marked random tree with marks $\ell: V(\cT^-_y) \to [n]$, where $V(\cT^-_y)$ is the set of vertices of the tree,
having root $a_0$ with $\ell(a_0)=y$ and constructed iteratively with the following procedure,
starting with $a = a_{0}$:
\begin{enumerate}[(i)]
    \item Attach $d_{\ell(a)}^-$ children to $a$.
    \item Assign to each child $b$ of $a$ independently at random the mark $\ell(b)=z\in[n]$ with probability
        $d^+_z/m$.
    \item Choose the next $a$ to be the element in the tree which is one of the closest to the root
        among elements whose children have not been exposed.
        Terminate if no such element exists; otherwise go to step (i).
\end{enumerate}

Reversing the roles of in-degrees and out-degrees, we construct the random tree $\cT^+_x$.
Denote by $\cT^-_y(h)$ the subtree of $\cT^-_y$ containing the elements at distance at most $h$ from the root,
and $\partial \cT^-_y(h)$ the subtree containing those at distance exactly $h$;
similarly for $\cT^+_x(h)$ and $\partial \cT^+_x(h)$.
Notice that the random tree $\cT_{y}^{-}$ is obtained by gluing $d_y^-$ independent copies of
a \emph{Galton-Watson tree} with offspring distribution given by
\begin{equation}  \label{eq:GWoff}
    p^-(k)=\frac{1}{m}\sum_{v\in [n]} d^+_v \ind_{d^-_v=k} ,\qquad\forall k\ge 0.
\end{equation}

There is a natural coupling between the generating process of $\cB^\pm_v(h)$ and the construction of $\cT^\pm_v(h)$.
We now describe the coupling of $\cB^-_y(h)$ and $\cT^-_y(h)$.
The corresponding coupling of $\cB^+_x(h)$ and $\cT^+_x(h)$ can be obtained by reversing the role of heads and tails.

Clearly, step (ii) in the construction from \cref{sec:SG} can be modified by picking $e$ uniformly at random
among all (matched or unmatched) tails in $E^+$ and rejecting the proposal if the tail was already matched.
The tree can then be generated by iteration of the same sequence of steps with the difference
that at step (ii) we  never reject the proposal
and at step (iii) we add a new leaf to the current tree,
with mark $v$ if $e\in E_v^+$,
together with a new set of $d^-_v$ unmatched heads attached to it.

Call $\tau$ the first time that a uniform random choice among all tails gives $e\in E^+_v$ for some
mark $v$ already in the tree.
By construction,  the in-neighborhood and the tree coincide up to time $\tau$.
At the $k$-th iteration, the probability of picking a tail with a mark already used is at most $k\Delta^+/m$.
Therefore, by a union bound, for any $k\in\bbN$,
\begin{equation}\label{eq:coup1}
    \bbP(\tau\le k) \le \frac{k^2\Delta^+}{m}.
\end{equation}

\subsection{In-neighborhoods}\label{sec:in}
We start with an estimate of the  size of the in-neighborhoods and
then proceed with the analysis of the coupling with random trees described above.

For all $\varepsilon>0$ define
\begin{equation}\label{def-hslash}
    h_\varepsilon
    \coloneqq
    \frac{\varepsilon \log n}{20\log \Delta^+},
\end{equation}
and the event
\begin{equation}\label{eq:seps-}
    \cS^-_\varepsilon
    \coloneqq\left \{\forall y\in[n],\, \abs{\cB^-_y(h_\varepsilon)}\le n^{1/2+\varepsilon}\right\}.
\end{equation}
\begin{lemma}\label{lemmaGW}
    For all $\varepsilon>0$,  $\P\(\cS^-_\varepsilon \)=1-o(1)$.
\end{lemma}

\begin{proof}
    Fix $y\in[n]$, $\varepsilon>0$, and let $h=h_\varepsilon$ and $\cB^\pm_y=\cB^\pm_y(h_\varepsilon)$.
    It is enough to show that, for all $n$ large enough
    \begin{equation}\label{eq:sec-mom}
        \E\[|\cB^-_y|^2\]\le (d_y^-)^2n^{\varepsilon}.
    \end{equation}
    Indeed, \cref{eq:sec-mom} and Markov's inequality imply
    \begin{equation}
        \P\(|\cB^-_y| > n^{1/2+\varepsilon} \)\le \frac{\E\[|\cB^-_y|^2\]}{n^{1+2\varepsilon}}\le \frac{(d_y^-)^2}{n^{1+\varepsilon}}.
    \end{equation}
    Therefore, by taking a union bound over $y\in[n]$ and applying \cref{eq:2+eta}
    \begin{equation}
        \P(\cS^-_\varepsilon)\ge 1-\sum_{y\in[n]}\frac{(d_y^-)^2}{n^{1+\varepsilon}}=1-o(1).
    \end{equation}
To prove
\cref{eq:sec-mom}, note that
    \begin{equation}
        |\cB^-_y|=\sum_{v\in[n]}\ind_{y\in\cB^+_v},
    \end{equation}
    and therefore
    \begin{equation}\label{eq:sec-mom2}
        \E\[|\cB^-_y|^2 \]=\sum_{v\in[n]}\sum_{z\in[n]}\P\(y\in \cB^+_v,\:y\in\cB^+_z\).
    \end{equation}

    Observe that the out-neighborhood $\cB^+_v$ has at most $(\Delta^+)^h$ edges
    and at each step of the generation of  $\cB^+_v$ one has
    a probability of matching a head of $y$ bounded above
    by $ \frac{d_y^-}{m-(\Delta^+)^h}\le \frac{d_y^-}{n}$ for $n$ large enough.
    Thus, by a union bound over all steps of the generation of $\cB^+_v$, for all $v\in[n]$,
    \begin{equation}\label{eq:outbo}
        \P\(y\in\cB^+_v \)\le (\Delta^+)^h\frac{d_y^-}{n}.
    \end{equation}
Next, for all  $z\neq v$,
    \begin{equation}\label{eq:outbo2}
        \P\(y\in\cB^+_v,\:y\in\cB^+_z \) \le \P\(y,z\in\cB^+_v \)+\P\(y\in\cB^+_v,\:y\in\cB^+_z,\:z\notin \cB^+_v \).
    \end{equation}
    For the event $y,z\in\cB^+_v$ to occur, one must match a head of $y$ and a head of $z$ during the generation of $\cB^+_v$. Since there are at most $(\Delta^+)^h$ steps during the generation of $\cB^+_v$, and as in \cref{eq:outbo} one has a probability at most $\frac{d_y^-}{n}$ to match a head of $y$ at any given step (and $\frac{d_z^-}{n}$ for $z$),  a union bound gives
    \begin{equation}\label{eq:FRDD}
        \P\(y,z\in\cB^+_v \)
        \le
        (\Delta^+)^{2h}\frac{d_z^- d_y^-}{n^2}.
    \end{equation}

    Let us bound the second term in \cref{eq:outbo2}.
        We generate first $\cB^+_v$ and then $\cB^+_z$.
       Given the realization of $\cB^+_v$, and assuming $y\in \cB^+_v$ and $z\notin \cB^+_v$,
the event $y\in\cB^+_z$ can be obtained in two ways: either we match a fresh head of $y$ during the generation of $\cB^+_z$, or we match a fresh head of another vertex $w\neq y$ which was already discovered
during the generation of $\cB^+_v$. Note that the first scenario, reasoning as \cref{eq:outbo}, has probability at most $(\Delta^+)^{2h}\frac{(d_y^-)^2}{n^2}$. To handle the second scenario, define the event
$E_{v,z,w}$ that during the generation of $B_v^+$ we match a fresh head of $y$ and a fresh head of $w$, and then during the generation of $B_z^+$ we match a fresh head of $w$.
This event then satisfies
        \begin{equation}
            \label{eq:JDUN}
            \P(E_{v,z,w})
            \le
            \frac{d_y^-\left(\Delta ^+\right)^h}{n}
           \Big(d_w^- \frac{\left(\Delta ^+\right)^h}{n}\Big)^{2}  = \frac{d_{y}^{-}(d_{w}^{-})^{2}}{n^{3}}
 \, (\Delta^{+})^{3h}            .
        \end{equation}
        Summing over all possible choices of $w$
        and using
         \cref{cond:main}, we have
        \begin{equation}\label{eq:LLXJ}
            \begin{aligned}
&         \P\(y\in\cB^+_v,\:y\in\cB^+_z,\:z\notin \cB^+_v \)\le (\Delta^+)^{2h}\frac{(d_y^-)^2}{n^2}+
           \sum_{w\neq y}\P(E_{v,z,w})\\
            &\qquad \qquad \le (\Delta^+)^{2h}\frac{(d_y^-)^2}{n^2}+
            (\Delta^{+})^{3h} d_{y}^{-} \sum_{w\neq y}\frac{(d_{w}^{-})^{2}}{n^{3}}
            \le
            \left(\Delta ^+\right)^{4 h}\frac{(d_y^-)^2}{n^2}
            .
            \end{aligned}
        \end{equation}
        Combining \cref{eq:outbo,eq:outbo2,eq:FRDD,eq:LLXJ} we obtain,
        \begin{equation}\label{eq:LLXS}
        \P\(y\in \cB^+_v,\:y\in\cB^+_z\) \le (\Delta^+)^h\frac{d_y^-}{n} \ind(v=z) +  (\Delta^+)^{2h}\frac{d_z^- d_y^-}{n^2}+ \left(\Delta^+\right)^{4 h}\frac{(d_y^-)^2}{n^2}\;.
        \end{equation}
        Inserting the above estimates into \cref{eq:sec-mom2}, we obtain \cref{eq:sec-mom}.
\end{proof}
In what follows $\P$ denotes the probability under the coupling defined in \cref{sec:GW}.
We shall often take the parameter $\varepsilon>0$ smaller than some $\varepsilon_0=\varepsilon_0(\eta)$ where $\eta>0$ is the parameter appearing in \cref{cond:main}. To avoid repetitions, we will simply say that our statements hold for $\varepsilon>0$ sufficiently small. 
\begin{lemma}\label{le:coupling}
    Fix $\varepsilon>0$ sufficiently small and $h=h_\varepsilon$ as in \cref{def-hslash}.
    For any $y\in[n]$ ,
    \begin{equation}\label{eq:n2eps1}
    \P\(\cB^-_y(h)\neq  \cT^-_y(h)\)\le (d_y^-)^2 n^{-\frac13}.
    \end{equation}
    Moreover, for any $a>0$
    \begin{equation}\label{eq:n2eps}
        \P\(\abs{\cB^-_y(h)}> d_y^-n^a \)\le n^{\frac{\varepsilon}{4}-2a} + (d_y^-)^2 n^{-\frac13}.
    \end{equation}
    Moreover, for any $y\in [n]$ the coupling of $\cB^-_y(h)$ and $\cT^-_y(h)$ succeeds \ac{whp} and there exists $\eta'>0$ such that
        \begin{equation}\label{eq:n3eps}
        \P\(\abs{\cB^-_y(h)}> n^{\frac{1}{2}-\eta'}\) = o(1).
    \end{equation}
\end{lemma}
\begin{proof}
    Call $\cC_y=\{\cB^-_y(h)= \cT^-_y(h)\}$.
    Then, for all $a>0$
    \begin{equation}\label{eq:bound-coupling}
        \P\(\cC_y^c \)\le \P\(|\cT^-_y(h)|> d_y^- n^a\) +\P\(|\cT^-_y(h)|\le  d_y^- n^a,\: \cC_y^c \).
    \end{equation}
    Reasoning as in \cref{eq:coup1}
    \begin{equation}\label{eq:bound-CT}
        \P\(|\cT^-_y(h)|\le  d_y^- n^a,\: \cC_y^c \)\le \frac{\Delta^+ (d_y^-n^{a})^2}{m}.
    \end{equation}
    Notice that for each $j\ge 1$,
    \begin{equation}
        |\partial \cT^-_y(j)|=\sum_{i=1}^{d_y^-} Z^{(i)}_j,
    \end{equation}
    where the $Z^{(i)}_j$, $i=1,\dots,d_y^-$ are \ac{iid} random variables
    representing the size of the $j$-th generation of a Galton-Watson process with offspring distribution given by \cref{eq:GWoff}.
    The latter has expected value and variance
    \begin{equation}\label{SLDK}
        \nu= \frac{1}{m}\sum_{v\in[n]}d_v^-d_v^+ = O(1)\,,
        \quad
        \si^2= \frac1m\sum_{v\in [n]} d_v^+(d_v^-)^2 - \nu^2 = O(1),
    \end{equation}
    where the estimates follow from \cref{cond:main}.

    Setting $W_j=\nu^{-j}Z^{(1)}_j$, standard martingale computations
    (see~\cite[Chapter~I.4]{athreya1972}) show that
    \begin{equation}
        \E[Z^{(1)}_{j}] = \nu^{j},
        \qquad
        \var(Z^{(1)}_j)=\nu^{2j}\var(W_j) = \nu^{2j}\sum_{\ell=0}^{j-1}\nu^{-\ell-2}\si^2\le C\nu^{2j} ,
    \end{equation}
    for some constant $C>0$. It follows that
    \begin{equation}\label{eq:bound-CT01}
        \E[{\abs{\partial\cT_{y}^{-}(j)}}]
        =
        d_{y}^{-} \nu^{j},
        \qquad
        \var(|\partial\cT^-_y(j)|) = d_y^- \var(Z^{(1)}_j) \le C\nu^{2j}d_y^- .
    \end{equation}
     By Markov's inequality, uniformly in $j\le h$, for all $s>0$,
    \begin{equation}
        \P\(|\partial\cT^-_y(j)|> s \)
        \le
        \frac{1}{s^2}\,\E[\abs{\partial \cT_{y}^{-}(j)}^{2}]
        =
        \frac{1}{s^2}\,\((d_y^-\nu^{j})^2+ C\nu^{2j}d_y^-\)\le \frac{1}{s^2}\,2C(d_y^-)^2\nu^{2j}.
    \end{equation}
    By a union bound over $j\le h$ and setting $s=d_y^-n^a/h$, it follows that
    \begin{equation}\label{eq:boundT}
        \P\(|\cT_y^-(h)|> d_y^-n^a \)\le \frac{2C h^3\nu^{2h}}{n^{2a}}\le n^{\varepsilon/4-2a}.
    \end{equation}
    for  $n$ large enough.
    By \cref{eq:boundT,eq:bound-CT,eq:bound-coupling} and choosing $a=\frac14$ we conclude that
    \begin{equation}
        \P(\cC_y^c)\le (d_y^-)^2 n^{-1/3}.
    \end{equation}
    The proof of \cref{eq:n2eps} is an immediate consequence of  \cref{eq:n2eps1} and \cref{eq:boundT}. Finally, by setting $a=\frac{\eta}{7}$ and using \cref{lem:s}, we obtain that the left-hand-side of \cref{eq:bound-CT,eq:boundT} is $o(1)$ for $\varepsilon<\eta$.
\end{proof}

\begin{corollary}\label{coro:coupling}
    Fix $\varepsilon>0$ sufficiently small and $h=h_\varepsilon$. Let $\cI$ be a  uniformly random vertex in $[n]$. Then,
    \begin{equation}\P\(\cB^-_\cI(h)\neq \cT^-_\cI(h)\)\le  n^{-1/4 }.\end{equation}
\end{corollary}
\begin{proof}
    The estimate follows from \cref{le:coupling}, summing over $y\in[n]$ and using
    \begin{equation}
        \frac{1}{n}\sum_{y \in [n]} (d_y^-)^2 n^{-\frac13}\le n^{-\frac14},
    \end{equation}
    which holds for all $n$ large enough because of \cref{cond:main}.
\end{proof}

\subsection{Out-neighborhoods}\label{sec:out}

In this section we focus on tree-excesses of out-neighborhoods.
Recall the definition of $h_\varepsilon$ in \cref{def-hslash}.
For $h \in \N$, consider the event
\begin{equation}\label{eq:defG}
    \cG^+(h) \coloneqq \cap_{x\in [n]} \{\tx(\cB^+_x(h))\le 1\}.
\end{equation}
\begin{lemma}\label{lemmaG}
   For $\varepsilon>0$ sufficiently small,
    \begin{equation}
        \P\(\cG^+(2h_\varepsilon) \) = 1-o(1),
    \end{equation}
    where $h_\varepsilon$ is defined as in \cref{def-hslash}.
\end{lemma}
\begin{proof}
    Fix $x\in [n]$ and $h=h_\varepsilon$.
    To generate $\cB^+_x(2h)$,
    we match at most $K\coloneqq (\Delta^{+})^{2h}= n^{\varepsilon/10}$ tails.
    The probability that at any given step we choose a head incident to
    an already revealed vertex
    is at most
    \begin{equation}\label{eq:qq}
        q\coloneqq \frac{K\Delta^-}{m-K}.
    \end{equation}
    Therefore,
    \begin{equation}
        \P\(\tx(\cB^+_x(2h))\ge 2 \)\le \P(\Bin(K,q) \ge 2)\le
        (K q)^2,
    \end{equation}
   where $\Bin(N,p)$ denotes a binomial random variable with parameters $N,p$ and we use the simple bound $\P(\Bin(N,p) \ge \ell)\le \binom{N}{\ell}p^\ell\le (Np)^\ell$, valid for all $N,\ell\in\bbN$, $p\in[0,1]$.
    By \cref{lem:s}, if $\varepsilon<\frac{\eta}{2}$ it follows that $(K q)^2=o(n^{-1})$, and
    the conclusion follows by a union bound over $x\in [n]$.
\end{proof}

While the previous lemma cannot be improved substantially, the set of vertices with positive tree-excess is small, as the next result shows.
Define
\begin{equation}\label{def-hslash*}
    V_\varepsilon\coloneqq \{x\in [n]: \tx(\cB^+_x(h_\varepsilon))=0\}\,.
\end{equation}
\begin{lemma}\label{lem:V_*}
For any $\varepsilon>0$,
    \begin{equation}\label{YRJF}
        \P\(\abs{[n]\setminus V_\varepsilon} \le n^{\frac{2}{3}}\) \ge 1-n^{-\frac{1}{6}}.
    \end{equation}
\end{lemma}

\begin{proof}
    We may assume that $\varepsilon $ is sufficiently small.  As in the proof of \cref{lemmaG} we know that
    \begin{equation}
        \P\(\tx(\cB^+_x(h_\varepsilon))\ge 1 \)\le \P(\Bin(K,q)\ge 1)\le  K q,
    \end{equation}
    where now
    \begin{equation}
        K = (\Delta^+)^{h_{\varepsilon}}\,,\quad q = \frac{K\Delta^-}{m-K}\,.
    \end{equation}
    Therefore, for any $x\in[n]$, $\P\(x\notin V_\varepsilon\)\le Kq$. As a consequence, $\E\[\abs{[n]\setminus V_\varepsilon}\]\le nKq\le K^2\Delta^-$.
    By Markov's inequality and~\cref{eq:Delta:m}
    \begin{equation}\label{YRJF-}
        \P\(\abs{[n]\setminus V_{\varepsilon}} \ge n^{\frac{2}{3}}\)
        \le K^2\Delta^{-}n^{-\frac{2}{3}}
        =O\( n^{\frac{\varepsilon}{10}-\frac{\eta }{6}-\frac{1}{6}}\).
    \end{equation}
    Since $\varepsilon<\eta$, the lemma follows for large enough $n$.
\end{proof}

\section{Random walk}\label{sec:randomwalk}
In this section we introduce the random walk on $G$ and prove our main results concerning convergence to stationarity.
For a given realization of the digraph $G=G_n$ and a probability distribution $\mu$ on $[n]$,
we will denote by $\PP_\mu$ the \emph{quenched law} of the random walk on $G$
with initial distribution $\mu$, i.e.,
the law of the Markov chain on $[n]$ with transition matrix $P$ as in \cref{eq:Pxy}.
When $\mu$ is concentrated on a single vertex $x\in [n]$, i.e., $\mu=\delta_x$,
we write $\PP_x$.
We write $X_t$ for
the position of the random walk at any given time $t\ge 0$.
Notice that,  for any $A\subset [n]$, $\PP_\mu(X_t\in A)=\mu P^t(A)$ is itself a random variable.
When we want to emphasize its dependence on the realization of the
bijection $\omega$ which induces the digraph $G$, we write $\PP^\omega_\mu$.
Thanks to \cref{lemmaG} and \cref{cond:main}(i), we can immediately infer the following property of the quenched law.
\begin{lemma}\label{lemmaG2}
    Fix $\varepsilon>0$ sufficiently small
    and let $V_\varepsilon$ be defined as in \cref{def-hslash*}. For all $t \in\N$ such that $t\le h_\varepsilon$,
    \begin{equation}
        \P\(\max_{x\in[n]}\PP_x(X_t\notin V_{\varepsilon})\le 2^{-t} \)=1-o(1).
    \end{equation}
\end{lemma}
\begin{proof}
From \cref{lemmaG}, it is sufficient to prove that the event $\cG^+(2h_\varepsilon)$ implies $\PP_x(X_t\notin V_\varepsilon)\le 2^{-t}$ for all $x\in[n]$ and for all $t\le h_\varepsilon$. Under the event $\tx(\cB^+_x(2h_\varepsilon))\le 1$ there is at most one path of length $t\le h_\varepsilon$ from $x$ to any  point $y\notin V_\varepsilon$. Since each path of length $t$ has weight at most $2^{-t}$ the conclusion follows.
\end{proof}

A key object in our analysis is the so-called \emph{annealed law}, obtained by averaging the quenched law over the environment:
\begin{equation}
    \label{eq:annx}\P^{{\rm an}}_{x}=\E\[\PP_x\]=\frac1{m!}\sum_\omega\PP^\omega_x\,,\qquad \P^{{\rm an}}_{\mu}=\sum_{x\in [n]}\mu(x)\P^{{\rm an}}_{x},
\end{equation}
where $\omega$ ranges over all bijections from $E^+$ to $E^-$.

It will also be important to consider the average over the environment of the quenched law of $K$ independent walks.   As observed in~\cite{bordenave2018,bordenave2019} the corresponding annealed law is a very powerful tool in estimating high order moments of random variables such as $ \PP_\mu(X_t\in A)$ for $A\subset[n]$ and $t\ge 0$.   More precisely,
for any probability distribution $\mu$ on $[n]$, for any subset $A\subset[n]$ and for all $t,K\in\bbN$,
\begin{equation}\label{eq:def-annealing}
\E\[\big(\mu P^t(A)\big)^K\]
=
\E\[\big(\PP_\mu(X_t\in A)\big)^K \]
=
\P^{{\rm an},K}_{\mu}\(X_t^{(k)}\in A,\:\forall k\in [K]\),
\end{equation}
where the \emph{annealed law} $\P_\mu^{{\rm an},K}$ is the deterministic law of a non-Markovian process
\begin{equation}
\left\{X_s^{(k)}\,,\;s\in\{0,\dots,t\},\;k\in\{1,\dots, K\}\right\},
\end{equation}
which can be described as follows. Start with an empty matching. For each $k$, given the first $k-1$ walks $(X_s^{(\ell)})_{s\le t,\ell\le k-1}$, to generate the $k$-th walk $X^{(k)}$,
\begin{enumerate}[(i)]
    \item
    start the $k$-th walk at a random vertex $X_{0}^{(k)}\eql \mu$.
    \item for all $s\in\{0,\dots,t-1\}$: select one of the tails of  $X_{s}^{(k)}$ uniformly at random, and call it $e$:
    \begin{itemize}
        \item If $e$ was already matched by one of the previous walks, or by $X^{(k)}$ itself at a previous step, to some head $f$, then let $X_{s+1}^{(k)}=v_f$.
        \item If $e$ is still unmatched select a uniformly random head, $f$, among the unmatched ones, match it to $e$, and let $X_{s+1}^{(k)}=v_f$.
    \end{itemize}
\end{enumerate}We may view the walks as generating the environment (the digraph) as they activate new matchings (the edges). In particular, for any $K,t\ge 1$, the digraph $G$ may be sampled using the edges revealed by the $K$ walks up to time $t$ and then by completing with a uniform matching of the remaining heads and tails.

\subsection{Law of large numbers}
Using the annealed process as a computational tool,
we prove a quenched law of large numbers for the weight of the path determined by the  random walk trajectory,
thus extending results in~\cite{bordenave2018} previously obtained in the case of bounded in-degrees.
In particular, we show that for any $t$ of order $\log n$,
\ac{whp} the quenched law of the random walk is concentrated on trajectories
which have weight $\ee^{-\hh t(1+o(1))}$, where $\hh$ is the entropy in \cref{eq:H}.
The proof follows very closely the original argument in~\cite{bordenave2018},
while some minor technical difficulties due to the unbounded in-degree setting
are overcome using the $2+\eta$ moment condition in \cref{cond:main}.
\begin{proposition}\label{LLN}
    Let $t=\Theta(\log(n))$ and fix a sequence $\theta=\theta_n\in(0,1)$. Recall the definitions of path $\p$ in \cref{eq:traj}, of weight $\w(\p)$ in \cref{eq:weight} and $\cP(x,y,t,G)$. Call
    \begin{equation}Q_{x,t}(\theta):=\sum_{y \in [n]}\sum_{\p\in\cP(x,y,t,G)}\w(\p)\ind_{\w(\p)>\theta}.\end{equation}
    If for $\lambda\neq \hh$ we have
    \begin{equation}\label{eq:cond-theta}
    -\frac{\log(\theta)}{t}\longrightarrow\lambda,
    \end{equation}
    then
    \begin{equation}
    \max_{x\in [n]}\left|Q_{x,t}(\theta)-\ind(\lambda> \hh) \right|\overset{\P}{\longrightarrow}0.
    \end{equation}
\end{proposition}

\begin{proof}
    Fix sequences $t=\Theta\(\log(n)\)$ and $\theta=\theta_n$
    such that \cref{eq:cond-theta} holds.
    Let $\ell=3\log\log(n)$.
    Let us consider the averaged probability
    \begin{equation}
    \bar Q_{x,t}(\theta)\coloneqq\sum_{y\in[n]}P^\ell(x,y)Q_{y,t}(\theta).
    \end{equation}
We will show that
$\bar Q_{x,t}(\theta)$ is well approximated by
    \begin{equation}
    q_t(\theta)\coloneqq \P\(\prod_{k=1}^t \frac{1}{D_k^+}>\theta \)=\P\(\frac{1}{t}\sum_{k=1}^t\log(D_k^+)<-\frac{\log(\theta)}{t} \),
    \end{equation}
    where $(D_k^+)_{k\in \N}$ is an \ac{iid} sequence of random variables with law
    \begin{equation}
    \P(D^+_{k}=\ell )=\sum_{v\in[n]}\muin(v)\ind(d_v^+=\ell) \quad \text{for }\ell\ge 2.
    \end{equation}
    Using \cref{eq:H} and the law of large numbers for the bounded sequence $\big(\log(D^+_k)\big)_{k\in\N}$ we have
    \begin{equation}\label{eq:q_t}
    q_t(\theta)\longrightarrow\begin{cases}
    1 & \text{if } \lambda>\hh, \\
    0 & \text{if }\lambda <\hh.
    \end{cases}
    \end{equation}

    We first show  that \cref{LLN} is implied by the following convergence
    \begin{equation}\label{eq:claim}
    \max_{x\in V_\varepsilon}\left|\bar Q_{x,t}(\theta)-q_t(\theta) \right|\overset{\P}{\longrightarrow}0,
    \end{equation}
    where $V_{\varepsilon} $ is defined in \cref{def-hslash*} and $\varepsilon\in(0,\eta/2)$, where $\eta\in(0,1)$ is as in  \cref{cond:main}.
    Since the weight of a path of length $\ell$ is always in $[(\Delta^+)^{-\ell}, 2^{-\ell}]$ we may estimate
    \begin{equation}
        \begin{aligned}
            \max_{x\in [n]}Q_{x,t}(\theta)& \le \max_{x\in[n]}P^{\ell}(x,[n]\setminus V_\varepsilon)
            +\max_{x\in V_\varepsilon}\bar Q_{x,t-2\ell}(\theta 2^\ell) \\
            & \le 2^{-\ell}+o_{\P}(1)+\max_{x\in V_\varepsilon}\bar Q_{x,t}\(\theta 2^{\ell}(\Delta^+)^{-2\ell}\)         \\
            & \le q_t\(\theta 2^{\ell}(\Delta^+)^{-2\ell}\)+o_{\P}(1),
        \end{aligned}
    \end{equation}
    where the second line follows from \cref{lemmaG2} and the third line from \cref{eq:claim}.
    Similarly,
    \begin{equation}
        \begin{aligned}
            \min_{x\in [n]}Q_{x,t}(\theta) & \ge  \min_{x\in[n]}P^{\ell}(x, V_\varepsilon)\min_{x\in V_\varepsilon}\bar Q_{x,t-2\ell}\(\theta (\Delta^+)^{2\ell}\) \\
            & \ge  (1-2^{-\ell}-o_{\P}(1))\min_{x\in V_\varepsilon}\bar Q_{x,t-2\ell}\(\theta (\Delta^+)^{2\ell}\)        \\
            & \ge                                \min_{x\in V_\varepsilon}\bar Q_{x,t}\(\theta (\Delta^+)^{2\ell}\)+o_{\P}(1)                           \\
            & \ge           q_t\(\theta (\Delta^+)^{2\ell}\)+o_{\P}(1).
        \end{aligned}
    \end{equation}
    Moreover, notice that if for some $\theta'$ it holds $\log(\theta')=\log(\theta)+O(\log\log(n))$ then
    \begin{equation}
    |q_t(\theta)-q_t(\theta')|\to 0.
    \end{equation}
    Thus, \cref{eq:q_t,eq:claim} imply  \cref{LLN}.

To prove  \cref{eq:claim}, we show that
for all $\delta>0$
    \begin{equation}\label{eq:claim2}
    \P\(\ind_{x\in V_\varepsilon}\bar Q_{x,t}(\theta)\ge q_t(\theta)+\delta \)=o(n^{-1}).
    \end{equation}
    This, together with a union bound over $x\in V_\varepsilon$, establishes one half of \cref{eq:claim}; the other half can be obtained in the same fashion replacing
    $\bar Q_{x,t}(\theta)$ by $1-\bar Q_{x,t}(\theta)$ and $q_t(\theta)$ by $1-q_t(\theta)$, which amounts to inverting the inequality signs in the definition of $Q_{x,t}(\theta)$ and $q_t(\theta)$.

We now prove \cref{eq:claim2}.
    By Markov's inequality, for all $K\ge 1$,
    \begin{equation}
    \P\(\ind_{x\in V_\varepsilon}\bar Q_{x,t}(\theta)\ge q_t(\theta)+\delta \)\le \frac{\E\[\(\ind_{x\in V_\varepsilon}\bar Q_{x,t}(\theta)\)^K \]}{(q_t(\theta)+\delta)^K}.
    \end{equation}
    Hence, it is enough to show that if $K=\floor{\log^2(n)}$, for all $\delta>0$ and $n$ large enough one has
    \begin{equation}\label{eq:claim3}
\E\[\(\ind_{x\in V_\varepsilon}\bar Q_{x,t}(\theta)\)^K \]\le \(q_t(\theta)+\tfrac{\delta}{2}\)^K.
    \end{equation}
    Notice that
    \begin{equation}
    \E\[\(\ind_{x\in V_\varepsilon}\bar Q_{x,t}(\theta)\)^K \]\le \P^{{\rm an},K}_x(B_K),
    \end{equation}
    where $\P^{{\rm an},K}_x$ is the law of $K$ annealed walks of length $\ell+t$ all started at $x$, see \cref{eq:def-annealing}, and
    for all $j\le K$, $B_j$ is the event that
    \begin{enumerate}[(i)]
        \item the union of the first $j$ trajectories up to time $\ell$, that is $(X_s^{(1)},\dots,X_s^{(j)})_{s\le \ell}$, forms a directed tree.
        \item for each $i\le j$, the last $t$ steps of the $i$-th walk, that is $(X_s^{(i)})_{s \in[\ell+1,\ell+t]}$, define a path $\p$ of weight $\w(\p)>\theta$.
    \end{enumerate}
    Since
    \begin{equation}
    \P^{{\rm an},K}_x(B_K)=\P^{{\rm an},K}_x(B_1)\prod_{j=2}^{K}\P^{{\rm an},K}_x\left(B_j\cond{}B_{j-1}\right),
    \end{equation}
    it is enough to show that, uniformly in $j\le K$,
    \begin{equation}\label{eq:claim4}
    \P^{{\rm an},K}_x\left(B_j\cond{}B_{j-1}\right)\le q_t(\theta)+\frac\delta{2}.
    \end{equation}
    In order to check that \cref{eq:claim4} holds, we note that, given $B_{j-1}$:
    \begin{enumerate}[(i)]
        \item either the $j$-th walk attains length $\ell$ before reaching an unmatched tail: thanks to the tree structure, there are at most $K-1$ possible paths of length $\ell$ to follow starting at $x$, and each has weight at most $2^{-\ell}$. Thus, the conditional probability of this scenario is less than $K2^{-\ell} = o(1)$.
        \item or the $j$-th walk has reached an unmatched tail by time $\ell$: then, the remainder of the path after the first unmatched tail can be coupled with an \ac{iid} sample from the in-degree distribution on $[n]$ at a total-variation cost less than $\frac{\Delta^- K(t+\ell)^2}{m}$, and the latter is deterministically $o(1)$ thanks to \cref{lem:s}. Indeed, there are at most $t+\ell$ heads matched in the generation of the $j$-th walk and the probability that the coupling fails at a given step is bounded uniformly by $\frac{\Delta^- K(t+\ell)}{m}$, as all heads can be chosen. Thus, the conditional probability that the walk meets the requirement in that case is at most $q_t(\theta)+o(1)$.
    \end{enumerate}
\end{proof}

\subsection{A weighted out-neighborhood construction}\label{suse:out-phase}

Following an idea introduced in~\cite{bordenave2018,bordenave2019}, we now consider a further construction of the out-neighborhood of a given vertex that reveals only the directed paths
which have a sufficiently large probability to be followed by the random walk.

Fix $\gamma\in (0,1)$. All parameters defined below depend implicitly on $\gamma$. Let
\begin{equation}\label{OMLZ}
    \hhigh
    \coloneqq
    (1+\gamma)\hh
    ,
\end{equation}
where $\hh$ is the entropy in \cref{eq:H}. For any integer $s\ge 1$ and every constant $\gamma>0$, let $\cG_x(s)$ be the weighted directed graph spanned by the set of paths of length at most $s$, starting from $x$, and having weight $\w(\p) \ge \ee^{-\hhigh s}$.
As in~\cite[Section~4.1]{bordenave2019},
we construct a sequence
$\procG$, where $\cG^\ell$ is a subgraph of $\cG_{x}(s)$ with $\ell$ edges, constructing $\cG_{x}(s)$ edge by edge, and such that $\cT^\ell$ is a spanning tree of $\cG^\ell$ for each $\ell$. We call $\kappa_{x} = \kappa_{x}(s)$ the random number of edges needed to construct the whole digraph $\cG^{\kappa_{x}}= \cG_{x}(s)$. We now explain the detailed construction of the sequence $\procG$. As initialization, let $\cE_0$ be the set of tails of $x$, $\cF_0=\emptyset$ and let $\cG^0=\cT^0$  be the digraph containing only $x$ and no edges. Then, for all $\ell\ge 1$:
\begin{enumerate}[(i)]
    \item Let ${\cE}_{\ell-1}$ be the set of \emph{unmatched} tails which are incident to a node in
          $\cG^{\ell-1}$. For a tail $e \in {\cE}_{\ell-1}$, define its \emph{cumulative
              weight} by
          \begin{equation}\label{PLGD}
              \widehat{\w}(e) \coloneqq \frac{\w(\p)}{d_{v_e}^{+}} ,
          \end{equation}
where $\p$ is the unique path from $x$ to
          $v_e$ in $\cT^{\ell-1}$.
          In words, $\widehat{\w}(e)$ is the probability for the random walk to follow $\p$ and
          then the edge containing $e$.
    \item  Pick $e_{\ell} \in {\cE}_{\ell-1}$ such that
          \begin{enumerate}
              \item $v_{e_\ell}$ is at distance at most $s-1$ from $x$,
              \item $\widehat{\w}(e_\ell) \ge \w_{\min}\coloneqq\ee^{-\hhigh s}$,
              \item $e_\ell$ has maximum cumulative weight among all tails in ${\cE}_{\ell-1}$ which satisfy
                    (a) and (b), and use some arbitrary rule to break ties if needed.
          \end{enumerate}
    \item If no such $e_{\ell}$ exists, then terminate and let $\kappa_{x}
              = \ell-1$.
    \item Otherwise, pair $e_\ell$ with a head $f_{\ell}$ chosen uniformly at random in $E^-\setminus \cF_{\ell-1}$ and set $\cF_\ell=\cF_{\ell-1}\cup\{f_\ell\}$.
          Let $\cG^{\ell}$ be the resulting partial pairing. If $v_{f_{\ell}} \in
              V(\cG^{\ell-1})$, the set of vertices of $\cG^{\ell-1}$, then let $\cT^{\ell}=\cT^{\ell-1}$; otherwise let
          $\cT^{\ell}=\cT^{\ell-1} \cup (e_{\ell}, f_{\ell})$. Return to step (i).
\end{enumerate}

When the process terminates, the construction yields $\cG_x(s)\coloneqq \cG^{\kappa_{x}}$ and $\cT_x(s) \coloneqq\cT^{\kappa_{x}}$.
By~\cite[Lemma 11]{bordenave2018}, one has the deterministic bound
\begin{equation}\label{eq:weights}
   \w_{\min}\le  \widehat{\w}(e_{\ell}) \le \frac{2}{2+\ell} ,
\end{equation}
for all $\ell\le \kappa_{x}$.
From this, it follows that,
\begin{equation}\label{eq:kappax:s}
    \kappa_{x}(s) \le \frac{2}{\w_{\min}}=  2\ee^{\hhigh s }.
\end{equation}

The main motivation for the above construction is the fact that with high probability the random walk trajectories up to time $s\le (1-o(1))\tent$ are concentrated on the tree $\cT_x(s)$ provided the starting point $x$ is locally tree-like, as we now explain. Note that the probability that a walk trajectory $(X_0,\dots,X_s)$ with $X_0=x$ stays on the tree $\cT_x(s)$ may be written as
\begin{equation}\label{eq:bennett1}
p(x,s) =\sum_{y\in  [n]}\sum_{\p\in\cP(x,y,s,\cT_x(s))}\w(\p),
\end{equation}
where $\cP(x,y,s,\cT_x(s))$ denotes the set of all paths of length $s$ from $x$ to $y$ in $\cT_x(s)$.
\begin{lemma}\label{lem:bennett}
Recall the definition of $V_\varepsilon$ in \cref{def-hslash*}.
For all $\varepsilon >0$, $\gamma>0$, and for any $s\le(1-\gamma)T_\ent$,
\begin{equation}\label{eq:bennett01}
\min_{x\in V_\varepsilon} \, p(x,s) \inprob 1.
\end{equation}
\end{lemma}
\begin{proof}
Let $\cP_{x,y}^*=\cP(x,y,s,G)\setminus \cP(x,y,s,\cT_x(s))$
denote the set of paths starting at $x$ and ending at $y$ of length $s$
which are not on the tree $\cT_x(s)$.
We need to show that
\begin{equation}\label{eq:bennett2}
\max_{x\in V_\varepsilon} \, \sum_{y}\sum_{\p\in\cP_{x,y}^*}\w(\p)\inprob 0.
\end{equation}
If a path $\p$ is in $\in \cP_{x,y}^{*}$, then at least one of the following conditions is satisfied:
\begin{enumerate}[(i)]
    \item the weight of the path satisfies $\w(\p) < \ee^{-\hhigh s}$;
    \item $\p$ contains an edge in $\cG_x(s)\setminus\cT_x(s)$.
\end{enumerate}
We call $\cP_{x,y}^{1,*}$ the set of paths in case (i),
and $\cP_{x,y}^{2,*}$ the set of paths in case (ii) but not in case (i).
From the law of large numbers  in \cref{LLN} we know that
\begin{equation}\label{eq:bennett3}
\max_{x\in [n]} \, \sum_{y\in [n]}\sum_{\p\in\cP_{x,y}^{1,*}}\w(\p)\inprob 0.
\end{equation}
It remains to show that
\begin{equation}\label{eq:bennett4}
\max_{x\in V_\varepsilon} \, \sum_{y}\sum_{\p\in \cP_{x,y}^{2,*}}\w(\p)\inprob 0.
\end{equation}

Define a process $(M_{\ell})_{\ell\ge 0}$ by $M_0=0$ and for $\ell\ge 1$
\begin{equation}\label{MTRW}
    M_{\ell}=M_{\ell-1}+ \widehat{\w}(e_\ell) \ind(\ell\le \kappa_{x})\ind(v_{f_{\ell}}\in
    V(\cG^{\ell-1}))
    .
\end{equation}
Notice that
\begin{equation}
M_{\kappa_{x}}=\sum_{y\in [n]}\sum_{\p\in\cP_{x,y}^{2,*}}\w(\p).
\end{equation}
We will show that $\P(M_{\kappa_{x}}>\delta)=o(n^{-1})$, uniformly in $x\in V_\varepsilon$, and for all $\delta>0$. By a union bound over $x\in V_\varepsilon$, this implies the desired claim.

Since $x\in V_\varepsilon$,
it follows that $\cG_x(h_\varepsilon)$ is a tree, where $h_\varepsilon$ is defined as in \cref{def-hslash}.
Therefore, setting $\ell_\varepsilon=2^{h_\varepsilon}$,
one must have  $M_{\ell}=0$ for all $\ell<\ell_\varepsilon$.
Thus, the desired conclusion follows once we prove
\begin{equation}
\P\(M_{\kappa_{x}}-M_{\ell_\varepsilon}>\delta\)=o(n^{-1}).
\end{equation}
To prove it we use a martingale version of Bennett's inequality obtained by Freedman~\cite{freedman1975}.

Since $\cG_x(h_\varepsilon)$ is a tree of size at least $\ell_{\varepsilon}$ and
$\widehat{\w}(e_\ell)$ is decreasing in $\ell$,
we have that $M_{\ell} = 0$ for all $\ell \le \ell_{\varepsilon}$ and, by \cref{eq:weights}, that
\begin{equation}\label{eq:bennett5}
    0\le M_{\ell}-M_{\ell-1}\le 2^{-h_\varepsilon+1}=o(1),
\end{equation}
for all $\ell\ge\ell_\varepsilon+1$.
Let $\cF_{\ell}$ be the filtration associated to $(\cG^\ell,\cT^\ell)$.
Call
\begin{equation}
g_\ell= \sum_{v\in V(\cG^{\ell})} d^-_v,
\end{equation}
and note that by
\cref{eq:kappax:s}, the vertices in $V(\cG^{\ell})$ are at most $\ell \le \kappa_x  \le 2 \ee^{\hhigh (1-\gamma)T_\ent}= 2n^{1-\gamma ^2}$.
Thus, by
Lemma~\ref{lem:s} we have $g_\ell= o(\sqrt{\ell n})$.
Using~\cref{eq:weights},
\begin{equation}
    \begin{aligned}
        \E[M_{\ell}-M_{\ell-1}\mid \cF_{\ell-1}]
         & \le \ind_{\ell\le \kappa_{x}} \frac{\widehat{\w}(e_{\ell}) g_{\ell-1}}{m-(\ell-1)}
        = o\(\frac{1}{\sqrt{n\ell}}\);                                                                  \\
        \E[(M_{\ell}-M_{\ell-1})^2\mid \cF_{\ell-1}]
         & \le \ind_{\ell\le \kappa_{x}} \frac{\(\widehat{\w}(e_{\ell})\)^2 g_{\ell-1}}{m-(\ell-1)}
        =
        o\(
        \frac{1}{\sqrt{n\ell^3}} \)
        .
    \end{aligned}
\end{equation}
Adding over all steps until $\kappa_{x}$, we get
\begin{equation}
    \begin{aligned}
        a &
        \coloneqq
        \sum_{\ell=1}^{\kappa_{x}}
        \E[M_{\ell}-M_{\ell-1}\mid \cF_{\ell-1}]
        \le \:o\(n^{-1/2}\)
        \sum_{\ell=1}^{\kappa_{x}}
        \frac{1}{\sqrt{\ell}}
        =o\(n^{-\gamma ^2/2}\)
        ,
        \\
        b & \coloneqq \sum_{\ell=1}^{\kappa_{x}}
        \E[(M_{\ell}-M_{\ell-1})^2\mid \cF_{\ell-1}]
        \le o\(n^{-1/2}\)
        \sum_{\ell=1}^{\kappa_{x}}
                \frac{1}{\ell^{3/2}}
        = o\(n^{-1/2}\)
        ,
    \end{aligned}
\end{equation}
where we used $\sum^k_{\ell=1} \ell^{-1/2}= O(\sqrt{k})$, $\sum_{\ell\ge 1} \ell^{-3/2} =O(1) $, and $\kappa_x \le 2n^{1-\gamma ^2}$.
For $\ell\ge \ell_\varepsilon$, define
\begin{equation}\label{SQGQ}
    Z_{\ell+1} = \frac{5}{\delta}(M_{\ell+1}-M_{\ell}-\E[M_{\ell+1}-M_{\ell}\mid\cF_{\ell}])
    ,
\end{equation}
which satisfies $|Z_{\ell+1}|\le 1$ by \cref{eq:bennett5}.
Let $\phi_{u} = \sum_{i=\ell_\varepsilon }^u Z_{i+1}$. Then $ (\phi_{u})_{u \ge \ell_\varepsilon}$ is a martingale and
$M_{\kappa_{x}} = a+\frac{\delta}{5}\phi_{\kappa_x}$. When $n$ is large enough, we have
$a\le \delta/5$. So
\begin{equation}
    \P\(M_{\kappa_{x}}\ge \frac{4}{5}\delta\) \le \P\(\phi_{\ell}\ge 3, \;\text{for some}\;\ell\ge \ell_\varepsilon\)
    .
\end{equation}
The conditional variance of $\phi_{\ell}$ is $b'\coloneqq \sum_{i=1}^\ell
    \var\(Z_{i}\cond{} \cF_i\)\le (\delta/5)^{-2}b= o(n^{-1/3})$.
Then by~\cite[Theorem~1.6]{freedman1975},
\begin{equation}
    \P\(\phi_{\ell}\ge 3, \text{for some }\ell\ge \ell_\varepsilon\)\le \ee^3\left(\frac{b'}{3+b'}\right)^{3+b'}
    =
    o\left(n^{-1}\right)
    .\qedhere
\end{equation}
\end{proof}

\subsection{Mixing time}
In this section we prove
\cref{cutoff}. We adapt the arguments in~\cite{bordenave2018}, which established the same result under a bounded degree assumption.
As we will see, the role played by the boundedness of the in-degrees in~\cite{bordenave2018} will be replaced by \cref{lem:s}.
We prove separately the lower and the upper bound on the total variation distance.

\subsubsection{Proof of the upper bound of \texorpdfstring{\cref{cutoff}}{Theorem~\ref{cutoff}}}\label{sec:upper:proof}
Let us first explain the overall strategy of the proof, which is based on ideas introduced in~\cite{bordenave2018,bordenave2019}.
For each pair of vertices $x, y\in[n]$,
we shall define a set $\wt \cP(x,y,t,G)\subset \cP(x,y,t,G)$ of \emph{nice paths} of length $t$ starting at $x$ and ending at $y$.
We will also let $\wt{P}^t(x,y)\le {P}^t(x,y)$ denote
the probability to go from $x$ to $y$ in $t$ steps following a nice path, i.e.,
\begin{equation}\label{eq:ptildexy}
    \wt{P}^t(x,y)
    \coloneqq
    \sum_{\p\in\wt \cP(x,y,t,G)}\w(\p).
\end{equation}
Suppose that for some $\delta>0$ and some probability distribution $\wt\pi$ on $[n]$ it holds that
\begin{equation}\label{eq:proxy}
    \wt{P}^t(x,y)\le (1+\delta)\wt\pi(y)+\frac{\delta}{n},\qquad \forall x,y\in[n].
\end{equation}
For $x\in \bbR$, define $[x]_+=\max\{x,0\}$. If \cref{eq:proxy} holds,
then
\begin{equation}\label{eq01}
    \[\wt\pi(y)-P^t(x,y)\]_+
    \le
    \big[\wt\pi(y)-\wt{P}^t(x,y)\big]_+
    \leq
    (1+\delta)\wt\pi(y)+\frac{\delta}{n}-\wt P^t(x,y) .
\end{equation}
Therefore,
\begin{equation}\label{eq3}
    \begin{aligned}
        \| P^t(x,\cdot)-\wt\pi \|_{\tv}=
        &\frac{1}{2} \sum_{y\in[n]}|\wt\pi(y)-P^t(x,y)|
        =
       \sum_{y\in[n]}\[\wt\pi(y)-P^t(x,y) \]_+                             \\
        \le
        & \sum_{y\in[n]}\Big((1+\delta)\wt\pi(y)+\frac{\delta}{n}-\wt P^t(x,y) \Big)
        =
        2\delta + \wt q(x),
    \end{aligned}
\end{equation}
where $\wt q(x)$ is the probability that the walk starting at $x$ follows a path of length $t$ which
is not nice, i.e.,
\begin{equation}\label{eq:ptildexy2}
    \wt q(x) \coloneqq 1-\sum_{y\in [n]} \sum_{\p\in\wt \cP(x,y,t,G)}\w(\p).
\end{equation}

Next, we define the set of nice paths.
\begin{definition}\label{def:nice}
Let $\eta\in(0,1)$ such that \cref{cond:main} is satisfied. Fix $\varepsilon\in(0,\eta/6)$, $h=h_\varepsilon$, $\gamma=\frac{\varepsilon}{80\log(\Delta^+)}$ and set
\begin{equation}\label{eq:def-s-t}
    s\coloneqq(1-\gamma)\tent,\qquad t\coloneqq s + h+1.
\end{equation}
{Notice that $s= \tent-\frac{h}{4\hh}$, and that our choice of parameters $\varepsilon,h,s,t$ and $\gamma$ is such that $t=(1+\delta')\tent + 1$ for some $\delta'=\delta'(\varepsilon)>0$ such that $\delta'(\varepsilon)\to 0$ as $\varepsilon\to 0$.}
We define the set of nice paths as follows.
Given $x,y\in[n]$, a path $\p$ of length $t$ starting at $x$ and ending at $y$ is \emph{nice}, if
\begin{enumerate}[(i)]
    \item\label{req:1} the first $s$ steps are contained in the tree $\cT_x(s)$ defined in \cref{suse:out-phase};
    \item\label{req:2} the first $s+1$ steps of $\p$ form a path $\p_{s+1}$
        such that
        \begin{equation}
            \w(\p_{s+1})\le n^{-1+2\gamma};
        \end{equation}
    \item\label{req:3} the last $h$ steps of $\p$
        form the unique path in $G$ of length at most $h$ from its origin to its destination.
\end{enumerate}
\end{definition}
Notice that (ii) implies that for every nice path $\p$
\begin{equation}\label{eq:hp4}
\w(\p)\le \w(\p_{s+1})\cdot 2^{-h}\le 2^{-h}n^{-1+2\gamma}\le n^{-1-\delta'},
\end{equation}
{for some $\delta'>0$, where we use the definition of $h=h_\varepsilon$ and $\gamma$, and the fact that $\log(2)>2/3$}.

With this definition at hand, we show
that, \ac{whp}, starting from any locally-tree-like vertex, the quenched probability to follow a nice path converges to 1.
\begin{proposition}\label{prop:nice}
    Recall the definition of $V_\varepsilon$ in \cref{def-hslash*}.
    The probability $\wt q(x)$ defined in \cref{eq:ptildexy2} satisfies
    \begin{equation}
        \max_{x\in V_\varepsilon} \wt q(x)\overset{\P}{\longrightarrow}0.
    \end{equation}
\end{proposition}
\begin{proof}\label{sec:Bennett}

We are going to check that the three requirements in \cref{def:nice} are satisfied \ac{whp} by the quenched law uniformly in the starting state $x\in V_\varepsilon$. Notice that requirement (i) follows, uniformly in $x\in V_\varepsilon$, by \cref{lem:bennett}. By our choice of the parameter $\gamma$, {$n^{-1+2\gamma}\ge \ee^{-(1-\delta')\hh s}$ if $\delta'>0$ is small enough, and therefore} the requirement in (ii) is a simple consequence of \cref{LLN}, which holds uniformly in $x\in[n]$. Finally, the probability of the event in requirement (iii) can be bounded from above by the probability that at time $s+1$ the walk is in $V_\varepsilon$. Moreover,  for any $\ell\le s$:
\begin{equation}
\min_{x\in V_\varepsilon}\PP_x\(X_{s+1}\in V_\varepsilon\)
\ge 1-\max_{v\in [n]}\PP_v\(X_{\ell}\not\in V_\varepsilon\).
\end{equation}
where the last step follows from \cref{lemmaG2}.
\end{proof}

The last ingredient we need is an approximation $\wt\pi$ for the stationary distribution $\pi$,
which satisfies \cref{eq:proxy}.
To this end,
we define
    \begin{equation}
\label{eq:pitilde}
\wt\pi= \mu_{\rm in}P^{h}.
\end{equation}
\begin{proposition}\label{prop:proxy}
For all $\delta>0$, 
    \begin{equation}
        \P\(\max_{x\in[n]}
        \wt{P}^t(x,y)\le (1+\delta)\mu_{\rm
        in}P^h(y)+\frac{\delta}{n}
        , \;\forall y\in[n]\)=1-o(1)
        .
    \end{equation}
\end{proposition}
We first conclude the proof of the upper bound in \cref{cutoff}, and then provide the proof of \cref{prop:proxy}.
\begin{proof}[Proof of the upper bound of \cref{cutoff}]
    Fix the parameters $\varepsilon, h,s,t$ and $\gamma$ as above.
    It follows from the argument in
    \cref{eq3} and \cref{prop:nice,prop:proxy} that
    \begin{equation}
 \max_{x\in V_\varepsilon}\|P^{t}(x,\cdot)-\mu_{\rm in}P^{h} \|_{\tv} \overset{\P}{\longrightarrow} 0.
    \end{equation}
Thus, the upper bound in \cref{cutoff}{, i.e., when $\rho>1$,}
holds for all starting states $x\in V_\varepsilon$ with $\wt\pi$ in place of $\pi$.
Setting $ t'=t+\ell$,
    \begin{equation}\label{eq:339}
        \max_{v\in[n]}\|P^{t'}(v,\cdot)-\mu_{\rm in}P^{h} \|_{\tv} \le \max_{x\in V_\varepsilon}\|P^{t}(x,\cdot)-\mu_{\rm in}P^{h} \|_{\tv}+
        \max_{v\in[n]}\PP_v(X_{\ell}\not\in V_\varepsilon).
    \end{equation}
    Taking, e.g., $\ell =\log\log(n)$, by \cref{lemmaG2}, the last term in \cref{eq:339} tends to zero in probability.
    Thus,
    \begin{equation}\label{eqbound}
        \max_{x\in[n]}\|P^{t'}(x,\cdot)-\mu_{\rm in}P^{h} \|_\tv\overset{\P}{\longrightarrow} 0.
    \end{equation}
    Since the latter convergence is uniform in the starting position $x$, it must hold for every initial distribution. Starting at stationarity, it follows by \cref{eqbound} that
    \begin{equation}\label{eqbound2}
        \|\pi-\wt\pi \|_\tv=\|\pi-\mu_{\rm in}P^{h} \|_\tv \overset{\P}{\longrightarrow}  0.
    \end{equation}
   Hence by the triangular inequality
     \begin{equation}\label{eqboundo1}
        \max_{x\in[n]}\|P^{t'}(x,\cdot)-\pi \|_\tv\overset{\P}{\longrightarrow} 0.
    \end{equation}
    {By monotonicity of the total variation distance, this implies the same estimate for all times larger than $t'$. Moreover, since $\varepsilon$ can be taken arbitrarily small, the above defined $t'$ is sufficient to cover the whole range of times of the form $\rho\,\tent$, with fixed $\rho>1$}.
\end{proof}

\begin{proof}[Proof of \cref{prop:proxy}]
Fix the parameters $\varepsilon, h,s,t$ and $\gamma$ as above.
Given $x,y\in[n]$, we first generate the pair $(\cG_x(s),\cT_x(s))$ using the  construction in \cref{suse:out-phase}, and then sample the in-neighborhood of $y$  up to height $h$, $\cB^-_y(h)$,
using the procedure in \cref{suse:local} with the \ac{bfs} rule. Some of the matching defining $\cB^-_y(h)$ may have been already revealed during the construction of  $\cG_x(s)$. Call $\kappa_y$ the additional (random) number of matchings needed to complete the construction of $\cB_y^-(h)$. In total, thanks to \cref{lemmaGW} and \cref{eq:kappax:s}, \ac{whp} and for all $x,y\in [n]$, the total number of matchings revealed is bounded by
\begin{equation}\label{eq:req}
    \kappa_x+\kappa_y\le 2 \ee^{s\hhigh}
  +  n^{\frac{1}{2}+\varepsilon}=
   2 \ee^{(1-\gamma)\tent\hhigh}+ n^{\frac{1}{2}+\varepsilon}\le 3n^{1-\gamma^2},
\end{equation}
for all sufficiently large $n$, where we used the definitions of $\gamma$ and $\hhigh$.

Let $\sigma=\sigma(x,y)$ denote the partial environment obtained after the generation of the neighborhoods $(\cG_x(s),\cT_x(s))$ and $\cB^-_y(h)$. Let $\cW_{x,y}$ be the event that $\sigma$ satisfies \cref{eq:req}. Let $\cF_\sigma$ denote the set of unmatched heads $f$ with $v_f\in \partial\cB^-_y(h)$ that admit a unique path of length $h$ ending at $y$, and with a slight abuse of notation, let $\widehat \w(f)$ denote the weight of such a unique path. Similarly, call $\cE_\sigma$ the set of unmatched tails at height $s$ in $\cT_x(s)$. Notice that, as a consequence of these definitions,
\begin{equation}\label{349}
    \sum_{e\in\cE_\sigma}\widehat\w(e)\le 1,\qquad \sum_{f\in\cF_\sigma}{\widehat \w}(f)\le m\, \mu_{\rm in} P^h(y).
\end{equation}
Moreover, our construction is such that the probability to follow a nice path of length $t$ from $x$ to $y$ can be written as
\begin{equation}\label{eq:nice_traj}
    \wt P^t(x,y)=\sum_{e\in\cE_\sigma}\sum_{f\in\cF_\sigma}\widehat\w(e){\widehat \w}(f)\ind_{\widehat\w(e)\le n^{-1+2\gamma}}\ind_{\omega(e)=f},
\end{equation}
where we use the fact that for a nice path $\p$ the first $s+1$ steps are such that $\w(\p_{s+1})=\widehat\w(e)$ for a suitable $e\in\cE_\sigma$.

Given the partial environment $\sigma$, the sampling of the full environment is completed by using a random permutation $\omega$ of the remaining $m-\kappa_x-\kappa_y$ heads and tails. In particular, conditionally on the  partial environment $\sigma$,
for all $(e,f)\in\cE_\sigma\times\cF_\sigma$, the random variable $\ind_{\omega(e)=f}$
is marginally distributed as a Bernoulli random variable with parameter
$\frac{1}{m-\kappa_x-\kappa_y}$.
Therefore, using \cref{349}, for each $\sigma\in \cW_{x,y}$ we estimate
\begin{equation}\label{eq:use1}
    \begin{aligned}
    \E\[\wt P^t(x,y)\cond{}\sigma\]& = \frac{1}{m-\kappa_x-\kappa_y}\sum_{e\in\cE_\sigma}\sum_{f\in\cF_{\sigma}}\widehat\w(e){\widehat \w}(f)\ind_{\widehat\w(e)\le n^{-1+2\gamma}} \\
     & \le \(1+3n^{-\gamma^2} \)\mu_{\rm in}P^h(y).
    \end{aligned}
\end{equation}
It follows that, uniformly in $x,y$ and $\sigma\in \cW_{x,y}$, for all fixed $\delta>0$ and all $n$ large enough, one has \begin{equation}\label{eq:use2}
    (1+\delta/2)\E\[\wt P^t(x,y) \cond{} \sigma\]\le (1+\delta)\mu_{\rm in}P^h(y).
\end{equation}

A concentration result for  functions of a random permutation due to Chatterjee (see~\cite[Proposition 1.1]{chatterjee2007}) shows that for all $a>0$,
\begin{equation}\label{eq:chatterjee}
    \P\(\left|\wt P^t(x,y)-\E\left[\wt P^t(x,y)\cond{}\sigma\right]\right|\ge a\cond{} \sigma \)\le
 2   \exp\(-\tfrac{a^2}{2\|c\|_\infty\(2\E\[\wt P^t(x,y) \cond{}\sigma\]+a\)}\),
\end{equation}
where, using \cref{eq:hp4},
\begin{equation}\|c\|_{\infty}:=\max_{(e,f)\in\cE_\sigma\times\cF_\sigma}\widehat\w(e){\widehat \w}(f)\ind_{\widehat\w(e)\le n^{-1+2\gamma}}\le 2^{-h} n^{-1+\gamma}\le n^{-1-\delta'} .\end{equation}
Choosing $a=\frac{\delta}{2}\E\left[\wt P^t(x,y)\cond{}\sigma\right]+\frac{\delta}{n}$ in \cref{eq:chatterjee},
and using \cref{eq:use2}, we get that, for all $\delta>0$,
uniformly in $x,y\in[n]$ and $\sigma\in \cW_{x,y}$ one has
\begin{equation}\label{eq:chat-final}
    \P\(\wt P^t(x,y)\ge (1+\delta )\mu_{\rm in}P^{h}(y)+\frac{\delta}{n}\: \bigg\rvert\:\sigma \)\le \exp(-\log^2(n)),
\end{equation}
for all $n$ large enough.
Note that a lower bound of $\E\left[\wt P^t(x,y)\cond{}\sigma\right]$ is not needed here. In fact, this bound also holds if $y$ has in-degree $d_y^-=0$, in which case $ \E\left[\wt P^t(x,y)\cond{}\sigma\right]=0$.

In order to get the desired conclusion, let $\cW:=\cap_{x,y}\cW_{x,y}\supset\cS_\varepsilon^-$; see \cref{eq:seps-}.
Define
\begin{equation}\cZ_{x,y}\coloneqq\left\{\wt P^t(x,y)\le (1+\delta )\mu_{\rm in}P^{h}(y)+\frac{\delta}{n}\right\}.\end{equation}
Then, using \cref{lemmaGW} one has
\begin{equation}
    \begin{aligned}
    \P\(\cap_{x,y\in[n]}\cZ_{x,y}\)& = 1-\P\(\cup_{x,y\in[n]}\cZ_{x,y}^c \)
    \ge 1-\P\(\cup_{x,y\in[n]}\cZ_{x,y}^c\cap \cW \)-\P(\cW^c)                                                \\
    & \ge 1-n^2\:\max_{x,y\in [n]}\P\(\cZ_{x,y}^c\cap \cW_{x,y} \) - o(1).
    \end{aligned}
\end{equation}
Moreover,  \cref{eq:chat-final} implies
\begin{equation}
n^2\:\max_{x,y\in [n]}\P\(\cZ_{x,y}^c\cap \cW_{x,y} \)
\le
   n^2\:\max_{x,y\in [n]} \max_{\sigma\in \cW_{x,y}}\P\(\cZ_{x,y}^c\cond{}\sigma \)=o(1).
   \qedhere
\end{equation}
\end{proof}

In what follows, we will use the following corollary of the upper bound in \cref{cutoff}. For $t\in \bbN$ and $y\in [n]$, define
\begin{equation}\label{eq:def-lambda2}
\mu_t(y)\coloneqq \frac{1}{n}\sum_{x\in[n]}P^t(x,y).
\end{equation}
\begin{corollary}\label{coro:mixing}
    With high probability,
    for all $t=\Omega(\log^3(n))$,
    \begin{equation}
        \max_{x\in[n]}\|P^t(x,\cdot)-\pi \|_\tv\le \ee^{-\log^{3/2}(n)}.
    \end{equation}
In particular,  $\|\mu_t-\pi \|_\tv\le \ee^{-\log^{3/2}(n)}$.
\end{corollary}
\begin{proof}
Let $d(s)=\max_{x\in[n]}\|P^s(x,\cdot)-\pi \|_\tv$.
It is standard that $d(ks)\le 2^{k}d(s)^k$ for all $k\in \bbN$; see~\cite[Section~4.4]{levin2017}.
The upper bound in \cref{cutoff} implies that \ac{whp} \ $d(2T_\ent)\le 1/2e$.
Therefore if $t=\Omega(\log^3(n))$ one may take $k=\Omega(\log^2(n))$ and $s=2T_\ent$ to conclude.
\end{proof}

\subsubsection{Proof of the lower bound in \texorpdfstring{\cref{cutoff}}{Theorem~\ref{cutoff}}}

We will use the fact that  \cref{LLN} implies that, uniformly on the starting point $x$, the distribution of the location of the random walk at time $t=(1-\beta)\tent$ is concentrated on a set of size $O(n^{1-\beta^2})$.
More precisely, pick $\beta\in(0,1)$
and $t= (1-\beta)\tent$. For all $x,y\in[n]$, call $\cP^\beta_{x,y}$ the set of paths of length $t$ starting at $x$ and ending at $y$ having weight at least $\ee^{-(1+\beta)\hh t}=n^{-1+\beta^2}$. Clearly, for all $x\in[n]$
\begin{equation}\sum_{y\in [n]}\sum_{\p\in\cP^\beta_{x,y}}\w(\p)\le 1 .\end{equation}
Therefore,  for all $x\in[n]$
\begin{equation}\sum_{y\in [n]}|\cP^\beta_{x,y}|
\le n^{1-\beta^2}. \end{equation}
In particular, the set  $S_x:=\left\{y\in[n]\cond{} \cP^\beta_{x,y}\not=\emptyset  \right\}$ satisfies  $|S_x|\le n^{1-\beta^2}$ and, by \cref{LLN},
\begin{equation}\min_{x\in[x]}P^t(x,S_x)=1-o_\P(1).\end{equation}
Thus,
\begin{equation}
    \min_{x\in[n]} \|P^{t}(x,\cdot)-\pi \|_\tv\ge P^{t}(x,S_x)-\pi(S_x)=1-o_\P(1)-\max_{x\in[n]}\pi(S_x).
\end{equation}
\cref{prop:pi-small-set} below with $\delta=\beta^2/6$ implies that $\max_{x\in[n]}\pi(S_x)=o_\P(1)$, which concludes the proof of the  lower bound in \cref{cutoff}.

\begin{proposition}\label{prop:pi-small-set}
For any $\delta\,\in\(0, \tfrac{1}{6}\)$,
    we have
    \begin{equation}\P\(\forall S\subset [n], |S|\le n^{1-6\delta} ,\:\pi(S)\le n^{-\delta/2}  \)=1-o(1).\end{equation}
\end{proposition}
\begin{proof}
It suffices to prove the statement for $S$ of size exactly $L:=\ceil{n^{1-6\delta}}$.
By \cref{coro:mixing}, for $t=\log^3(n)$,
\begin{equation}\label{eq:corollary-mix}
    \max_{v\in[n]}\left|\pi(v)-\mu_t(v) \right|\le \ee^{-\log^{3/2}(n)},
\end{equation}
Hence, it is enough to prove that
\begin{equation}\label{eq:enough}
    \max_{S\::|S|=L}\P(\mu_t(S)\ge n^{-\delta})=o(n^{-L}),
\end{equation}
and then apply a union bound over all sets $S\subset[n]$ of cardinality $L$.

To prove \cref{eq:enough}, fix a set $S$ with cardinality $L$ and let $K=\delta^{-1} L$.
Consider the annealed random walk construction described in the beginning of the section with $K$
walks of length $t$ starting at uniform and independent random vertices.
Call $B_j$, $j\le K$, the event that the first $j$ walks end at a vertex in $S$.
Thanks to \cref{lem:s},
for each $j\le K$ and at each time $s\le t$ there are $A=o(\sqrt{(Kt+L)n})$ unmatched heads incident
to either $S$ or to the vertices visited by the first $j-1$ walks,
or by the $j$-th walk up to time $s$.
Therefore, conditionally on the first $j-1$ walks, the probability that the $j$-th walk ends at $S$ is at most
\begin{equation}\label{KNVY}
  \P_{\rm unif}^{{\rm an},K}\left(B_j\cond{}B_{j-1}\right)\le  \frac{K t}{n}+ \frac{t\: A}{m- Kt}\le \sqrt{L/n}\, \log^{5}n= o(n^{-2\delta}),
\end{equation}
where $\P_{\rm unif}^{{\rm an},K}$ is defined by \cref{eq:annx} with $\mu$ uniform over $[n]$.
Indeed, in order to end in $S$ the walk needs to visit at some $s\le t$ a vertex which is either in $S$ or has already been visited by one of the previous walks. The probability that such an event occurs at the initialization step is bounded by $Kt/n$, while $tA/(m-Kt)$ bounds the probability that the event occurs at some later step.
It follows that
\begin{equation}\label{CVVI}
    \E[(\mu_t(S))^K]=\P_{\rm unif}^{{\rm an},K}(B_K)=\P_{\rm unif}^{{\rm an},K}(B_1)\prod_{j=2}^{K} \P_{\rm unif}^{{\rm an},K}\left(B_j\cond{} B_{j-1}\right) =o(n^{-2L}).
\end{equation}
By Markov's inequality,
\begin{equation}\label{ERXC}
    \P\(\mu_t(S)\ge n^{-\delta}\)\le \frac{\E[(\mu_t(S))^K]}{n^{-L}} = o(n^{-L}).
\end{equation}
This implies \cref{eq:enough}.
\end{proof}

\section{Bulk behavior}
In this section we prove \cref{th:bulk}. As we will see, the distribution $\cL_n$ approximating the bulk values of the stationary distribution can be characterized as the almost sure limit of an $L^2$-bounded martingale.

\subsection{The martingale}
Fix $y\in[n]$, an arbitrary $h\in\bbN$, and consider the random tree $\cT^-_y(h)$ constructed in~\cref{sec:GW} with marks $\ell(\cdot)$. For $a\in \partial \cT^-_y(h)$, define
\begin{equation}
\w_{\cT}(a) \coloneqq d^-_{\ell(a)} \prod_{i=1}^{h}\frac{1}{d^+_{\ell(a_{i})}},
\end{equation}
where $(a_0,\dots,a_h=a)$ is the unique path joining $a$ with the root $a_0$ for witch
$\ell(a_{0}) = y$.
If $a=a_0$, then the empty product is interpreted as $1$ and we define $\w_{\cT}(a_0)= d^-_{y}$ in this case.

Define the random process
\begin{equation}\label{eq:martingale}
M_y(h) =\sum_{a\in \partial \cT^-_y(h)} \w_{\cT}(a), \qquad\forall h \ge 0.
\end{equation}
\begin{lemma}\label{lem:marting}
Let $\cF_{h}$ be the sigma algebra generated by the random tree $\cT_{y}^{-}(h)$. Then $(M_y(h))_{h\ge 0}$ is a martingale satisfying $\E\[M_y(h)\]=d^-_y$ and, uniformly in $h\in\N$, $\var(M_y(h))=O(d^-_y)$.
\end{lemma}
\begin{proof}
For simplicity, write $M_{h}=M_{y}(h)$, and note $M_0= \w_{\cT}(a_0)=d^-_y$.
For each $a\in\cT^-_y(h)$, let $N(a)$ denote the set of its children.
Then,  for all $h\ge0$,
\begin{equation}\label{eq:m1m0}
    M_{h+1}-M_{h}=\sum_{a\in\cT_y^-(h)}\frac{\w_\cT(a)}{d_{\ell(a)}^-}\( \sum_{b\in N(a)} \left(\frac{d^-_{\ell(b)}}{d^+_{\ell(b)}} - 1\right)\).
\end{equation}
Let $\cJ$ denote a random vertex in $[n]$ distributed as $\muout(x)$ defined in~\cref{eq:def-muout}. Using
\begin{equation}\label{eq:mean1}
    \E\[\frac{d_\cJ^-}{d_\cJ^+} \]=\sum_{v\in[n]}\frac{d_v^+}{m}\frac{d_v^-}{d_v^+}=1,
\end{equation}
and the fact that the marks in the tree are distributed according to $\muout$, we obtain
\begin{equation}\E[M_{h+1}\mid \cF_{h}]=M_{h}.\end{equation} Hence, $(M_h)_{h\ge 0}$ is a martingale with expectation $d_y^-$.    It remains to compute its variance. Let \begin{equation}\Sigma_h\coloneqq\var(M_{h+1}-M_h\mid \cF_h).\end{equation} By \cref{eq:m1m0},
$M_{h+1}-M_h$ is given by
\begin{equation}
    \sum_{a\in\cT_y^-(h)}\frac{\w_\cT(a)}{d_{\ell(a)}^-}\,Y_a,
\end{equation}
where, for each $a$,  $Y_a$ is the sum of $d^-_{\ell(a)}$ \ac{iid} copies of the random variable
$\frac{d^-_\cJ}{d^+_\cJ}-1$.
Therefore, by conditioning on $\cF_h$, we obtain
\begin{equation}\label{eq:sigma0}
    \Sigma_h=(A-1)\sum_{a\in\cT_y^-(h)}\frac{\w_\cT(a)^2}{d_{\ell(a)}^-} ,
\end{equation}
where, by \cref{cond:main},
\begin{equation}\label{eq:def-A}
    A\coloneqq  \E\[\(\frac{d^-_\cJ}{d^+_\cJ}\)^2\]=\sum_{v\in[n]}\frac{d_v^-}{m}
    \left(\frac{d_v^-}{d_v^+}\right)^{2}
    =O(1).
\end{equation}
Since
\begin{equation}\label{eq:sigma00}
    \sum_{b\in\cT_y^-(h+1)}\frac{\w_\cT(b)^2}{d_{\ell(b)}^-}=\sum_{a\in\cT_y^-(h)}\frac{\w_\cT(a)^2}{(d_{\ell(a)}^-)^2}\sum_{b\in N(a)}\frac{d_{\ell(b)}^-}{(d_{\ell(b)}^+)^2},
\end{equation}
we conclude that
\begin{equation}\label{eq:sigmah}
    \E\[\Sigma_{h+1}\mid \cF_{h}\]=  \E\[\frac{d^-_\cJ}{(d^+_\cJ)^2}\] \Sigma_{h}.
\end{equation}
Using  \cref{eq:mean1} and the fact that the out-degrees are at least $2$ we see that $\E\[\Sigma_{h+1}\mid \cF_{h}\] \le \frac12\Sigma_{h}$.
Thus, taking the expectation and applying induction on $h$, we have
\begin{equation}\label{eq:EBIU}
    \E\[\Sigma_{h}\] \le 2^{-h}\Sigma_0, \qquad \forall h \ge 0,
\end{equation}
where $\Sigma_0=\var(M_1)=(A-1)d_y^-$.
By orthogonality of the martingale increments, and using \cref{eq:EBIU}, it follows that
\begin{equation}
    \var(M_h) = \sum_{i=0}^{h-1} \E\[\Sigma_i\] \le 2(A-1)d_y^-.
\end{equation}
\end{proof}

\begin{corollary}\label{coro:marting}
    Fix $n\in\bbN$ and let $\cI$ be a uniform random vertex in $[n]$.
    Define
    \begin{equation}
    \Phi_h:=\frac{1}{\pl d \pr}M_{\cI}(h), \qquad h\ge 0,
    \end{equation}
where $\pl d \pr=m/n$ is the average degree. Then $(\Phi_h)_{h\ge 0}$ is a martingale satisfying $\E\[\Phi_h\]=1$ and its limit $\Phi_\infty\coloneqq\lim_{h\to\infty} \Phi_h$ exists almost surely and in $L^2$. Moreover, there exists  $C>0$ independent of $n,h$
such that for all $h\ge 0$
    \begin{equation}\label{eq:exp-l2-conv}
    \E\[\(\Phi_h-\Phi_\infty\)^2 \]\le C 2^{-h}.
    \end{equation}
\end{corollary}
\begin{proof}
The first assertion follows from \cref{lem:marting} by averaging over $y$. We are left to show \cref{eq:exp-l2-conv}. Arguing as in \cref{eq:sigmah,eq:EBIU},
\begin{equation}
    \label{eq:CDHY}
\begin{aligned}
\E[(\Phi_h-\Phi_\infty)^2]&=\sum_{j=h}^\infty \E\[\var\( \Phi_{j+1}-\Phi_{j} \cond{}\cF_{j} \) \]\\
&\le\E\left[\var(\Phi_1-\Phi_0)\cond{}\cF_0\right]\sum_{j=h}^\infty 2^{-j}\le 2^{-h+1}\:\frac{(A-1)}{\pl d \pr},
\end{aligned}
\end{equation}
where $A$, defined as in \cref{eq:def-A}, is uniformly bounded in $n$ thanks to \cref{cond:main}.
\end{proof}

\subsection{Proof of \texorpdfstring{\cref{th:bulk}}{Theorem~\ref{th:bulk}}}

Fix $n\in\N$ and consider the martingale $(\Phi_h)_{h\ge 0}$ in \cref{coro:marting}.
It follows from~\cite[Lemma 16]{bordenave2018} that
the random variable $ \Phi_\infty$ has law $\cL_n$ as in \cref{eq:distr-lim-mart}.
Hence, we are left to show that as $n\to\infty$ the convergence in \cref{eq:w1-conv} takes place.

Thanks to the characterization of $\cW_1$ convergence via non-expansive functions (see~\cite[Lemma 19]{bordenave2018}), it is enough to show that, for all $g:\R\to\R$ such that $|g(x)-g(y)|\le |x-y|$ and for all $x,y\in\R$,
\begin{equation}\label{eq:char}
    \frac{1}{n}\sum_{v\in[n]}g(n\pi(v))- \E[g(\Phi_\infty)] \overset{\P}{\longrightarrow}0.
\end{equation}
Since $|g(x)-g(y)|\le |x-y|$, for all $x,y\in\R$, for all $h\in\N$,
\begin{equation}\label{eq:char01}
    \Big|\frac{1}{n}\sum_{v\in[n]}g(n\pi(v))- \frac{1}{n}\sum_{v\in[n]}g(n\mu_{\rm in}P^h(v)) \Big|\le 2 \|\pi-\mu_{\rm in}P^h \|_\tv.
\end{equation}
From now we fix $h=h_\varepsilon$ as in \cref{def-hslash}. For definiteness, we take $\varepsilon=\eta/10$, where $\eta\in(0,1)$ is such that \cref{cond:main} holds.
Hence, by \cref{eqbound2},
\begin{equation}\label{eq:char001}
    \frac{1}{n}\sum_{v\in[n]}g(n\pi(v))=\frac{1}{n}\sum_{v\in[n]}g(n\mu_{\rm in}P^h(v))+o_{\P}(1).
\end{equation}
We now show that the first term on the \ac{rhs} of \cref{eq:char001} concentrates, that is
\begin{equation}\label{eq:claim-mart}
    \frac{1}{n}\sum_{v\in[n]}g(n\mu_{\rm in}P^h (v))=\frac{1}{n}\sum_{v\in[n]}\E\Big[g(n\mu_{\rm in}P^h(v))\Big]+o_\P(1).
\end{equation}
For every realization of the matching $\omega$ inducing the digraph,
define
\begin{equation}Z(\omega)\coloneqq \frac{1}{n}\sum_{v\in[n]}g(n\mu_{\rm in}P^h(v)).\end{equation}

Consider a realization $\omega'$ obtained from $\omega$ by switching two edges: there exists $e,e'\in E^+$ and $f,f'\in E^-$ such that $(\omega(e),\omega(e'))=(f,f')$, and $(\omega'(e),\omega'(e'))=(f',f)$, while $\omega=\omega'$ for all other tails in $E^+$. Then, as in \cref{eq:char01}
\begin{equation}
    |Z(\omega)-Z(\omega')|\le 2 \|\mu_{\rm in}P_\omega^h-\mu_{\rm in}P_{\omega'}^h  \|_\tv\eqqcolon b,
\end{equation}
where $P_\omega$ denotes the transition matrix of the random walk in the digraph induced by $\omega$.
Let $\cB^{+,\omega}_{v}$ denote the out-neighborhood of $v$ up to height $h$ in the digraph induced by $\omega$.
Notice that the probability that a random walk starting
with distribution $\mu_{\rm in}$ reaches a vertex $v$
after $h$ steps coincides under $\omega$ and $\omega'$ for all vertices $v\notin \cQ(\omega,\omega')$,
where
\begin{equation}
\cQ(\omega,\omega')\coloneqq \cB^{+,\omega}_{v_f}\cup\cB^{+,\omega'}_{v_f}\cup \cB^{+,\omega}_{v_{f'}}\cup\cB^{+,\omega'}_{v_{f'}}.
\end{equation}
Therefore, we can bound
\begin{equation}
    \begin{aligned}
    b&=\sum_{v\in[n]}\left|\mu_{\rm in} P_{\omega}^h(v)- \mu_{\rm in} P_{\omega'}^h(v) \right|\ind_{v\in\cQ(\omega,\omega') }
    \le 4(\Delta^+)^h
    W(\omega,\omega'),
    \end{aligned}
\end{equation}
where we use the simple uniform bound $\max_{v\in[n],\bar{\omega}}|\cB^{+,\bar{\omega}}_v|\le (\Delta^+)^h$,
and we define
\begin{equation}
    W(\omega,\omega') \coloneqq \max_{\bar{\omega}\in\{\omega,\omega'\}}\max_{v\in[n]}\mu_{\rm in} P_{\bar{\omega}}^h(v).
\end{equation}
Consider the event
\begin{equation}
    \cA=\left\{\omega: \; \max_{v\in[n]}\mu_{\rm in} P_\omega^h(v)\le n^{-\frac{1}{2}-\frac{\varepsilon}{5}} \right\}.
\end{equation}
If $\omega,\omega'\in\cA$ and since  $(\Delta^+)^h= n^{\frac{\varepsilon}{20}}$, then
$ b\le n^{-\frac{1}{2}-\frac{\varepsilon}{10}}$.
By a generalization of Azuma's inequality (see, e.g., Theorem 3.7 in~\cite{mcdiarmid1998}),
we have that, for all $\delta>0$,
\begin{equation}\label{eq:combes}
    \P\(\big| Z-\E[Z] \big|\ge \delta \)\le 2\(\P(\cA^c)+\exp\(-\frac{\delta^2}{m b^2} \) \).
\end{equation}
Since $m b^2\to 0$ as $n\to\infty$, to conclude the proof of \cref{eq:claim-mart} it suffices to show that $\P(\cA^c)=o(1)$.

Fix $v\in[n]$,  $K=2/\varepsilon$ and notice that
\begin{equation}
    \P\(\mu_{\rm in}P^h(v)>n^{-\frac{1}{2}-\frac{\varepsilon}{5}} \)\le \frac{\E\[\(\mu_{\rm in}P^h(v)\)^K \]}{n^{-\frac{K}{2}-\frac{K\varepsilon}{5}}}.
\end{equation}
Using the annealed process as in \cref{KNVY,CVVI}, replacing the uniform measure by $\mu_{\rm in}$ and $t$ by $h$, we infer that, for $n$ large enough
\begin{equation}\E\[\(\mu_{\rm in}P^h(v)\)^K \]\le \(\frac{Kh \Delta^-}{n}\)^K\le n^{-\frac{K}{2}-K\varepsilon},\end{equation}
where in the last inequality we used \cref{lem:s} and the fact that $\varepsilon<\eta/6$.
Therefore,
\begin{equation}
    \P\(\mu_{\rm in}P^h(v)>n^{-\frac{1}{2}-\frac{\varepsilon}{5}} \)\le n^{-\frac{4K\varepsilon}{5}}=o\(n^{-1} \).
\end{equation}
By a union bound over $v\in[n]$ we get $\P(\cA^c)=o(1)$. This ends the proof of \cref{eq:claim-mart}.

Thanks to \cref{eq:char001,eq:claim-mart}, the proof of \cref{th:bulk} will be completed by showing
\begin{equation}\label{eq:finale}
    \left|\E[Z]-\E\[g\(\Phi_\infty\)\] \right|\to 0\,,\quad n\to\infty.
\end{equation}
Since $|g(x)-g(y)|\le |x-y|$, using Cauchy-Schwarz inequality and \cref{eq:exp-l2-conv}  we have
\begin{equation}
    \left|\E\[g\(\Phi_{h}\)\]-\E\[g\(\Phi_\infty\)\]\right| \le \E\[\left|\Phi_{h}-\Phi_\infty\right|\]\le \sqrt{\E\[\left(\Phi_{h}-\Phi_\infty\right)^2\]}=O(2^{-h/2}).
\end{equation}
Hence, by the triangular inequality, it suffices to show that $\left|\E[Z]-\E\[g\(\Phi_h\)\] \right|\to 0$.

Let $\wt\P$ denote the joint law of the uniform random choice of $\cI\in[n]$ and the coupled construction of the in-neighborhood of $\cI$ and the random tree in \cref{sec:GW}, with root having label $\cI$. Notice that if the coupling succeeds, then one has
$\Phi_k=n\mu_{\rm in}P^{k}(\cI)$ for all ${k\le h}$. Call $\cC$ the event that the coupling succeeds. By \cref{coro:coupling} we know that $\wt\P(\cC^c)\le n^{-1/4}$. Following the same argument as in \cref{eq:char01} and by Cauchy-Schwarz inequality,
\begin{equation}   \label{eq:CS}
    \big|\E[Z]-\E[g(\Phi_{h})]\big| \le\wt\E\[\big|n\mu_{\rm in}P^h(\cI)-\Phi_{h} \big|\:\cdot \ind_{\cC^c} \]
    \le\sqrt{\wt\P(\cC^c)\:\wt\E\[\big(n\mu_{\rm in}P^h(\cI)-\Phi_{h} \big)^2 \]},
\end{equation}
where $\wt\E$ denote taking expectation under the law $\wt\P$.
Therefore, it is enough to show that
\begin{equation}\label{eq:est2}
    \wt\E\[\big(n\mu_{\rm in}P^h(\cI)-\Phi_{h} \big)^2\] = o(n^{1/4}).
\end{equation}
To prove \cref{eq:est2} we write
\begin{equation}
    \wt\E\[ \big(n\mu_{\rm in}P^h(\cI)-\Phi_{h} \big)^2\]\le 2\E\[(n\mu_{\rm in}P^h (\cI) )^2 \]+2\E\[\Phi_h^2 \].
\end{equation}
By \cref{coro:marting}, $\E[\Phi_h^2]=O(1)$. Concerning the first term above, notice that
\begin{equation}
    \label{eq:JWVX}
    \begin{aligned}
        \E\[\(n\mu_{\rm in}P^h(\cI) \)^2 \]&=n\sum_{x\in[n]}\E\[(\mu_{\rm in} P^h(x))^2\]\\
        &\le\frac{1}{n}\sum_{x\in[n]}\E\Big[\Big(\textstyle{\sum_{y\in[n]}} d_y^-\ind_{x\in\cB^+_y} \Big)^2 \Big]\\
        &=\frac{1}{n}\sum_{x,y,z\in[n]}
        d_y^-d_z^-\P\(x\in\cB^+_y,\:x\in\cB^+_z  \).
    \end{aligned}
    \end{equation}
Arguing as in \cref{eq:LLXS}, we obtain
\begin{equation}
    \P\(x\in\cB^+_y,\:x\in\cB^+_z \)\le (\Delta^+)^h\frac{d_x^-}{n} \,\ind(y=z) +
    (\Delta^+)^{2h}\frac{d_z^- d_x^-}{n^2} + (\Delta^+)^{4h}\frac{(d_x^-)^2}{n^2}.
\end{equation}
Recall that \cref{cond:main} implies that $\sum_{x\in[n]}(d_x^-)^2=O(n)$. Thus,
we obtain
\begin{equation}
    \E\[\(n\mu_{\rm in}P^h(\cI) \)^2\] = O\((\Delta^+)^{4h}\).
\end{equation}
By our choice of $\varepsilon$ and $h=h_\varepsilon$, we have $(\Delta^+)^{4h}=o(n^{1/4})$, which ends the proof of \cref{eq:est2}.

\section{Lower bounds}\label{sec:lowbounds}

\subsection{Access probabilities to the maximum in-degree vertex}\label{sec:SMindeg}

\begin{proposition}\label{prop:LB2}
Assume that $\Delta^-=\Delta^-_n\to \infty$ as $n\to \infty$.
For every sufficiently small $\varepsilon>0$ and for any $y\in [n]$ with $d^-_y=\Delta^-$, we have
\begin{equation}\label{eq:GYOS}
    \P\(
        \min_{x\in V_{\varepsilon}}
            P^{t}(x,y) \geq (1-\varepsilon)\frac{\Delta^-}{m}
    \)
    =
    1-o(1)
    ,
\end{equation}
where $t=(1-\gamma)\tent+h_\varepsilon+1$,
$\gamma=\frac{\varepsilon}{80\log(\Delta^+)}$,
and $\tent$, $h_\varepsilon$ and $V_{\varepsilon}$ are defined as in \cref{eq:H,def-hslash,def-hslash*}.
\end{proposition}

Throughout \cref{sec:SMindeg} we write $h=h_\varepsilon$,  and set $s$ and $t$ as in \cref{eq:def-s-t}.
Fix $x\in [n]$.
Recall the out-neighborhood exploration defined in~\cref{suse:out-phase}
which exposes $\cG_x(s)$ and $\cT_x(s)$ in
at most $\kappa_x=\kappa_{x}(s)\le 2n^{1-\gamma^2}$ steps (see \cref{eq:kappax:s}).
We let $\sigma^+_x$ be the partial pairing obtained after the generation of $\cG_x(s)$.

Generate the in-neighborhood of $y$ using the sequential generation in~\cref{sec:SG} process according to the \ac{bfs} rule. Conditional on $\sigma^+_x$, the in-neighborhood generation constructs another sequence $\procH$ that exposes edge by edge the subgraph induced by $\cB^-_y(h)$. Let $\kappa_{y}$ be the number of edges that have been paired during the generation of the in-neighborhood. Let $\sigma=\sigma(x,y)$ be the partial pairing revealed after the two exploration processes.  Let $\omega$ be a complete pairing of half-edges chosen uniformly at random among all extensions of $\sigma$.

Given $\sigma$, call $\cE_\sigma$ the set of unmatched tails at height $s$ in $\cT_x(s)$ and $\cF_\sigma$ the set of unmatched heads incident to $\partial\cB^{-}_y(h)$ that admit a unique path of length $h$ ending at $y$.
Recall that, for $f\in \cF_\sigma$, ${\widehat \w}(f)$ is defined as before \cref{349}.

We now use the notion of nice path given in~\cref{def:nice}.
Let $\wt P^t(x,y)\le P^t(x,y)$ be the probability of following a nice path of length $t$ from $x$ to $y$.
Conditional on $\sigma$, \cref{eq:nice_traj} holds:
\begin{equation}
    \wt{P}^{t}(x,y)
    =
    \sum_{e\in \cE_{\sigma}}
    \sum_{f\in \cF_{\sigma}}
    \widehat{\w}(e) \widehat{\w}(f) \ind_{\widehat{\w}(e)\le n^{-1+2\gamma}} \ind_{\omega(e)=f}
\end{equation}
Observe that
\begin{equation}\label{ZYZI2}
    \E\left[\wt{P}^{t}(x,y) \cond{} \sigma \right]
    \ge
    \frac{1}{m} A_{x,y}(\sigma) B_{x,y}(\sigma)
    ,
\end{equation}
where
\begin{equation}
A_{x,y}(\sigma) = \sum_{e \in \cE_{\sigma}} \widehat\w(e)\ind_{\widehat\w(e)\le n^{-1+2\gamma}}\quad \text{and}\quad
B_{x,y}(\sigma) = \sum_{f \in \cF_{\sigma}} {\widehat \w}(f).
\end{equation}

Choose $\delta > 0$ sufficiently small such that $(1-\delta)^{5} > 1-\varepsilon$. Since $y$ is fixed, we define
\begin{equation}\label{GJQL2}
\cY_{x}=\{\sigma: A_{x,y}(\sigma)\ge (1-\delta)^2, B_{x,y}(\sigma)\ge (1-\delta)^2 \Delta^-\}.
\end{equation}
If $\sigma \in \cY_{x}$, we have
\begin{equation}
    \E[\wt{P}^{t}(x,y) \,| \, \sigma]
\ge
\frac{(1-\delta)^4 \Delta^-}{m}
.
\end{equation}
Define also the event
\begin{equation}
\cY= \cap_ {x\in [n]} \(\{x\notin V_{\varepsilon}\} \cup\cY_{x}\).
\end{equation}
We state the following fact that we will prove later.
\begin{lemma}\label{lem:typ2}
    We have $\P\(\cY\)= 1-o(1)$.
\end{lemma}

\begin{proof}[Proof of~\cref{prop:LB2}]
Define the event
\begin{equation}
\cZ_{x} = \left\{P^{t}(x,y) \ge (1-\varepsilon)\frac{\Delta^{-}}{m}\right\}.
\end{equation}
Let $\sigma\in \cY_x$.
Swapping two pairings in $\omega$ which are not fixed by $\sigma$ can change $\wt{P}^{t}(x,y)$ by at most $\|c\|_\infty\le n^{-1-\delta'}$ by~\cref{eq:hp4}.
Chatterjee's inequality in \cref{eq:chatterjee} implies that
\begin{equation}\label{SKEM2}
    \begin{aligned}
        \P\(\cZ_x^c\mid \sigma\)
        &\le
        \P\(
            \wt{P}^{t}(x,y) \le (1-\delta)\E[\wt{P}^{t}(x,y) \,| \, \sigma]
            \cond
            \sigma
        \)
        \\
        &
        \le 2 \exp\left(-\frac{(1-\delta)^{4}(2+\delta)^{-1} \delta^2 \Delta^{-}}{2 \|c\|_\infty m}\right)
        =
        o\left(n^{-1} \right)
        .
    \end{aligned}
\end{equation}
We then have
\begin{equation}\label{FINP}
    \begin{aligned}
    &
    \P\(
        \cup_{x \in [n]} (\{x\in V_\varepsilon\}\cap\cZ_{x}^c)
    \)
    \le
    \P\(
        \cup_{x\in [n]}
        \(\cZ_{x}^c\cap \cY_{x}\)
    \)
    +
    \P\(\cY^{c}\)
    =
    o(1)
    ,
    \end{aligned}
\end{equation}
where the first term is bounded by~\cref{SKEM2} and the second by \cref{lem:typ2}.
\end{proof}

\begin{proof}[Proof of~\cref{lem:typ2}]
Define the events
\begin{equation}\label{NNMQ2}
\begin{aligned}
    \cY^{(1)} &= \cap_{x\in [n]}\left(\{x \notin V_{\varepsilon}\} \cup \{A_{x,y}(\sigma)\ge (1-\delta)^2\}\right),\\
    \cY^{(2)} &=\cap_{x\in [n]} \{B_{x,y}(\sigma)\ge (1-\delta)^2\Delta^-\},
\end{aligned}
\end{equation}
and note that $\cY = \cY^{(1)}\cap  \cY^{(2)}$.

We first focus on $\cY^{(1)}$. Let $\cE_{x}(s)$ be the set of tails at height $s$ in $\cT_x(s)$. Write
\begin{equation}\label{eq:defsA}
A_{x}(\sigma)
=\sum_{e \in \cE_{x}(s)} \widehat\w(e)\ind_{\widehat\w(e)\le n^{-1+2\gamma}}
\quad
\text{and}
\quad
\overline{A}_{x,y}(\sigma) =A_{x}(\sigma) -A_{x,y}(\sigma).
\end{equation}
Let $\wt{q}_0(x)$ and $\wt{q}(x)$ be the probabilities that the walk starting at $x$ violates the condition of a nice
path within the first $s$ and $t$ steps respectively.
So $\wt{q}_0(x)\le \wt{q}(x)$. \cref{prop:nice} implies that \ac{whp},
\begin{equation}\label{eq:A0}
\min_{x\in V_\varepsilon} A_{x}(\sigma) \ge
1- \max_{x\in V_\varepsilon}\wt{q}_0(x) \ge 1- \max_{x\in V_\varepsilon} \wt{q}(x) \ge 1-\delta,
\end{equation}
for sufficiently large $n$.

We now turn our attention to the in-neighborhood exploration to bound $\overline{A}_{x,y}(\sigma)$.
Let $\cT^-_y$ be a Galton-Watson tree as defined in~\cref{sec:GW}.
Consider the event
\begin{equation}
\cC_{1}=\{\cB^-_y(h)=\cT^-_y(h),\kappa_y\le n^{1/2}\}
.
\end{equation}
That is, the coupling succeeds up to depth $h$, and not too many edges are revealed by it.
Since under $\cC_{1}$ we have $\kappa_y\le |\cB^-_y(h)|$,
\cref{le:coupling} implies that $\P(\cC_{1})=1-o(1)$,
provided that $\varepsilon<\eta$.

Recall that $\widehat{\w}(e)\le n^{-1+2\gamma}$ for all $e\in \cE_x(s)$.
Moreover, $\E[\ind_{\cC_{1}}\ind_{e\notin \cE_\sigma}]= \P\(\cC_{1}, e\notin \cE_\sigma\)$
is uniformly bounded from above by $\frac{\kappa_y}{m-\kappa_x-\kappa_y}\le n^{-1/2}$.
Therefore, the random variable $\ind_{\cC_{1}}\overline{A}_{x,y}(\sigma)$
is stochastically dominated by $X=\sum_{i=1}^m X_i$,
where the $X_i$ are independent random variables satisfying
$X_i\in [0,n^{-1+2\gamma}]$
and $\E[X]\le mn^{-3/2+2\gamma}\le n^{-1/3}$ for $\varepsilon$ small enough.
Using Hoeffding's inequality for the sum of independent bounded random variables
(see, e.g., Theorem 2.5 in~\cite{mcdiarmid1998}),
we obtain
\begin{equation}\label{eq:A_Hoeff}
\P\(\ind_{\cC_{1}}\overline{A}_{x,y}(\sigma)\ge \delta/2\)
\le \P\(X\ge \E[X]+\delta/3\)\le \exp\(-\frac{2\delta^2}{9m n^{-2+4\gamma}}\)
=o(n^{-1}).
\end{equation}
By $\P(\cC_{1})=1-o(1)$ and a union bound with \cref{eq:A_Hoeff}, whp
\begin{equation}\label{eq:overA}
    \max_{x\in[n]} \overline{A}_{x,y}(\sigma) =\max_{x\in [n]} \ind_{\cC_{1}}\overline{A}_{x,y}(\sigma) \le \delta/2.
\end{equation}
Combining \cref{eq:A0,eq:overA}, whp
\begin{equation}
\min_{x\in V_\varepsilon}A_{x,y}(\sigma) \ge \min_{x\in V_\varepsilon} A_{x}(\sigma) - \max_{x\in[n]} \overline{A}_{x,y}(\sigma)\ge (1-\delta)-\frac{\delta}{2}\ge (1-\delta)^2\;.
\end{equation}
for $\delta\le 1/2$, which  implies that $\P\((\cY^{(1)})^c\)=o(1)$.

It remains to bound the probability of $\cY^{(2)}$. It will be convenient to interchange the order of the out- and in-neighborhood exploration processes. We first generate $\cB^-_y(h)$ which reveals the set of heads at distance $h$ from $y$, denoted by $\cF_{y}(h)$, and then generate $\cG_x(s)$. Let $\sigma^-_y$ be the partial paring revealed by generating $\cB^-_y(h)$, and call again $\sigma=\sigma(x,y)$ the partial pairing obtained after the generation of both $\cB^-_y(h)$ and $\cG_x(s)$.

For all $j\ge 0$, define
\begin{equation}\label{eq:def-gamma}
    \Gamma_y(j) = \sum_{z\in \partial \cB^-_h(j)} d^-_z P^j(z,y).
\end{equation}
Under $\cC_{1}$, there is a unique path of length  $h$ from each head $f$ incident to $\partial \cB_y^-(h)$ to $y$, so we can write
\begin{equation}\label{eq:defBs}
\Gamma_y(h) = \sum_{f\in\cF_{y}(h)} {\widehat \w}(f) \quad\text{and} \quad \overline{B}_{x,y}(\sigma) = \Gamma_y(h)-B_{x,y}(\sigma) .
 \end{equation}
Recall the definition of $M_y(h)$ given in~\cref{eq:martingale}.
By Chebyshev's inequality and using~\cref{lem:marting}, we have
\begin{equation}\label{eq:chevy_mart}
    \P\(M_y(h)\le (1-\delta)d^-_y\)\le \frac{\var(M_y(h))}{\delta^2 (d^-_y)^2}\le \frac{1}{\delta^2 \Delta^-} = o(1),
\end{equation}
as $\Delta^-\to \infty$. Moreover, under $\cC_{1}$, we have $\{\Gamma_y(h)=M_y(h)\}$.

Consider the event
\begin{equation}
    \label{eq:MZRM}
    \cC_{2}
    = \{\Gamma_y(h) \ge (1-\delta)\Delta^-\}
    .
\end{equation}
Using \cref{le:coupling} and \cref{eq:chevy_mart}, we have that
\begin{equation}
    \label{eq:B0}
        \P\(\cC_{2}^{c}\)
        \le \P\(\cC_{2}^{c}, \cC_{1}\) + \P\(\cC_{1}^c\)
        \le \P\(M_y(h) < (1-\delta)\Delta^-\) + o(1) =o(1).
\end{equation}

Fix $x\in [n]$. Conditional on $\sigma^-_y$, we now generate $\cG_x(s)$ obtaining the partial pairing $\sigma=\sigma(x,y)$.
Now we argue as in \cref{eq:A_Hoeff} to bound $\overline{B}_{x,y}(\sigma)$.
On the one hand, for any $f\in \cF_{y}(h)$ we have ${\widehat \w}(f)\le 2^{-h} \le n^{-2\gamma}$ and
\begin{equation}\label{eq:sum_max}
\sum_{f\in\cF_{y}(h)} ({\widehat \w}(f))^2\le
\left(\max_{f\in\cF_{y}(h)}  {\widehat \w}(f) \right)
\sum_{f\in\cF_{y}(h)} {\widehat \w}(f) \le  n^{-2\gamma}\Gamma_y(h).
\end{equation}
On the other hand, $\E[\ind{}_{\cC_{1}}\ind{}_{f\notin \cF_{\sigma}}]\le \frac{\kappa_x}{m-\kappa_x-\kappa_y}\le 3n^{-\gamma^2}$. Fix a realization $\sigma_y^-$, let $\ell=|\cF_y(h)|$ and let $f_1,f_2\dots, f_{\ell}$ be an arbitrary ordering of $\cF_y(h)$. Let $(Y_i)_{i\in [\ell]}$ be independent random variables such that $Y_i={\widehat \w}(f_i)$ with probability $3n^{-\gamma^2}$ and $Y_i=0$ with probability $1- 3n^{-\gamma^2}$.
Then, $Y=\sum_{i\in [\ell]} Y_i$ stochastically dominates $\ind{}_{\cC_{1}}\overline{B}_{x,y}(\sigma)$ conditionally on $\sigma_y^-$.
Moreover, the expected value of $Y$ satisfies $\bbE[Y|\sigma_y^-]\le 3n^{-\gamma^2}\Gamma_y(h)= o(\Gamma_y(h))$
and, by \cref{eq:sum_max}, the sum of the squared ranges of the random variables $(Y_i)_{i\in [\ell]}$ is at most $n^{-2\gamma}\Gamma_y(h)$.
Applying Hoeffding's inequality,
\begin{equation}
    \label{eq:B_Hoeff}
    \begin{aligned}
        \P\(\ind_{\cC_{1}} \ind_{\cC_{2}}\overline{B}_{x,y}(\sigma)\ge \delta\Gamma_y(h) \)
        &
        \le
        \E
        \[
        \ind_{\cC_{2}}
        \P\(Y\ge
        \E\left[Y \cond{} \sigma^-_{y}\right]
        +\frac{\delta}{2}\Gamma_y(h) \cond \sigma^-_{y}\)
        \]
        \\
        &
        \le
        \E
        \[
        \ind_{\cC_{2}}
        \exp\(-\frac{\delta^2\Gamma_y(h)}{2 n^{-2\gamma}}\)
        \]
        =
        o(n^{-1})
        .
    \end{aligned}
\end{equation}
Since $\P(\cC_{1})=1-o(1)$, by \cref{eq:B0} and a union bound with \cref{eq:B_Hoeff}, whp
\begin{equation}\label{eq:overB}
\max_{x\in[n]} \overline{B}_{x,y}(\sigma)
=
\max_{x\in [n]} \ind_{\cC_{1}} \ind_{\cC_{2}}
\overline{B}_{x,y}(\sigma)
\le \delta\Gamma_y(h).
\end{equation}
Combining \cref{eq:B0,eq:overB}, \ac{whp} we have,
\begin{equation}
\min_{x\in [n]} B_{x,y}(\sigma) \ge  \Gamma_y(h)-\max_{x\in [n]}  \overline{B}_{x,y}(\sigma)\ge   (1-\delta)\Gamma_y(h) \ge (1-\delta)^2\Delta^-,
\end{equation}
and we conclude that $\P\((\cY^{(2)})^c\)=o(1)$.
\end{proof}

\subsection{Access probabilities to large in-degree vertices}\label{sec:LMindeg}

For any $a\in (0,1)$, define the set
\begin{equation}\label{def-V(alpha)}
    V(a)
=    \left\{
        y \in [n]
        \;:\;
        d^{-}_{y} > n^{a}
    \right\}
    .
\end{equation}
Note that by \cref{lem:s}, for any $a > \frac{1}{2} - \frac{\eta}{6}$, $V(a)=\emptyset$.
\begin{proposition}\label{prop:LB1}
For all $\varepsilon>0$ sufficiently small, $\gamma=\gamma(\varepsilon)\coloneqq\frac{\varepsilon}{80\log (\Delta^+)}$ and $a \in (2\gamma,1)$,
\begin{equation}
  \P\(
         \min_{x\in V_\varepsilon}
            P^{t}(x,y) \ge (1-\varepsilon)\frac{d_{y}^{-}}{m}\,,\;\;\forall y\in V(a)
    \)
    =
    1-o(1)
    ,
\end{equation}
where $t=(1-\gamma)\tent+1$, and $\tent$ and $V_{\varepsilon}$ are defined as
in~\cref{eq:H,def-hslash*}.
\end{proposition}

Throughout \cref{sec:LMindeg} we write $h=h_\varepsilon$ and set
\begin{equation}\label{eq:def-s-t3}
t\coloneqq s+1 , \qquad
s\coloneqq (1-\gamma)\tent.
\end{equation}

Note that $s=\tent- \frac{h}{4\hh}$ as in our previous proofs but this time the overall time $t$ is smaller than the mixing time $\tent$. Fix $x\in [n]$ and $y\in V(a)$.  Generate $\cG_x(s)$ and $\cT_x(s)$ as described in~\cref{suse:out-phase}.  In contrast to the previous section, here we do not generate the in-neighborhood of $y$. Let $\sigma= \sigma^+_x$. Call $\cE_\sigma$ the set of all tails at height $s$ in $\cT_s(x)$, by definition they are all unmatched. Call $\cF_\sigma$ the set of unmatched heads in $E^-_y$. Let $\omega$ be a complete pairing of half-edges chosen uniformly at random among all extensions of $\sigma$.

We need to slightly adjust the notion of nice path in \cref{def:nice}, by letting $t=s+1$. In particular, condition (3) is now void.
Recall that $\wt{P}^{t}(x,y)\le P^{t}(x,y)$ is the quenched probability of
a random walk following a nice path of length $t$ starting at $x$ and ending at $y$.
Conditional on $\sigma$ we have
\begin{equation}
    \wt{P}^{t}(x,y)
    =
    \sum_{e\in \cE_{\sigma}}
    \sum_{f\in \cF_{\sigma}}
    \widehat{\w}(e) \ind_{\widehat{\w}(e)\le n^{-1+2\gamma}} \ind_{\omega(e)=f}
    =
    \sum_{e\in \cE_\sigma}
    \sum_{f\in E^-_y}
    \widehat{\w}(e) \ind_{\widehat{\w}(e)\le n^{-1+2\gamma}} \ind_{f\in \cF_{\sigma}}\ind_{\omega(e)=f}.
\end{equation}
Define,
\begin{equation}\label{IIBC}
    A_{x}(\sigma)
    =
    \sum_{e \in \cE_\sigma} \widehat{\w}(e) \ind_{\widehat{\w}(e)\le n^{-1+\gamma}}
    \quad
    \text{and}
    \quad
    B_{x,y}(\sigma) = \abs{\cF_\sigma}.
\end{equation}
Choose $\delta > 0$ sufficiently small such that $
(1-\delta)^{3} > 1-\varepsilon$. Let
\begin{equation}\label{GJQL}
    \cY_{x,y}\coloneqq
    \left\{
        \sigma
        \,:\,
        A_{x}(\sigma)\ge 1-\delta, B_{x,y}(\sigma)\ge (1-\delta)
        d^{-}_{y}\
    \right\}
    .
\end{equation}
Then, for all $\sigma\in \cY_{x,y}$, we have
\begin{equation}
      \E\left[\wt{P}^{t}(x,y)  \cond  \sigma\right]
\ge  \frac{1}{m} A_{x}(\sigma) B_{x,y}(\sigma)
\ge
\frac{(1-\delta)^2 d^{-}_{y}}{m}
.
\end{equation}
Write
\begin{equation}
\cY\coloneqq\bigcap_ {x\in V_{\varepsilon}, y\in V(a)} \cY_{x,y}
.
\end{equation}
\begin{lemma}\label{lem:typ}
We have $\P\(\cY\)=1-o(1)$.
\end{lemma}

\begin{proof}[Proof of~\cref{prop:LB1}]
Define
\begin{equation}
\cZ_{x,y} \coloneqq \left\{P^{t}(x,y) \ge (1-\varepsilon)\frac{d^{-}_{y}}{m}\right\}.
\end{equation}
Let $\sigma\in \cY_{x,y}$.
By definition, a nice path $\p$ has $\w(\p)\le n^{-1+2\gamma}$.
Therefore, swapping two pairings in $\omega$ which are not fixed by $\sigma$ can change
$\wt{P}^{t}(x,y)$ by at most $\|c\|_\infty  \le   n^{-1+2\gamma}$.

Applying Chatterjee's inequality given in \cref{eq:chatterjee}
and using $y\in V(a)$ with $a>2 \gamma$, we obtain
\begin{equation}\label{MEVJ}
\begin{aligned}
    \P\(
    \cZ_{x,y}^c
    \cond
    \sigma
    \)
    &
    \le
    \P\(
    \wt{P}^{t}(x,y)
    \le
    (1-\delta) \E\left[\wt{P}^{t}(x,y)\cond{} \sigma\right]
    \cond
    \sigma
    \)
    \\
    &
    \le 2\exp\left(-\frac{(1-\delta)^{2}(2+\delta)^{-1}\delta^2 d^{-}_{y}}{2\|c\|_\infty m}\right)
    \\
    &
    \le \exp \left(-\Theta(n^{a-2\gamma})\right)
    =
    o\(n^{-2} \).
\end{aligned}
\end{equation}
We then have
\begin{equation}\label{ZAZD}
    \begin{aligned}
    &
    \P\(
        \cup_{x \in V_{\varepsilon}, y \in V(a)} \cZ_{x,y}^c
    \)
    \le
    \P\(
        \cup_{x \in V_{\varepsilon}, y \in V(a)}
        (\cZ_{x,y}^c\cap \cY_{x,y})
    \)
    +
    \P\(\cY^{c}\)
    =
    o(1)
    ,
    \end{aligned}
\end{equation}
where the first term is bounded using a union bound and~\cref{MEVJ}, and the second term by \cref{lem:typ}.

\end{proof}

\begin{proof}[Proof of~\cref{lem:typ}]

Define the events
\begin{equation}\label{NNMQ}
    \begin{aligned}
        \cY^{(1)} &= \cap_{x\in [n]}\left(\{x \notin V_{\varepsilon}\} \cup \{A_{x}(\sigma)\ge 1-\delta\}\right),\\
        \cY^{(2)} &=\cap_{x\in [n]}\cap_{y \in V(a)} \{B_{x,y}(\sigma)\ge (1-\delta)d^{-}_{y}\},
    \end{aligned}
\end{equation}
and note that
\begin{equation}
    \cY =  \cY^{(1)}\cap \cY^{(2)}.
\end{equation}
Recall that $\wt{q}(x)$ is the probability of not following a nice path,
as defined as in \cref{eq:ptildexy2}.
Since the definition of nice path used in this section is less restrictive, we have $A_x(\sigma)\ge 1-\wt{q}(x)$. \cref{prop:nice} directly implies that $\P\(\cY^{(1)}\)=1-o(1)$.

Let us now show that $\P\(\cY^{(2)}\)=1-o(1)$. Fix $x\in [n]$ and $y \in V(a)$.
Recall that $\sigma$ has paired $\kappa_x\le n^{1-\delta'}$ tails (see \cref{eq:kappax:s}).
For each such tail,
the probability of pairing it to a head in $E^-_{y}$ is uniformly bounded from above by
$q\coloneqq \frac{d^{-}_{y}}{m-\kappa_x}$.
So the number of heads in $E_y^-\setminus \cF_{\sigma}$ is stochastically dominated by a Binomial random variable with parameters $\kappa_x$ and $q$.
Since $\kappa_x q= o(d^{-}_{y})$ and $y\in V(a)$,
Chernoff's inequality (e.g., Corollary 2.4 in~\cite{janson2011a}) implies,
for all $\delta>0$,
\begin{equation}
    \P\(B_{x,y}(\sigma)<(1-\delta)d^-_y\)= \P\(|E^-_y\setminus \cF_{\sigma}| > \delta d^{-}_{y}\) \le
    \exp\(-\delta d^{-}_{y}\) = o(n^{-2}).
\end{equation}
The desired bound follows from a union bound over $x\in [n]$ and $y\in V(a)$.
\end{proof}

\subsection{Lower bounds on stationary values}\label{sec:completeingLB}

The following result implies that lower bound on the access probabilities give lower bounds on the stationary values.
\begin{lemma}\label{lem:access_stationary}
Let $\varepsilon>0$, $K=K_n>0$, $t=t_n\in \N$ and $Y=Y_n\subset [n]$. Suppose the following holds \ac{whp}:
\begin{equation}
P^t(x,y)\ge K, \quad \forall x\in V_\varepsilon, \:  y\in Y.
\end{equation}
Then, the following holds \ac{whp}:
\begin{equation}
\pi(y)\ge (1-\varepsilon)K, \quad \forall y\in Y.
\end{equation}
\end{lemma}
\begin{proof}
We may assume that $\varepsilon$ is sufficiently small with respect to the constant $\eta>0$ appearing in \cref{eq:2+eta}. By applying~\cref{lem:V_*} and~\cref{prop:pi-small-set} with $\delta=1/18$,
we have that \ac{whp}
\begin{equation}\label{eq:pi_*}
\pi_0 \coloneqq \sum_{v\in V_{\varepsilon}} \pi(v)
=
1 -\sum_{v\notin V_{\varepsilon}} \pi(v)
\ge
1- \max_{|S|\le n^{2/3}}  \pi(S) = 1-o(1).
\end{equation}
Using~\cref{eq:pi_*}, we can conclude that \ac{whp}, for every $y\in Y$
\begin{equation}\label{XODF}
\pi(y)
= \sum_{x\in [n]} \pi(x) P^{t}(x,y)
\ge
\sum_{x\in V_{\varepsilon}} \pi(x) P^{t}(x,y)
\ge
\pi_0
\min_{x\in V_{\varepsilon}} P^{t}(x,y)
\ge
(1-\varepsilon)K
.
\end{equation}
\end{proof}

\begin{proof}[Proof of~\cref{eq:lowb} in~\cref{th:main1}]
Again, we may assume that $\varepsilon$ is sufficiently small with respect to the constant $\eta>0$ appearing in \cref{eq:2+eta}.
We apply~\cref{prop:LB2} to $y\in [n]$ with $d^{-}_{y}=\Delta^-$. Let $h=h_\varepsilon$, $\gamma=\frac{\varepsilon}{80\log(\Delta^+)}$,  $s=(1-\gamma)\tent$ and $t=s+h+1$. Then, \ac{whp} and uniformly over $x\in V_\varepsilon$
\begin{equation}
P^{t}(x,y)\ge (1-\varepsilon)\frac{\Delta^-}{m}.
\end{equation}
By \cref{lem:access_stationary} with $Y=\{y\}$, we conclude that whp $\pi_{\max}\ge \pi(y)\ge (1-2\varepsilon)\frac{\Delta^-}{m}$. As $\varepsilon$ can be made arbitrarily small,~\cref{eq:lowb} holds.
\end{proof}

The following is a  direct consequence of \cref{prop:LB1,lem:access_stationary}.
\begin{corollary}\label{coro:lower_bound_large_degs}
Fix $\varepsilon,a>0$ and let $V(a)$ as in \cref{def-V(alpha)}. Then \ac{whp}, for every $y\in V(a)$ we have $\pi(y)\ge (1-\varepsilon)\frac{d^-_y}{m}$.
\end{corollary}

\section{Upper bounds}\label{sec:UB}

This section is devoted to the proof of the upper bound in \cref{th:main1}. The proof is based on the analysis of the \emph{annealed process} introduced in \cref{sec:randomwalk}.  More precisely, we will need to control the high moments
of the random distribution $\mu_t$ defined in \cref{eq:def-lambda2}
for  $t=\log^3(n)$. Thanks to \cref{coro:mixing}, the measure $\mu_t$ is a good approximation of the stationary distribution $\pi$, and this will give the desired result.

In what follows we will consider $\eta\in(0,1)$ satisfying \cref{cond:main},
$\varepsilon\in(0,\eta/6)$, $h=h_\varepsilon$ as in \cref{def-hslash} and $\cG=\cG^+(h)$ as in \cref{eq:defG}.
\begin{lemma}\label{lemmaM}
    For any constant $C>0$, taking $t=\log^3(n)$ and $K=C\log(n)$, one has
    \begin{equation}\label{cond3}
\E\[ \ind_{\cG}\(\mu_t(y)\)^K\]\le \left(\frac{10 K \Delta^-}{n} \right)^K,
    \end{equation}
 for all $y\in[n]$ and all $n$ sufficiently large.
 \end{lemma}
\begin{proof}
    Consider the non-Markovian process of $K$ annealed walks, each of length $t$, starting
    at independent uniformly random vertices. Let $\P^{\mathrm{an}}=\P^{\mathrm{an},K}_{\rm unif}$ denote their joint law as defined by \cref{eq:annx} with $\mu$ uniform over $[n]$.  Call $D_{\ell}$ the digraph generated by the first $\ell$ walks and, for $s\ge 1$, we define  $D_{\ell,s}$ as the union of $D_{\ell-1}$ and the edges generated by $X_0^{(\ell)},\dots,X_{s-1}^{(\ell)}$. Note that $D_{\ell-1}=D_{\ell,1}$. We also write $D_{\ell,0}=D_{\ell-1}$. 

    As usual, with a slight abuse of notation, 
    we identify a digraph $G'\subset G$ with the partial matching of heads and tails that defines it.
    Recall the definition of  $\cP(x,y,s,G')$ in \cref{suse:local}.
    Let $\text{\rm dist}_{G'}(x,y)$ be the length of the shortest path starting at $x$ and ending at
    $y$ in $G'$ and $\cB^{-,G'}_y(h)$ ($\partial\cB^{-,G'}_y(h)$) the set of vertices
    $x\in[n]$ such that $\text{\rm dist}_{G'}(x,y)\le h$ ($\text{\rm dist}_{G'}(x,y)=h$).

    Notice that, by definition,
    \begin{equation}\label{eq:Dk1}
    D_{\ell-1}\subset D_{\ell,s}\subset D_{\ell},\qquad\forall \ell\in\{1,\dots,K \},\: s\in\{1,\dots, t \},
    \end{equation}
    where we used the convention $D_0=\emptyset$. Moreover, for all $\ell\in\{1,\dots,K \}$ we have $|D_{\ell }|\le \ell t$.
    We say that $D_\ell$ is \emph{compatible} with an event $\cA$,
    denoted by $D_\ell\sim \cA$,
    if there exists a realization of the environment $\omega$ that contains $D_\ell$ and such that
    $\omega\in \cA$.
    In particular, $D_\ell \sim \cG$ if all out-neighborhoods of depth $h$ have tree-excess at most
    $1$ in $D_\ell$.
    Fix $y\in[n]$ and let us consider the following events which implicitly depend on $y$: for every
    $\ell\in\{1,\dots, K\}$
    \begin{equation}
        \begin{aligned}
            B_\ell &\coloneqq \{X_t^{(1)}=\cdots=X_t^{(\ell)}=y \},\\
            F_\ell&\coloneqq\{D_\ell\sim \cG \},\\
            H_\ell&\coloneqq\{\forall j\le h,\: |\partial\cB^{-,D_{\ell}}_y(j)|\le 2K \}.
        \end{aligned}
    \end{equation}
    Notice that
    \begin{equation}\label{fin}
    \E[\ind_{\cG}\mu_t(y)^K] \le\P^{\mathrm{an}}(B_K\cap F_K\cap H_K)+\P(H_K^c).
    \end{equation}
    Moreover, letting $B_0=F_0=H_0$ be sure events and using the monotonicity of the events defined above,
    \begin{equation}\label{cond111}
    \P^{\mathrm{an}}(B_K\cap F_K\cap H_K)=\prod_{\ell=1}^K \P^{\mathrm{an}}\left(B_\ell\cap F_\ell\cap H_{\ell} \cond B_{\ell-1}\cap F_{\ell-1}\cap H_{\ell-1}\right).
    \end{equation}
    Therefore, in order to prove \cref{cond3} it suffices to show that
    \begin{equation}\label{cond1}
    \P^{\mathrm{an}}\left(B_\ell\cap F_\ell\cap H_\ell \cond B_{\ell-1}\cap F_{\ell-1}\cap H_{\ell-1} \right)\le\frac{9K \Delta^-}{n} ,\qquad\forall \ell\in\{1,\dots, K\},
    \end{equation}
    and, moreover,
    \begin{equation}\label{cond11}
    \P^{\mathrm{an}}(H_K^c)=o\(\(\frac{\log(n) \Delta^-}{n}\)^K\).
    \end{equation}
    Plugging \cref{cond1,cond11,cond111} into \cref{fin} we obtain \cref{cond3}.

    We start by proving \cref{cond11}. For $j\in\{1,\dots,h \}$, set
    \begin{equation}H_{K,j}=\{|\partial\cB^{-,D_{K}}_y(j)|\le 2K \}, \end{equation}
    so we may write
        \begin{equation}\label{cond5}
        \P^{\mathrm{an}}(H_K^c)= \sum_{j\le h}\P^{\mathrm{an}}(H_{K,j}^c\cap (\cap_{i<j} H_{K,i})).
        \end{equation}
    Notice that $H_{K,1}^c$ is the event
    that the $K$ walks have matched more than $2K$ heads of $y$. Hence, for all $n$ large enough,
    \begin{equation}\label{eq:boundH1}
    \P^{\mathrm{an}}(\hh^c_{K,1})\le \P\(\Bin\(Kt,\frac{\Delta^-}{m-Kt}\)>2K \)\le \(\frac{Kt\Delta^-}{n}\)^{2K} \le \(\frac{\Delta^-\log(n)}{n}\)^{\tfrac{3}{2}K}.
    \end{equation}
For $j\in\{2,\dots, K\}$, under the event $\cap_{i<j}H_{K,i}$, there are at most $2K\Delta^-$ heads that can be matched to violate $H_{K,j}$, and therefore
    \begin{equation}\label{eq:boundHh}
        \begin{aligned}
    \P^{\mathrm{an}}(H_{K,j}^c \cap (\cap_{i<j}H_{K,i}))
    &\le  \P\(\Bin\(Kt,\frac{2K\Delta^-}{m-Kt}\)>2K \) \\
    &\le \(\frac{2K^2 t\Delta^-}{n}\)^{2K}
    \le \(\frac{\Delta^-\log(n)}{n}\)^{\tfrac{3}{2}K}.
        \end{aligned}
    \end{equation}
        Plugging \cref{eq:boundH1} and \cref{eq:boundHh} into \cref{cond5}, we obtain \cref{cond11}.

    We now turn to the proof of \cref{cond1}.
    We fix a realization $D_{\ell-1}$ of the partial matching generated by the first $\ell-1$ walks,
    and assume that it satisfies $B_{\ell-1}\cap F_{\ell-1}\cap H_{\ell-1}$.
    Let $V(D_{\ell,s})$ be the set of vertices previously visited by the other walks or by the
    $\ell$-th walk itself up to time $s-1$ for $\ell \ge 1$, together with $y$.
    For the event $B_\ell$ to occur, the $\ell$-th walk must enter at some time $s\in\{0,\dots,t\}$
    and then traverse only edges in $D_{\ell,s}$ from time $s$ up to time $t$.
    If $s\in\{1,\dots,t\}$,
    the event that the $\ell$-th walk enters in $D_{\ell,s}$ at time $s$ in a given vertex 
    $z\in V(D_{\ell,s})$ has  probability bounded above by $d_z^-/(m-t\ell)\le d_z^-/n$,
    uniformly in the realization of $D_{\ell,s}$.
    On the other hand, the probability that the $\ell$-th walk enters in $D_{\ell-1}$ at time $s=0$ in $z\in D_{\ell-1}$ is $1/n$.

    Given $D_{\ell,s}=D$ and $z\in V(D)$, let $q_D(z,z',r)$ denote the probability that a walk started
    at $z$ at time $0$ arrives in $z'$ at time $r$ by traversing only edges in $ D$.
    Note that $ q_D(z,z',r)=0$ if $z'\notin V(D)$, $q_{D}(z,z,0)=1$, and that
    \begin{equation}
q_D(z,z',r)\le     2^{-r}|\cP\(z,z',r,D\)|,
    \end{equation}
    since the out-degrees are at least $2$. In conclusion, we can bound uniformly in $D_{\ell-1}$
        \begin{equation}\label{eq:boundan2}
         \P^{\mathrm{an}}(B_\ell\cap F_\ell\cap H_\ell\:\big\rvert \:D_{\ell-1}  )
        \le \sum_{s=0}^{t}\max_{D\in \cA_\ell(s)} \sum_{z\in V(D)}\frac{d^-_z}{n}\,q_D(z,y,t-s),
        \end{equation}
with the convention that $d_z^-$ must be replaced by $\max\{d^-_z,1\}$ when $s=0$, and where we define the set $\cA_\ell(s)$ of all possible realizations $D$ of $D_{\ell,s}$ such that
    \begin{equation}\label{eq:D}
\P^{\rm an}(D_{\ell,s}=D,\:B_{\ell}\cap F_\ell\cap H_\ell\ | \ D_{\ell-1})>0.
\end{equation}

    We split the interval $s\in\{0,\dots,t\}$
    into two parts:
    \begin{equation}I_1=\{0,\dots, t-h\},\qquad I_2=\{ t-h+1,\dots, t\}.\end{equation}
    Observe that if $s\in  I_1$, for every realization $D$ of $D_{\ell,s}$:
    \begin{equation}\label{eq:qdest}
q_D(z,y,t-s) = \sum_{z'\in V(D)} q_D(z,z',t-s-h)q_D(z',y,h)\le \max_{z'\in V(D)}2^{-h}|\cP(z',y,h,D)|.
    \end{equation}
Thanks to the event $F_\ell$,
we know that $|\cP(z',y,h,D)|\le 2$ for all $z'\in V(D)$,
and $D\in\cA_\ell(s)$.
Moreover, $|V(D)|\le (1+K t)$, and therefore
    \begin{equation}\label{eq:qzzz}
        \begin{aligned}
            &\sum_{s\in I_1}\max_{D\in \cA_\ell(s)} \sum_{z\in V(D)}\frac{d^-_z}{n}\,q_D(z,y,t-s) \\ 
            &\qquad  \le\sum_{s\in I_1}\max_{D\in \cA_\ell(s)}\frac{|V(D)|\Delta^-}{n}2^{-h+1}
            \le |I_1|
            \frac{(1+K t)\Delta^-}{n}\,2^{-h+1}=o\(\frac{\Delta^-}{n}\).
        \end{aligned}
    \end{equation}

    We now turn to bound the sum in \cref{eq:boundan2} for $s\in I_2$. Notice that
    \begin{equation}\label{eq:sumzvd}
        \sum_{s\in I_2}\max_{D\in\cA_\ell(s)}\sum_{z\in V(D)}\frac{d^-_z}{n}q_D(z,y,t-s) =
        \sum_{a=0}^{h-1}\max_{D\in\cA_\ell(t-a)} \sum_{j=0}^{a}
        \sum_{z\in\partial \cB_y^{-,D}(j)}
        \frac{d^-_z}{n}\,q_D(z,y,a).
    \end{equation}
    As the event $F_\ell$ holds, if $D\in\cA_\ell( t-a)$,  if $a\leq h$, we may estimate
    \begin{equation}\label{cond44}
        q_D(z,y,a)\le 2^{-a}|\cP(z,y,a,D)|\le 2^{-a+1}.
    \end{equation}
    Therefore, using that $H_\ell$ holds and \cref{cond44}, for all $D\in\cA_\ell(t-a)$, it follows that
    \begin{equation}\label{cond4}
        \sum_{j=0}^{a}
        \sum_{z\in\partial \cB_y^{-,D}(j)}\frac{d^-_z}{n}\,q_D(z,y,a)\le \frac{\Delta^-}{n}(2Ka+1)2^{-a+1}.
    \end{equation}
    Using \cref{cond4} we see that \cref{eq:sumzvd} is bounded above by
    $8\frac{\Delta^-}{n}K$.
    Recalling \cref{eq:boundan2,eq:qzzz}, we conclude the
    validity of \cref{cond1}.
\end{proof}

\begin{proof}[Proof of the upper bound in \cref{th:main1}] We show how the desired bound follows by the moment estimates in \cref{lemmaM} and \cref{coro:mixing}. Fix $C>0$ and consider
    \begin{equation}\label{617}
        \begin{aligned}
    \P\(\max_{y\in [n]}\mu_t(y)> \frac{20\: C\log(n) \Delta^-}{n}\)&\le\P\(\max_{y\in [n]}\mu_t(y)\ind_{\cG}> \frac{20\: C\log(n) \Delta^-}{n}\)+\P(\cG^c)\\
    &\le n\max_{y\in [n]}\P\(\mu_t(y)\ind_{\cG}> \frac{20\: C\log(n) \Delta^-}{n}\)+o(1),
        \end{aligned}
    \end{equation}
    where the second inequality follows by a union bound over $y\in[n]$ and by \cref{lemmaG}. By Markov's inequality, for $K=C\log(n)$, and using \cref{lemmaM},
    \begin{equation}\label{618}
    \P\(\mu_t(y)\ind_{\cG}> \frac{20 K \Delta^-}{n}\)\le
 \frac{\E[\ind_{\cG}\mu_t(y)^K]}{\(\frac{20 K \Delta^-}{n}\)^K}\le 2^{-K}.
    \end{equation}
    Choosing $C>1/\log(2)$ we conclude that the probability in \cref{618} is $o(n^{-1})$.
    Since we may take $C$ such that  $20C<29$,
    plugging \cref{618} into \cref{617} we obtain, \ac{whp}
    \begin{equation}\label{eq:finn}
        \max_{y\in[n]} \mu_t(y)< \frac{29 \log(n) \Delta^-}{n}.
    \end{equation}

    Define the event $\cE=\{\forall y\in[n],\; 
	\pi(y) \leq \mu_t(y)+ 2\ee^{-\log^{3/2}(n)}    
    \}$,
    with $t=\log^3(n)$,  and observe that
    \begin{equation}
        \begin{aligned}
            \P\(n\pi_{\max}> 30 \log(n)\Delta^- \)&\le \P\(n\pi_{\max}>  30\log(n)\Delta^-;\:\cE\)+\P(\cE^c)
                                                \\&\le \P\(\max_{y\in[n]} \mu_t(y)>  \frac{30\log(n)\Delta^-}{n}-2\ee^{-\log^{3/2}(n)} \) + \P(\cE^c)\\
            &\le  \P\(\max_{y\in[n]} \mu_t(y)>\frac{29\log(n)\Delta^-}{n}\) + \P(\cE^c).
        \end{aligned}
    \end{equation}
    \cref{coro:mixing} implies $\P(\cE^c)=o(1)$ and the desired conclusion follows from \cref{eq:finn}.
\end{proof}

\section{Power-law behavior: Proof of \texorpdfstring{\cref{th:main2}}{Theorem~\ref{th:main2}}}\label{sec:PL}

In this section we prove \cref{th:main2}.
Recall the definition of $\phi(k)$ in \cref{nt} as the proportion of vertices having in-degree $k$.
The lower bound in \cref{hppsi1} is an immediate corollary of the results in \cref{sec:completeingLB}.
Indeed, from \cref{coro:lower_bound_large_degs} 
it follows that \ac{whp} for all $y\in[n]$ such that $d_y^->n^a$,
one has {$n\pi(y)>\frac{n^a}{2\pl d\pr}$,
where $\pl d\pr$ is the average degree}.
Therefore, for all $a\in(0,1/\kappa)$,
\begin{equation}
\psi(n^a,\infty)\geq \phi(2\pl d\pr n^a,\infty)\geq n^{-a\kappa -\varepsilon}\,,
\end{equation}
for all $\varepsilon >0$ by the assumed power-law behavior of the degree sequence.

The rest of this section is concerned with the proof of the upper bound in \cref{th:main2}. As announced in \cref{rem:general_ubtail}, we actually prove a slightly more general result; see  \cref{th:klight} below.

\begin{definition}\label{def:klight}
A bi-degree sequence is \emph{$\kappa$-light} if for all $\varepsilon>0$, for all $a>0$ we have
            \begin{equation}\label{eq:size_biased_tail}
            \sum_{k>n^a} k\phi(k)\le n^{-a(\kappa-1)+\varepsilon}.
            \end{equation}
\end{definition}
\begin{proposition}\label{pr:hp4}
    If $\bfd_n$ satisfies~\cref{cond:main}, then $\bfd_n$ is $(2+\eta)$-light, where $\eta>0$ is such that \cref{eq:2+eta} holds.
    Moreover, if $\bfd_n$ has power-law behavior with {index} $\kappa>2$ as in \cref{hp},  then $\bfd_n$ is $\kappa$-light.
\end{proposition}
\begin{proof}
For any $a>0$ and using the bounded $(2+\eta)$-moment we have
        \begin{equation}
        \sum_{k>n^a} k \phi(k) \le  n^{-a(1+\eta)} \sum_{k>n^a}k^{2+\eta}\phi(k)= O(n^{-a(1+\eta)}).
        \end{equation}
This proves the first assertion.

To prove the second one, for any $\varepsilon>0$ and using~\cref{hp}, we have
        \begin{equation}
        \sum_{k\ge \floor{n^a}} k \phi(k) =
        \sum_{r=0}^{\ceil{1/\varepsilon}} \sum_{k= \floor{n^{a+r\varepsilon}}}^ {\ceil{n^{a+(r+1)\varepsilon}}} k \phi(k)
        \le \sum_{r=0}^{\ceil{1/\varepsilon}} \ceil{n^{a+(r+1)\varepsilon}} n^{-(a+r\varepsilon)\kappa+\varepsilon} =O(n^{-a(\kappa-1)+2\varepsilon}).
        \end{equation}
Since $\varepsilon$ can be arbitrarily small, this implies the claim.

\end{proof}
From the previous facts, we see that the upper bound in \cref{th:main2} and the claim in \cref{rem:general_ubtail} both follow from the next result.

\begin{theorem}\label{th:klight}
    Suppose $\bfd_n$ satisfies \cref{cond:main} and assume that $\bfd_n$ is $\kappa$-light for some $\kappa>2$.
    Then for all $\varepsilon>0$, \ac{whp} for all $a>0$,
    \begin{equation}\label{eq:psiklight}
        \psi(n^a,\infty)\le n^{-a\kappa+\varepsilon}.
    \end{equation}
\end{theorem}

\subsection{The \texorpdfstring{$a$}{a}-skeleton}\label{suse:skeleton}
The proof of \cref{th:klight} is based on the following construction.
For any $a\in(0,1/\kappa)$,
the \emph{$a$-skeleton} $\xi_a$  is the partial matching of heads and tails defined as follows.
Recall the definition of $V(a)$ in \cref{def-V(alpha)}.
For any $z\in V(a)$, 
call $W_{a,z}$ the set of paths $\p$ starting at $z$ whose weight $\w(\p)$ is at least 
\begin{equation}
    \w_{a,z} \coloneqq \frac{n^a}{d_{z}^-}.
\end{equation}
Note that we have not fixed the length of the paths in this definition, and that $W_{a,z}$ covers a portion of the out-neighborhood of $z$ with depth growing logarithmically as a function of
$d_z^-$.
The $a$-skeleton is defined by
\begin{equation}\xi_a=\bigcup_{z\in V(a)}W_{a,z}.\end{equation}
We call $V(\xi_a)$ the set of $y\in[n]$ such that at least one of the heads or tails of $y$ is matched in $\xi_a$.
In particular, $V(a)\subset V(\xi_a)$.

\begin{lemma}\label{lem:sigma}
    Under the assumptions of \cref{th:klight}, for all $a\in(0,1/\kappa)$, $\varepsilon>0$,
    $\xi_a$ has at most  $n^{1-a\kappa+\varepsilon}$ edges for $n$ large enough. In particular,
    \begin{equation}\label{eq:vsia}
        |V(\xi_a)|\le n^{1-a\kappa+\varepsilon}.
    \end{equation}
\end{lemma}
\begin{proof}
    Fix some $z\in V(a)$ and notice that the set $W_{a,z}$ can be generated with the weighted out-neighborhood construction given in \cref{suse:out-phase},
    with the only difference that here we impose no constraint on the distance to the root $z$,
    and the minimal weight $\w_{\min}$ is now replaced by $\w_{a,z}$.
    Thus, it follows from the argument used in \cref{eq:kappax:s},
    that the number of edges in $W_{a,z}$ is at most $2/\w_{a,z}=2 n^{-a}d_z^-$.
    Since the degree sequence is $\kappa$-light, the total number of edges in $\xi_a$ is deterministically bounded by
    \begin{equation}
        \sum_{z\in V(a)}2 n^{-a}d_z^-
        \le 2n^{-a}\sum_{k>n^a} k n \phi(k)\le n^{1-a\kappa+\varepsilon}.
    \end{equation}
    This finishes the proof.
\end{proof}
\begin{remark}\label{rem:lowerb}
Let us observe that \ac{whp} if  $y\in V(\xi_a)$ then $n\pi(y)\ge \frac{n^{a}}{2\pl d\pr}$.
Indeed, if $y\in V(a)$ then this is a consequence of \cref{coro:lower_bound_large_degs}.
If instead $y\in V(\xi_a)\setminus V(a)$,
then by definition of $a$-skeleton there exists a vertex $z\in V(a)$ and a path $\p$ starting at $z$
and ending at $y$ with $\w(\p)\ge \w_{a,z}=n^a/d_z^-$.
Thus, if $s$ denotes the length of $\p$, then
\begin{equation}
    n\pi(y)\ge n\pi(z)P^s(z,y)\ge n\pi(z)\w_{a,z} 
    \ge 
    n \frac{d_{z}^{-}}{2 m}\frac{n^{a}}{d_{z}^{-}} =
    \frac{n^{a}}{2\pl d\pr},
\end{equation}
where we used again \cref{coro:lower_bound_large_degs} to lower bound $\pi(z)$. The heart of the proof of   \cref{th:klight} will consist in showing that for all $y\notin V(\xi_a)$ except for at most
$n^{o(1)}$ of them one has $n\pi(y)\le n^{a+o(1)}$.
\end{remark}

\subsection{Proof of \texorpdfstring{\cref{th:klight}}{Theorem~\ref{th:klight}}}\label{sec:proofofub}

Recall that
\begin{equation}\label{def:psia}
\psi(n^a,\infty)=\frac1n\sum_{y\in[n]}\ind(\pi(y)>n^{a-1}).
\end{equation}
We need to prove that for any $\varepsilon>0$,
\begin{equation}\label{claim:psia}
\bbP\(\exists a\in(0,\infty):\; \psi(n^a,\infty)>n^{-a\kappa +\varepsilon}\)=o(1).
\end{equation}
Let $t=\log^3(n)$ and consider $\mu_t$ as in \cref{eq:def-lambda2}. Define
\begin{equation}\label{def:Za}
Z_a=\sum_{y\in[n]}\ind(\mu_t(y)>\tfrac12\,n^{a-1}).
\end{equation}
The next lemma is the main technical estimate in this section.
\begin{lemma}\label{lem:zlemma}
For any $\varepsilon>0$, for any fixed $a\in(0,1/\kappa)$,
\begin{equation}\label{claim:psiaoz}
\bbP\(Z_a>n^{1-a\kappa +\varepsilon}\)=o(1).
\end{equation}
\end{lemma}
Before proving \cref{lem:zlemma}, let us show that this estimate implies \cref{claim:psia}.
We start by proving that
\begin{equation}\label{claim:psiamax}
\bbP\(\psi(n^{a},\infty)>0\)=o(1),
\end{equation}
for all $a>1/\kappa$.
{The maximum in-degree $\Delta^-$ is the largest $\ell$ such that $\phi(\ell)\ge 1/n$, or equivalently, such that $n\sum_{k\ge \ell} k\phi(k)\ge \ell$. If $\Delta^-> n^{\frac{1}{\kappa}}$, by definition of $\kappa$-light degree sequence, for all $\varepsilon>0$ we have}
\begin{equation}
\Delta^-\le n\sum_{k>n^{1/\kappa}} k\phi(k)\le n^{\frac{1}{\kappa} + \varepsilon}.
\end{equation}
By the upper bound in \cref{th:main1}, we have \ac{whp} $n\pi_{\max}\le \Delta^-n^\varepsilon\le n^{\frac{1}{\kappa} + 2\varepsilon}$, and therefore $\psi(n^{a},\infty)=0$ if $a\ge\frac{1}{\kappa} + 2\varepsilon$.  Since $\varepsilon$ is arbitrarily small, this proves \cref{claim:psiamax} for all $a>1/\kappa$.

To prove \cref{claim:psia}, thanks to the monotonicity of $a\mapsto \psi(n^a,\infty)$ and  \cref{claim:psiamax} one can replace $\exists a\in(0,\infty)$ with $\exists a\in(0,b)$ for any fixed $b>1/\kappa$ in that statement. For simplicity, we take $b=2/\kappa$.
Moreover, since we have restricted to a bounded interval of exponents $a$,  it is actually sufficient to prove that  for any $\varepsilon>0$
\begin{equation}\label{claim:psiao}
\bbP\(\psi(n^a,\infty)>n^{-a\kappa +\varepsilon}\)=o(1),
\end{equation}
for each fixed $a\in(0,2/\kappa)$. Indeed, let $I_\varepsilon$ denote the set of integers in the interval $[0,4/\varepsilon]$. If $a\in(0,2/\kappa)$, there exists $i\in I_\varepsilon$ such that $i\varepsilon/2 \le a\kappa\le (i+1)\varepsilon/2$ and
$\psi(n^{i\varepsilon/(2\kappa)},\infty)\ge \psi(n^a,\infty)$. Therefore, if $\psi(n^a,\infty)>n^{-a\kappa +\varepsilon}$ for some $a\in(0,2/\kappa)$, there must exists $i\in I_\varepsilon$ such that \begin{equation}\psi(n^{i\varepsilon/(2\kappa)},\infty)>n^{-a\kappa +\varepsilon}\ge n^{-i(\varepsilon/2)+\varepsilon/2 }.\end{equation}
Using \cref{claim:psiao},  a union bound over the finite set $I_\varepsilon$ then allows us to conclude   \cref{claim:psia}.

Finally we observe that it is sufficient to establish \cref{claim:psiao} for all fixed $a\in(0,1/\kappa)$.
Indeed, the case $a>1/\kappa$ is covered by \cref{claim:psiamax},
and the case $a=1/\kappa$ follows by the arbitrariness of $\varepsilon$.

By \cref{coro:mixing}, 
we know that \ac{whp} $\mu_t(y)\ge \pi(y) - 1/n$ for all $y\in[n]$,
which implies $n\psi(n^a,\infty)\le Z_a$ for all $a>0$.
Therefore, in order to prove \cref{claim:psiao} it suffices to show \cref{lem:zlemma}.

\subsection{Proof of \texorpdfstring{\cref{lem:zlemma}}{Lemma~\ref{lem:zlemma}}}

To prove \cref{claim:psiaoz}, we fix $\eta\in(0,1)$ such that \cref{cond:main} applies, and define $h=c\log n$,
where $c=c(\eta,\Delta^+)$ can be taken, e.g., as $c=\eta/(200\log\Delta^+)$. As $h\le
2h_{\varepsilon}$, with, e.g., $\varepsilon=\eta/4$, by \cref{lemmaG} the event $\cG=\cG^+(h)$ has probability $1-o(1)$.
The key to our proof will be the following estimate on the moments of $\mu_t(y)$ for $y\notin\xi_a$.
\begin{lemma}\label{le:mom-indind}
    Fix $t=\log^3(n)$. For all constants $\varepsilon>0$,  $a\in(0,1/\kappa)$, $K>1$,
        \begin{equation}\label{expmarkov}
        \max_{y\in[n]}\E\[\ind_{\cG}\ind_{y\notin V(\xi_a)} \mu_t(y)^K\] \le n^{K(a+\varepsilon-1)},
        \end{equation}
        for all $n$ large enough.
\end{lemma}
\begin{proof}
Fix $a\in(0,1/\kappa)$ and $y \in [n]$.
We may assume that $y\notin V(a)$, 
since otherwise $y\in V(\xi_a)$ and the estimate becomes trivially satisfied.

We use the same construction based on the $K$ annealed walks as in \cref{lemmaM}.
With the notation used in that proof,
recall that $D_\ell$ is the digraph generated by the first $\ell$ trajectories 
and $D_{\ell,s}$ is the union of $D_{\ell-1}$ 
and the edges generated by the $\ell$-th walk up to time $s-1$.
For any $z\in V(D_\ell)\cap V(a)$,
define $W_{a,z}(D_\ell)$ as the union of all paths $\p$ contained in $D_\ell$,
starting at $z$, and such that $\w(\p)\ge \w_{a,z}$.
Recall the definition of an event being \emph{$D_\ell$ compatible} given in \cref{sec:UB}.
In particular, $D_\ell\sim \{y\notin V(\xi_a)\}$,
if $y\notin V(W_{a,z}(D_\ell))$ for all $z\in V(D_\ell)\cap V(a)$.
For $\ell\in\{1,\dots, K\}$, we consider the events
\begin{equation}E_\ell=\{D_\ell\sim \{y\notin V(\xi_a)\}\}, \quad
F_\ell=\{D_\ell\sim \cG \},
\quad
B_\ell = \{X_t^{(1)}=\cdots=X_t^{(\ell)}=y \}.
\end{equation}
We may write
\begin{equation}
    \begin{aligned}
    \E[\ind_{\cG}\ind_{y\not\in\xi_a}\mu_t(y)^K]        &\le\P^{\mathrm{an}}(B_K\cap E_K\cap F_K)\\
\label{need3}   &=\P^{\mathrm{an}}(B_1\cap E_1\cap F_1)\prod_{\ell=2}^K \P^{\mathrm{an}}(B_\ell\cap E_\ell \cap F_\ell\big\rvert B_{\ell-1}\cap E_{\ell-1}\cap F_{\ell-1}).
    \end{aligned}
\end{equation}

Bounding the first term is simple.
Indeed, the event that the first walk visits $y$ up to time $t$ has probability bounded by
$\frac{1}{n} + \frac{t d_y^-}{n}\le n^{a+\varepsilon-1}$,
since it has probability $1/n$ of hitting $y$ at time $0$ 
and probability at most $\frac{d_y^-}{m-t} \le \frac{d_y^-}{n} \le n^{a-1}$ 
of visiting $y$ for the first time at any subsequent step.
It follows that
\begin{equation}
    \P(B_1\cap E_1\cap F_1)\le\P(B_1)\le n^{a+\varepsilon-1}.
\end{equation}
Hence, it is suffices to show that
\begin{equation}\label{need2}
\P^{\mathrm{an}}(B_\ell\cap E_{\ell} \cap F_\ell\big\rvert B_{\ell-1}\cap E_{\ell-1}\cap F_{\ell-1} )\le n^{a+\varepsilon-1},\qquad\forall \ell\in\{2,\dots, K\}.
\end{equation}

Let $D_{\ell-1}$ denote a realization of the partial matching generated by the first $\ell-1$ walks,
and assume that $D_{\ell-1}$ satisfies $B_{\ell-1}\cap F_{\ell-1}\cap E_{\ell-1}$. 
For the event $B_\ell$ to occur, 
the $\ell$-th walk must enter at some time $s\in\{0,\dots,t\}$ in $D_{\ell,s}$.
Arguing exactly as in \cref{eq:boundan2} we estimate
\begin{equation}\label{eq:boundan}
\P^{\mathrm{an}}(B_\ell\cap F_\ell\cap E_\ell\:\big\rvert \:D_{\ell-1}  )\le\sum_{s=0}^{t}\max_{D\in \cA_\ell(s)} \sum_{z\in V(D)}\frac{d^-_z}{n}q_D(z,y,t-s),
\end{equation}
where we define the set $\cA_\ell(s)$ as in \cref{eq:D} with the only difference that the event $H_\ell$ is replaced by $E_\ell$.

We split the last sum in \cref{eq:boundan} according to whether $z$ is in $V(a)$.
If $z\notin V(a)$, for all $s\le t$, one has
\begin{equation}
    \sum_{z\in V(D)\setminus V(a)}\frac{d^-_z}{n}q_D(z,y,t-s)\le n^{a-1}|V(D)|.
\end{equation}
Since $|V(D)|\le Kt+1$ and $t=\log^3(n)$, 
the contribution of this term to the \ac{rhs} of \cref{eq:boundan} 
is at most $(t+1)(K t +1) n^{a-1}\le n^{a-1+\varepsilon}$.
Therefore, we may restrict to estimating the contribution of the terms corresponding to $z\in V(a)$.

We now show that, for every $D\sim \cG$ 
it is unlikely that a walk stays on $D$ for $3\log(n)$ steps.
First observe that
\begin{equation}\label{lemmaneew1}
    \max_{z\in V(D)}\sum_{x\in V(D)}q_D(z,x,h)
    \le 2|V(D)|2^{-h}
    \le n^{-\frac{c}{2}},
\end{equation}
where the first inequality is a consequence of $D\sim \cG$, as in \cref{eq:qdest}, 
and the second inequality follows from $|V(D)|\le Kt+1=n^{o(1)}$ and $h=c\log n$.
By the Markov property, for all $t\in \bbN$
\begin{equation}\label{lemmaneew2}
    \max_{z\in V(D)}\sum_{x\in V(D)}q_D(z,x,t) \le n^{-\floor{\frac{t}{h}}\frac{c}{2}}.
\end{equation}
Therefore, if $j\ge 3\log(n)$,
\begin{equation}\label{lemmanew1}
    \max_{z\in V(D)} q_D(z,y,j) 
    \le \max_{z\in V(D)}\sum_{x\in V(D)} q_D(z,x,j)
    \le n^{-\floor{3/c}c/2}
    \le n^{-1}.
\end{equation}
By \cref{lemmanew1},
\begin{equation}\label{lemmanew2}
    \sum_{s=0}^{t-3\log n}\sum_{z\in V(D)}\frac{d^-_z}{n}q_D(z,y,t-s)
    \le \frac{t}n\sum_{z\in V(D)}\frac{d^-_z}{n} 
    \le  \frac{\pl d \pr t}n
    \le n^{a+\varepsilon-1}.
\end{equation}

Hence we can restrict to the case $z\in V(a)$ and $s\in\{t-3\log(n), \dots,t\}$ in \cref{eq:boundan}.
If the event $E_\ell$ holds,
and the entry vertex $z$ of $D_{\ell,s}=D$ is in $V(a)$,
then  each path from $z$ to $y$ in $D$ has weight smaller than $\w_{a,z}$.
Therefore,
\begin{equation}\label{lemmaneew3}
    q_D(z,y,t-s)\le \w_{a,z}|\cP(z,y,t-s,D)|.
\end{equation}
Since $\frac{d^-_z}{n}\w_{a,z}{=} n^{a-1}$, it remains to show that under the event $D\sim \cG$ and for all fixed $\varepsilon>0$,
\begin{equation}\label{claim}
    \max_{z\in D}\max_{j\le 3\log(n)}|\cP(z,y,j,D)| \le n^\varepsilon,
\end{equation}
for all $n$ large enough.

Let $L = \floor{\frac{j}{h}}$.
To prove \cref{claim}, we split each path of length $j$ in $D$ from any $z$ to $y$
into $L+1$ consecutive paths of which the first $L$ have length $h$ 
and the last one has length $h' \le h$.
Note that
\begin{equation}\label{eq:def-Lh}
    L \le \frac{j}{h} \le \frac{3}{c}.
\end{equation}
Letting $u_{i}$ denote the end vertex of the $i$-th sub-path,
we have
\begin{equation}
    |\cP(z,y,j,D)|= \sum_{u_1\in V(D)}\cdots\sum_{u_{L}\in V(D)}|\cP(z,u_1,h,D)|\:|\cP(u_1,u_2,h,D)|\cdots|\cP(u_{L},y,h',D)|.
\end{equation}
Since $D\sim\cG$, for every $u,u'\in V(D)$ the number of paths of length $h'\le h$ from $u$ to $u'$ in $D$ is bounded by $2$. Thus,
\begin{equation}\label{eq:est-C-2}
    |\cP(z,y,j,D)|\le \sum_{u_1\in V(D)}\cdots\sum_{u_{L}\in V(D)}2^{L+1}
    \le 2\(2 |V(D)| \)^{L} 
    \le 2\(2 (K t +1) \)^{L} 
    \le n^{\varepsilon},
\end{equation}
for any $\varepsilon>0$, if $n$ is large enough.
\end{proof}

\begin{proposition}\label{prop:EZa}
    Fix any $a\in(0,1/\kappa)$ and consider $Z_a$ as in \cref{def:Za}. Then, for all $\varepsilon>0$,
    \begin{equation}\label{expmarkovZ}
    \E\[\ind_{\cG}Z_a\] \le n^{1-a\kappa+\varepsilon},
    \end{equation}
    for all $n$ large enough.
\end{proposition}
\begin{proof}
    Fix $a\in(0,1/\kappa)$, $\varepsilon\in(0,a/4)$, and $\delta=\varepsilon/\kappa$.
    Let $V_k$ denote the set of $y\in[n]$ with $d_y^-=k$.
    Let $\cE(y) = \{\ind_{\cG}\mu_t(y)>\tfrac12 n^{a-1}\}$.
    We have
    \begin{equation}
    \E[\ind_{\cG}Z_a] =
    \sum_{y\in[n]} \P\(\cE(y) \)
    = \sum_{k\ge 0}\sum_{y\in V_k}\P\(\cE(y)\).
    \end{equation}
    First note that
    \begin{equation}
        \sum_{k>n^{a-\delta}}\sum_{y\in V_k}\P\(\cE(y)\)\le \sum_{k>n^{a-\delta}}n\phi(k)\le n^{\delta-a+1}\sum_{k>n^{a-\delta}}k\phi(k)\le n^{-a\kappa + 1 +2\varepsilon},
    \end{equation}
    where the last bound follows from the assumption that the degree sequence is $\kappa$-light.
    Thus, in the rest of the proof we restrict to $k\le n^{a-\delta}$, 
    i.e., $y \notin V(a-\delta)$.

    For all $y\in V$, let
    \begin{equation}
        \label{eq:LJAQ}
        \cE_1(y)=\cE(y) \cap \{y\in V(\xi_{a-\delta})\},
        \qquad
        \cE_2(y)=\cE(y) \cap \{y\notin V(\xi_{a-\delta})\}.
    \end{equation}
    Therefore, by the arbitrariness of $\varepsilon$, the desired statement follows if we prove
    \begin{equation}\label{eq:E1}
    \sum_{k\le  n^{a-\delta}}\sum_{y\in V_k}\P(\cE_i(y))\le    n^{-a\kappa + 1 +3\varepsilon},\qquad i=1,2.
    \end{equation}

    To prove \cref{eq:E1} for $i=1$, we use the rough bound $\P(\cE_1(y))\le \P(y\in V(\xi_{a-\delta}))$. Since $y\notin V(a-\delta)$, in the generation of $\xi_{a-\delta}$ at least one head incident to $y$ has been matched.
    By \cref{lem:sigma} and the choice of $\delta=\varepsilon/\kappa$, we know that $\xi_{a-\delta}$ contains at most $n^{1-a\kappa+2\varepsilon}$ edges. Thus, the probability that during the generation of $\xi_{a-\delta}$ one of the heads of a given $y\notin V(a-\delta)$ gets matched is bounded by the probability that a binomial random variable with parameters $N=n^{1-a\kappa+2\varepsilon}$ and $p=d_y^-/(m-N)\le d_y^-/n$ is positive. Thus
    \begin{equation}\label{eq:E1-1}
        \P(y\in V(\xi_{a-\delta})) \le N p \le   d_y^- n^{-a\kappa+2\varepsilon}.
    \end{equation}
    Summing over $k \le n^{a-\delta}$ and $y\in V_k$, 
    we obtain \cref{eq:E1} for $i=1$.

    We actually prove a stronger estimate than \cref{eq:E1} for $i=2$.
    Indeed, for every $K>0$:
    \begin{equation}\label{eq:E2}
        \P\(\cE_2(y) \)=    \P\(\ind_{\cG}\ind_{ y\not\in\xi_{a-\delta}}\mu_t(y) >\tfrac12 n^{a-1} \)\le \frac{2^K\E\[\ind_{\cG} \ind_{y\not\in\xi_{a-\delta}} \mu_t(y)^K  \]}{n^{K(a-1)}}.
    \end{equation}
    By \cref{le:mom-indind}, for all $\varepsilon'>0$, taking $n$ large enough,
    \begin{equation}
        \E\[\ind_{\cG} \ind_{y\not\in\xi_{a-\delta}} \mu_t(y)^K  \]\le n^{K(a-\delta+\varepsilon'-1)}.
    \end{equation}
    Choosing $\varepsilon'=\delta/3$, the right hand side of \cref{eq:E2}  can be bounded by $n^{-\delta K/2}$. Therefore
    \begin{equation}\label{eq:E2-final}
        \sum_{k\le  n^{a-\delta}}\sum_{y\in V_k}\P(\cE_2(y))
        \le n^{1-\delta K/2}.
    \end{equation}
    The desired estimate follows by choosing, e.g., $K=\ceil{\frac{2}{\delta}}$.
\end{proof}

We are now able to conclude the proof of \cref{lem:zlemma}. 
Recall that all we needed is the estimate \cref{claim:psiaoz}.
We write
\begin{equation}
    \P\(Z_a > n^{1-\kappa a+\varepsilon} \)\le \P\(\ind_{\cG}Z_a> n^{1-\kappa a+\varepsilon} \)+\P(\cG^c).
\end{equation}
By \cref{lemmaG}, $\P(\cG^c)=o(1)$.
By \cref{prop:EZa},
$\E[\ind_{\cG}Z_a]\le n^{1-\kappa a+\varepsilon/2}$ 
and therefore \cref{claim:psiaoz} is a consequence of Markov's inequality.

\section{Power-law for PageRank: Proof of \texorpdfstring{\cref{lbtailPR}}{Theorem~\ref{lbtailPR}}}\label{sec:pagerank}

\subsection{Lower bound}
Let us take $\alpha=\alpha_n$ a sequence in $(0,1)$ 
and assume that 
\begin{equation}
\limsup_{n\to\infty}\alpha_n\le 1-\delta,
\end{equation} for some $\delta>0$.
For $x\in [n]$, it follows from \cref{eq:pi-PR} that
\begin{equation}
    \begin{aligned}
        \pi_{\alpha,\lambda}(x)
        &=\sum_{k=0}^\infty \alpha(1-\alpha)^k\lambda P^k(x)\\
        &\ge \alpha(1-\alpha)\lambda_{\min}\sum_{y\in[n]}P(y,x) \\
        &\ge \alpha(1-\alpha)\lambda_{\min}\,\frac{d_x^-}{\Delta^+}\ge \alpha\,\frac{\delta n^{-\varepsilon}}{2\Delta^+}\,\frac{d_x^-}{n}\,,\label{eq:lowerboPR}
    \end{aligned}
\end{equation}
where $\lambda_{\min}=\min_{y\in[n]}\lambda(y)$, and we have used $1-\alpha\ge \delta/2$,  $\lambda_{\min} \ge n^{-1-\varepsilon}$ by \cref{hp:lambda},
for all $\varepsilon>0$, and $n$ large enough.

On the other hand, by the definition of total variation distance (see \cref{eq:def-TV}) 
and the monotonicity of distance to equilibrium,
for all $x \in [n]$ and $k \ge t \in \N$,
\begin{equation}
    \label{eq:EWSA}
    \lambda P^{k}(x) 
    \ge
    \pi(x)-\|\lambda P^k - \pi\|_\tv
    \ge
    \pi(x)-\|\lambda P^t - \pi\|_\tv.
\end{equation}
Thus, for any $t\in\bbN$, 
\begin{equation}
    \begin{aligned}
        \pi_{\alpha,\lambda}(x)     
        &  \ge \sum_{k=t}^\infty \alpha(1-\alpha)^k\lambda P^k(x)\\
        &\ge (1-\alpha)^t\(\pi(x)
        - \|\lambda P^t - \pi\|_\tv\).
    \end{aligned}
\end{equation}
Taking $t=\log^3(n)$, by \cref{coro:mixing},
\ac{whp} $\|\lambda P^t - \pi\|_\tv\le n^{-1}$.
Thus, \ac{whp} for all $x\in [n]$
\begin{equation}\label{eq:lowerpi}
    n\pi_{\alpha,\lambda}(x)
    \ge \max\left\{\frac{\delta n^{-\varepsilon}}{2\Delta^+}\alpha d_x^-,(1-\alpha)^t (n\pi(x) - 1)
    \right\} =: u(x,\alpha).
\end{equation}

By the lower bound in \cref{coro:lower_bound_large_degs},
\ac{whp}  for all $x\in V(a)$, $a>0$, we have $n\pi(x)\ge \frac{d_x^-}{2\pl d\pr}$ where $\pl d\pr=m/n$.  If $\alpha \le n^{-\varepsilon}$ then $(1-\alpha)^t\ge 1/2$ for $n$ large enough.
Thus,  for all $\varepsilon>0$, $\max\left\{\alpha,(1-\alpha)^t\right\}\ge n^{-\varepsilon}$ for $n$ large enough. We obtain that \ac{whp}
\begin{equation}u(x,\alpha)\ge n^{-2\varepsilon}d_x^- \max\left\{\alpha,(1-\alpha)^t\right\} \ge n^{-3\varepsilon} d_x^-.\end{equation}
It follows that \ac{whp}
\begin{equation}
    n\psi_{\alpha,\lambda}(n^a,\infty)\ge  \sum_{x\in V(a+3\varepsilon)}
    \ind(u(x,\alpha)>n^a)
    \ge
    |V(a+3\varepsilon)|
    =n\phi(n^{a+3\varepsilon},\infty)
    .
\end{equation}
Hence, by the power-law assumption on the in-degree sequence, 
we conclude that for all $\varepsilon>0$,
\ac{whp} for all $a\in(0,1/\kappa)$,
$\psi_{\alpha,\lambda}(n^a,\infty)\ge n^{-a\kappa -(3\kappa+1)\varepsilon}$.

\subsection{Upper bound}
By Proposition 8 in~\cite{caputo2021} we have
\begin{equation}
    \max_{x\in [n]}\|P_{\alpha,\lambda}^t(x,\cdot)-\pi_{\alpha,\lambda}\|_{\tv}
    \le 2(1-\alpha)^t\max_{x\in [n]}\|P^t(x,\cdot)-\pi\|_{\tv}.
\end{equation}
Thus, by \cref{coro:mixing}, for $t=\log^3(n)$, \ac{whp}
\begin{equation}\label{eq:cormix}
    \max_{x\in[n]}\|P_{\alpha,\lambda}^t(x,\cdot)-\pi_{\alpha,\lambda} \|_\tv\le \ee^{-\log^{3/2}(n)}.
\end{equation}
Call  $\mu^{\alpha,\lambda}_t$ the probability measure
\begin{equation}\mu^{\alpha,\lambda}_t(y)=\frac1n\sum_{x\in[n]}P_{\alpha,\lambda}^t(x,y).\end{equation}
From \cref{eq:cormix} we have, \ac{whp}, $|\mu^{\alpha,\lambda}_t(y)-\pi_{\alpha,\lambda}(y)|\le n^{-1}$ for all $y\in[n]$.
Let us also introduce
\begin{equation}\label{def:Za1}
Z^{\alpha,\lambda}_a
=
\sum_{y\in[n]}\ind(\mu^{\alpha,\lambda}_t(y)>\tfrac12\,n^{a-1}).
\end{equation}
Then, \ac{whp} 
\begin{equation}\label{DDCB}
    \psi_{\alpha,\lambda}(n^a,\infty) 
    =
    \sum_{y\in[n]}\ind(\pi_{\alpha,\lambda}(y)>n^{a-1})
    \le
    Z^{\alpha,\lambda}_a.
\end{equation}
Thus, the upper bound in \cref{lbtailPR} follows from \cref{claim:psiaoz1} in the following lemma.

\begin{lemma}\label{lem:zzlemma}
    Fix an arbitrary sequence $\alpha=\alpha_n\in[0,1]$,
    and an arbitrary sequence of probability measures $\lambda=\lambda_n$ on $[n]$.
    Then, under \cref{cond:main},
    \begin{equation}\label{def:Zaz}
        \max_{x\in[n]}\pi_{\alpha,\lambda}(x)\le \frac{30\log(n)\Delta_*}n\,, 
    \end{equation}
    where $\Delta_*=\Delta^- + n\lambda_{\max}$ and $\lambda_{\max}=\max_{z\in[n]}\lambda(x)$.
    Moreover, if the empirical in-degree distribution has
    power-law behavior with index $\kappa>2$ and $\lambda$ satisfies~\cref{hp:lambda}, 
    then for any $a\in (0,1/\kappa)$
    \begin{equation}\label{claim:psiaoz1}
        \bbP\(Z^{\alpha,\lambda}_a > n^{1-a\kappa +\varepsilon}\)=o(1).
    \end{equation}
\end{lemma}
\begin{proof}
We introduce the annealed construction for PageRank walks with uniform starting vertices. Adapting the discussion in \cref{eq:def-annealing}, we write, for all $t,K\in\bbN$,
\begin{equation}\label{eq:def-annealingPR}
\E\[\big(\mu^{\alpha,\lambda}_t(A)\big)^K\]=\E\[\big(\PP_{\alpha,\lambda}(X_t\in A)\big)^K \]=\P^{{\rm an},K}_{\alpha,\lambda}\(X_t^{(k)}\in A,\:\forall k\in [K]\),
\end{equation}
where $A\subset[n]$,
\begin{equation}
\PP_{\alpha,\lambda}(X_t\in A)=\frac1n\sum_{x\in [n]}\sum_{y\in A}P_{\alpha,\lambda}^t(x,y),
\end{equation}
and $\P^{{\rm an},K}_{\alpha,\lambda}$ is the law of the non-Markovian process
\begin{equation}
    \left\{X_s^{(k)}\,,\;s\in\{0,\dots,t\},\;k\in\{1,\dots, K\}\right\}\,,
\end{equation} 
which can be described as follows. Start with an empty matching. For each $k\in[K]$, given the first $k-1$ walks $(X_s^{(\ell)})_{s\le t,\,\ell\le k-1}$, to generate the $k$-th walk,
\begin{enumerate}[(i)]
    \item
    start the $k$-th walk, i.e., $X^{(k)}$, at a uniformly random vertex $X_{0}^{(k)}\in[n]$;
    \item for all $s\in\{0,\dots,t-1\}$: select one of the tails of  $X_{s}^{(k)}$ uniformly at random, call it $e$, and draw an independent Bernoulli($\alpha$) random variable $U$,
    \begin{itemize}
        \item If $U=0$ and $e$ was already matched by one of the previous walks, or by $X^{(k)}$ itself at a previous step, to some head $f$, then let $X_{s+1}^{(k)}=v_f$;
        \item If $U=0$ and $e$ is still unmatched, then  select a uniformly random head, $f$, among the unmatched ones, match it to $e$, and let $X_{s+1}^{(k)}=v_f$;
        \item If $U=1$, then select a random vertex $Y \eql \lambda$ and set $X_{s+1}^{(k)}=Y$.
    \end{itemize}
\end{enumerate}

To prove the upper bound on $\pi_{\alpha,\lambda}$ \cref{def:Zaz},
we are going to show that, under \cref{cond:main},
if $K=\Theta(\log n)$, then
\begin{equation}\label{def:Zazz}
\E\[ \ind_{\cG}\(\mu^{\alpha,\lambda}_t(y)\)^K\]\le \left(\frac{10 K \Delta_*}{n} \right)^K,
\end{equation}
Note that this is the statement of \cref{lemmaM} with $\mu_t$ replaced by $\mu^{\alpha,\lambda}_t$ and $\Delta^-$ replaced by $\Delta_*$.
Once this bound is established, then the same argument in \cref{617,eq:finn} yields the estimate \cref{def:Zaz}.

Going over the proof of \cref{lemmaM} step by step,
we see that up to \cref{eq:boundan2} nothing is changed in the argument,
while \cref{eq:boundan2} continues to hold provided we replace $d_z^-/n$ with $d_z^-/n +
\lambda(z)$.  This is obtained by considering the last time $s$ such that the $\ell$-th walk enters
the set $D$ of previously activated edges $s$ and,
from then on, stays on $D$ without undergoing any teleportation.
The vertex $z\in D$ where this last entry occurs can be reached either by activating a fresh edge,
which contributes at most $d_z^-/n$ or by a teleportation which has probability $\lambda(z)$.
Once this modification is made, all arguments can be repeated without any change,
and \cref{def:Zazz} follows.

The same reasoning shows that the proof of \cref{lem:zlemma} goes through with the only change that
$d_z^-/n$ must be replaced by $d_z^-/n + \lambda(z)$ in \cref{eq:boundan}.
Using also the assumption
$\lambda_{\max}\le n^{\varepsilon-1}$ for all $\varepsilon>0$, this implies \cref{claim:psiaoz1}.
\end{proof}

To prove the upper bound in \cref{lbtailPR}, we can use \cref{eq:cormix} and \cref{lem:zzlemma} and the claim follows exactly as in \cref{sec:proofofub}.

\begin{remark}\label{rem:remalpha}
Note that the walk in the above proof differs from previously introduced annealed
walks only in that it is now possible to teleport, which corresponds to the event $U=1$.
Whenever this event occurs,
the walk does not activate any new matching and therefore the environment is left unchanged.
Since the main challenge in the analysis of the annealed walk is represented by the presence of the previously activated matching,
this feature makes the case of PageRank surfers actually simpler than the case with $\alpha=0$.
This also explains why the upper bound on $\psi_{\alpha,\lambda}(n^a,\infty)$ in \cref{lbtailPR}
holds \emph{uniformly} in the choice of $\alpha=\alpha_n\in[0,1]$.
Moreover, for the degree sequence it is sufficient to assume $\kappa$-lightness.
\end{remark}

\begin{remark}\label{rem:pPR}
We observe that the estimates above immediately imply  the bounds on the maximum PageRank score in \cref{rem:maxPR}. The lower bound is a consequence of
\cref{eq:lowerboPR}. 
The upper bound has been established in \cref{lem:zzlemma}.
\end{remark}

\section{Tightness of estimates on \texorpdfstring{$\pi_{\max}$}{pi(max)}}\label{sec:examples}

In this section we discuss the tightness of the bounds in~\cref{th:main1},
providing examples that show that both bounds in \cref{eq:upb} and \cref{eq:lowb}
cannot be substantially improved in general.

We first focus on~\cref{eq:upb}.
Fix $\varepsilon>0$.
Let $\Delta=\lceil \ee^{1/\varepsilon}\rceil$ and $\delta = 2$.
Consider a degree sequence $\bfd_{n}$ with half of the $n$ vertices having degrees $(\Delta,\delta)$,
and the other half of degrees $(\delta, \Delta)$.
Note that
\begin{equation}
\frac{\log \delta}{\log \Delta}
\le
\varepsilon \log 2
\le
\varepsilon
.
\end{equation}
Then, by Theorem~1.6 in~\cite{caputo2020a},
there exists a constant $c=c(\varepsilon)>0$ such that \ac{whp}
\begin{equation}\label{eq:LB_example}
\pi_{\max} \ge \frac{c \log^{1-\varepsilon}n}{n},
\end{equation}
proving that~\cref{eq:upb} is tight up to a sub-logarithmic multiplicative factor.

The rest of the section is devoted to provide
a wide class of examples where~\cref{eq:lowb} is tight.
We call a bi-degree sequence $\bfd_{n}$ \emph{extremal} if
there exists $w\in [n]$ such that $d_{w}^{-} = \Delta^{-}$ and for any $z\neq w$,
\begin{equation}\label{eq:extremal_degs}
 d^-_z=o\( \frac{d^-_w}{\log n}\).
\end{equation}

The following can be seen as a refinement of~\cref{eq:upb} for extremal sequences.
\begin{proposition}\label{lem:example}
    Let $\bfd_{n}$  be an extremal bi-degree sequence satisfying~\cref{cond:main}.
    Then, for any $\varepsilon>0$ whp
    \begin{equation}
        \abs*{\frac{\pi_{\max}}{\Delta^{-}/m}-1} \le \varepsilon.
    \end{equation}
    Moreover, the maximum stationary value is attained uniquely at the vertex of maximum in-degree.
\end{proposition}
Note that any non-trivial extremal sequence satisfies $\Delta^- \log^{-1}(n)\to \infty$ as $n\to \infty$. Requiring such condition is natural in view of the existence of sequences whose $\pi_{\max}$ exhibits $\log^{1-o(1)}(n)$ deviations with respect to $\Delta^-/m$ (see \cref{eq:LB_example}).

The lower bound follows from \cref{eq:lowb}.
The idea for the upper bound is to mimic the proof of the upper bound \cref{eq:upb}
while removing the factor $C\log{n}$ from it.
Let $w$ be the vertex attaining the maximum in-degree.
Let $\eta$ be the constant appearing in~\cref{eq:2+eta}.
Recall the definitions of $h_{\varepsilon}$ and $\cG^{+}(h)$ in \cref{def-hslash} and \cref{eq:defG}.
Choose $\varepsilon\in (0,\eta/6)$,
write $h=h_\varepsilon$ and $\cG=\cG^+(h)$,
and define $\cH \coloneqq \{\tx(\cB^-_w(h))=\tx(\cB^+_w(h))=0\}$.

For $t\in \N$, recall the definition of the measure $\mu_{t}$ on $[n]$ given in~\cref{eq:def-lambda2}. Following the argument in the proof of the upper bound in \cref{th:main1}, it suffices to prove the following strengthening of~\cref{lemmaM}.
\begin{lemma}\label{lemmaM:1}
    Let $\bfd_{n}$ be as in~\cref{lem:example}. For any $\gamma,C>0$, $t=\log^{3}(n)$ and $K=C\log(n)$, one has
    \begin{equation}\label{eq:lemmaM:1}
        \E[\ind_{\cG \cap \cH}(\mu_t(y))^K]
        \le
        \left(\(\frac{1}{2}+\frac{\ind({y=w})}{2}+\gamma\) \frac{\Delta^-}{m} \right)^K,
    \end{equation}
    for all $y\in [n]$ and all sufficiently large $n$.
\end{lemma}
\begin{proof}
\cref{le:coupling} ensures that the coupling between $\cB^-_w(h)$ and $\cT^-_w(h)$ succeeds whp, thus $\P(\tx(\cB^-_w(h))=0)=1-o(1)$. A similar argument as the one in~\cref{lemmaG} but only for the out-neighborhood of $w$ implies that $\P(\tx(\cB^+_w(h))=0)=1-o(1)$. Combining it with~\cref{lemmaG} to bound the probability of $\cG$, we have $\P(\cG\cap \cH)=1-o(1)$.

The proof is very similar to that of \cref{lemmaM}.
We reuse the notation defined there and omit the identical details.
As in~\cref{cond11}, we have
\begin{equation}\label{eq:H^c:1}
    \P^{\mathrm{an}}\(H_K^c \)= o\left((\Delta^-/m)^K\right),
\end{equation}
Having defined $B_{\ell}, F_{\ell}, H_{\ell}$ it suffices to bound the terms in~\cref{cond111}.
As in the proof of~\cref{lemmaM},
if the event $B_{\ell}$ holds then let $s$ be the time that the walk enters $D_{\ell,s}$
and traverses only edges in $D_{\ell,s}$ until reaching $y$.
Write  $I_1=\{0,\dots,t-h\}$ and $I_2=[t]\setminus I_1$.
As in~\cref{eq:qzzz}, the contribution of $s\in I_1$ is $o(\Delta^-/n)$.
To bound the contribution of $s\in I_2$,
we split the left-hand-side of \cref{eq:sumzvd} into two parts depending on whether we enter at $w$ or not:
\begin{equation}\label{eq:qzzzzz}
    \begin{aligned}
        \sum_{a=0}^{h-1}&\max_{D\in\cA_\ell(t-a)} \sum_{z\in V(D)}\frac{d^-_z}{m-Kt}q_D(z,y,a) \\
        &\le \sum_{a=0}^{h-1}\max_{D\in\cA_\ell(t-a)}\frac{d^-_w}{m-Kt} q_D(w,y,a)+ \sum_{a=0}^{h-1}\max_{D\in\cA_\ell(t-a)} \sum_{z\in V(D)\setminus\{w\}}\frac{d^-_z}{n} q_D(z,y,a).
    \end{aligned}
\end{equation}

Let us bound the first contribution in~\cref{eq:qzzzzz}.
If $y=w$, since $D\sim \cH$,
there is no path from $w$ to $w$ of length at least $1$
and at most $h$ and the term is bounded by the contribution of $a=0$, that is $\frac{\Delta^-}{m-tK}=(1+o(1))\frac{\Delta^-}{m}$.
If $y\neq w$, since $D\sim \cH$,
there is at most one path of length at most $h$ from $w$ to $y$.
As the minimum out-degree is at least $2$ by \cref{cond:main}, we have $\sum_{a=0}^{h-1} q_D(w,y,a)\le 1/2$ and the term is bounded by $(1+o(1))\frac{\Delta^-}{2m}$.
Therefore, the first contribution in~\cref{eq:qzzzzz} is bounded
by $(\frac{1}{2}+\frac{\ind(y=w)}{2}+o(1))\frac{\Delta^-}{m}$.

For the second contribution in~\cref{eq:qzzzzz}, the same bound as in~\cref{cond4} gives a total of
\begin{equation}
\sum_{a=0}^{h-1}\frac{(2Ka+1)2^{-a+1}\max_{z\neq w} d_z^-}{n}= O\(\frac{K\max_{z\neq w} d_z^-}{n}\)=o\left(\frac{\Delta^-}{n}\right),
\end{equation}
as $Kd^-_z=o(\Delta^-)$ for all $z\neq w$ by \cref{eq:extremal_degs}.

Putting all the contributions together, we have that~\cref{eq:lemmaM:1} holds for all $y\in [n]$, and the lemma follows.
\end{proof}

\section{Future research directions}\label{sec:future}

A number of open problems arise from empirical observations. Several papers have identified a consistent disagreement between the largest in-degree nodes and the ones attaining the maximum PageRank score in real-world networks (see, e.g.,~\cite{chen2007,volkovich2009}). Outliers in each ranking exhibit correlation but the top sets tend to disagree. This reinforces the idea that  rankings based on stationary values are  much more than the simple in-degree ranking and poses the question of determining under which conditions the top in-degree and top   score nodes coincide. Notably,~\cref{eq:lowb} tells us that the maximum  stationary value is \emph{never} asymptotically smaller than the maximum in-degree divided by
$m$ and these two asymptotically coincide for sequences with an outstanding maximum in-degree vertex (see~\cref{lem:example}).

In contrast, in~\cref{fig:sim} we display the results of a simulation done for the \ac{dcm} with power-law in-degree distribution and constant out-degree, which suggests that asymptotically the two may only differ by a non-trivial  multiplicative factor~\cite{cai2021}.
\begin{figure}
    \begin{center}
        \begin{subfigure}[b]{\linewidth}
            \centering
            \includegraphics[width=0.6\textwidth]{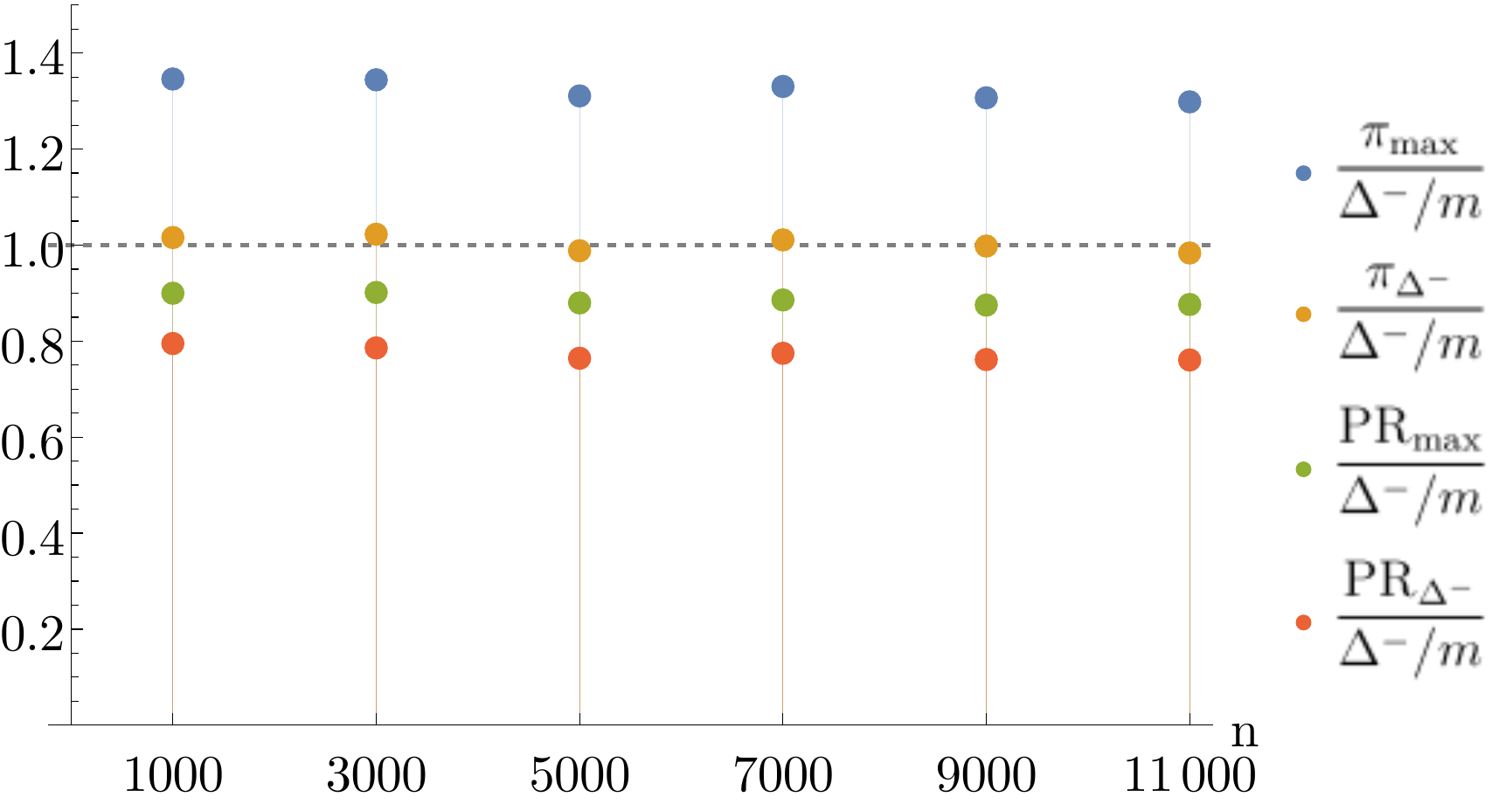}
            \caption{Extremal stationary values on average}
        \end{subfigure}
        \begin{subfigure}[b]{\linewidth}
            \centering
            \includegraphics[width=0.6\textwidth]{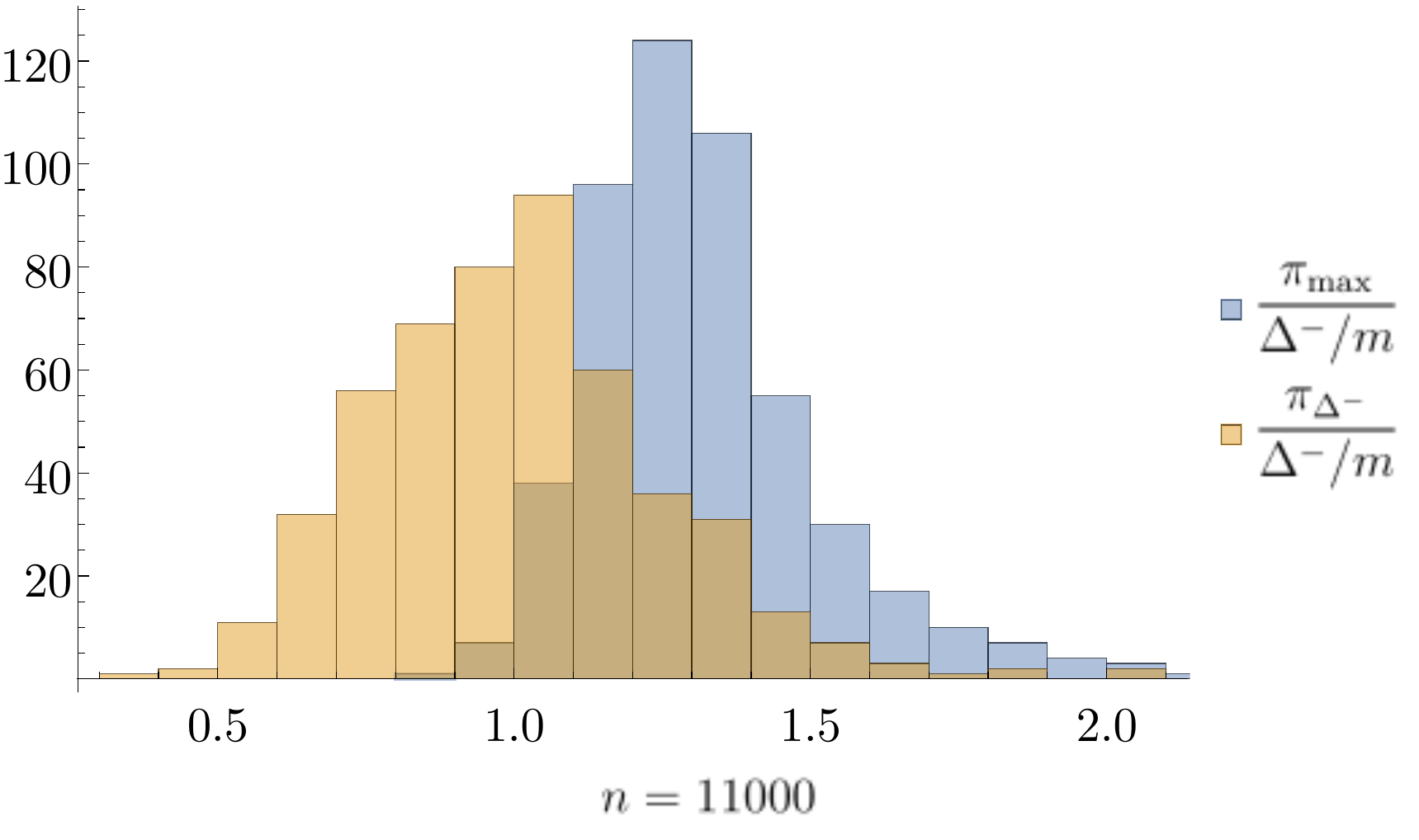}
            \caption{Histogram of stationary values}
        \end{subfigure}
    \end{center}
\caption[Simulation in \ac{dcm} with power-law in-degree distribution]{Simulation in \ac{dcm} with power-law in-degree distribution
    with index $\kappa = \frac{5}{2}$ and out-degree $2$.
    $\pi_{\max}$ and $\pi_{\Delta^{-}}$ denote the maximum stationary value and the stationary value of
    the node with maximum in-degree.
    $\mathrm{PR}_{\max}$ and $\mathrm{PR}_{\Delta^{-}}$ denote the maximum PageRank and the PageRank of
    the node with maximum in-degree, for teleporting probability $\alpha = \frac{1}{4}$ and the
    uniform teleporting distribution $\lambda$.
    We generated $500$ samples for each $n$.}\label{fig:sim}
\end{figure}
It would be interesting to determine under which conditions on the degree sequence, the largest in-degree and the largest stationary (or PageRank) value coincide in order, or asymptotically.

A possible extension of \cref{th:main2} is to study the upper tail of $\psi$ in \cref{eq:def-emperical} for degree sequences satisfying \cref{cond:main} but not necessarily having power-law in-degrees. In such case, \cref{rem:general_ubtail} gives an upper bound, which can be possibly refined if additional information about the upper tail of the empirical in-degree distribution $\phi$ is known. Our results in \cref{sec:PL} suggest that $\psi(n^a,\infty)$ could be approximated by the order of the $a$-skeleton as defined in~\cref{suse:skeleton}.

An important and challenging open problem is the extension of our results to the case of in-degrees with bounded first moment, that is replacing the $2+\eta$ in condition (\labelcref{it3}) of \cref{cond:main} by $1+\eta$, or even $1$. The structure and distances in random graphs with infinite variance degrees is strikingly different~\cite{vanderhofstad2007} from the ones satisfying~\cref{cond:main}. It would be interesting to determine whether the vertices of large in-degree will have a non-negligible effect, speeding-up the mixing time. This case is central in applications, as many real-world networks are believed to have power-law behavior with index $\kappa\in (1,2)$~\cite{pandurangan2002}.

Another interesting open problem concerns the relaxation of condition (\labelcref{it1}) of \cref{cond:main}.
Minimum out-degree at least $2$ is required
to ensure that the random walk has no trivial attractive strongly connected components,
and in particular avoids the existence of dangling nodes (i.e., nodes of out-degree $0$). While this is a necessary requirement for the random walk without teleporting, it is interesting to study the PageRank surfer walk under the presence of dangling nodes.
Condition (\labelcref{it2}) is mainly technical,
facilitating the exploration of out-neighborhoods
and the existence of a law of large numbers (\cref{LLN}).
It would be interesting to obtain a version of~\cref{lbtailPR}
that allowed dangling nodes and arbitrarily large out-degrees satisfying a suitable moment assumption.
Research on related stochastic models suggests that
the effect of the out-degree distribution
and of dangling nodes is essentially negligible~\cite{volkovich2007}.

\bibliographystyle{abbrvnat}
\bibliography{max}

\end{document}